\documentclass[a4paper,reqno]{amsart}
\usepackage{amscd, amssymb, pb-diagram, amsmath}
\usepackage{amsthm}
\usepackage[bookmarks=true]{hyperref} 
\usepackage{slashed}

\usepackage{tikz}

\usepackage{tikz-cd}
\usepackage[margin=3.7cm]{geometry}
\usepackage{color}
\usepackage{tensor}
\usetikzlibrary{arrows,matrix,backgrounds,decorations.markings}
\numberwithin{equation}{section}

\usepackage[all]{xy}

\DeclareFontFamily{OT1}{pzc}{}
\DeclareFontShape{OT1}{pzc}{m}{it}{<-> s * [1.2] pzcmi7t}{}
\DeclareMathAlphabet{\mathpzc}{OT1}{pzc}{m}{it}

\RequirePackage[textsize=tiny,color=green!20!white]{todonotes}

\def\Aut{\operatorname{Aut}}
\def\End{\operatorname{End}}
\def\Hom{\operatorname{Hom}}

\def\ker{\operatorname{Ker}}

\def\dim{\operatorname{dim}}

\def\rd{\operatorname{d}\!}

\def\C{\mathbb{C}}

\def\R{\mathbb{R}}
\def\N{\mathbb{N}}
\def\Z{\mathbb{Z}}

\newtheorem{thm*}{Theorem}
\newtheorem{thm}{Theorem}[section]

\newtheorem{cor}[thm]{Corollary}

\newtheorem{lemma}[thm]{Lemma}
\newtheorem{prop}[thm]{Proposition}

\theoremstyle{definition}
\newtheorem{definition}[thm]{Definition}

\newtheorem{ass}{Assumption}

\theoremstyle{remark}
\newtheorem{remark}[thm]{Remark}

\newtheorem{remark*}[thm*]{Remark}

\newtheorem{example}[thm]{Example}

\begin{document}

\date{\today}

\title[The index problem for BGG-sequences]{Solving the index problem for (curved) Bernstein-Gelfand-Gelfand sequences}

\author{Magnus Goffeng}

\address[M. Goffeng]{Department of Mathematics, Lund University, Sweden}
\email{magnus.goffeng@math.lth.se}

\maketitle

\begin{abstract}
We study the index theory of curved Bernstein-Gelfand-Gelfand (BGG) sequences in parabolic geometry and their role in $K$-homology and noncommutative geometry. The BGG-sequences fit into $K$-homology, and we solve their index problem. We provide a condition for when the BGG-complex on the flat parabolic geometry $G/P$ of a semisimple Lie group $G$ fits into $G$-equivariant $K$-homology by means of Heisenberg calculus. For higher rank Lie groups, we prove a no-go theorem showing that the approach fails. 
\end{abstract}

\section{Introduction}

In this paper we study curved Bernstein-Gelfand-Gelfand sequences in noncommutative geometry and index theory with the purpose of forming a global perspective on parabolic geometries. Parabolic geometry \cite{capslovak} can be seen as a vast generalization of Riemannian geometry, describing a plethora of geometric structures such as contact or CR structures, hyper-Kähler structures and the geometry of flag manifolds $\mathsf{G}/\mathsf{P}$, for examples see \cite[Chapter 4]{capslovak} and \cite{Cartan2,Cartan1,Cartan4}. In noncommutative geometry \cite{connesbook}, parabolic geometries arising from CR-structures and quaternionic analogues thereof have been of interest in the study of the Baum-Connes conjecture \cite{julgsun1,julgsp}. In representation theory function spaces on flat parabolic geometries play an important role in parabolic induction and Langlands classification of admissible representations \cite{knapp, wallach}, for instance through the Casselman-Wallach theorem.

 Parabolic geometries arise as a special case of Cartan geometries. Its name is derived from Fefferman's work \cite{feffpara} on the relation between parabolic invariant theory and local invariants on the boundary of pseudoconvex domains in complex manifolds, see \cite{fea}. In the spirit of geometric analysis and noncommutative geometry, a natural starting point to study the global aspects of a geometry is to study its geometric differential operators, e.g. the role that Laplace operators, the de Rham or Dirac operators play in Riemannian geometry. As pioneered by Čap-Slovák-Souček \cite{morecap}, building on work by Bernstein-Gelfand-Gelfand (BGG) \cite{bggoriginal} and Lepowsky \cite{lepog}, the natural geometric differential operators on a parabolic geometry arise from curved BGG-sequences. We study such geometric operators in this paper standing on the shoulders of the recently developed Heisenberg calculus on Carnot manifolds \cite{androerp,Dave_Haller1, ewertjncg, goffkuz,vanerpyuncken}.

The idea to describe parabolic geometry by means of curved BGG-sequences is well studied in the literature, for instance in \cite{GenericDave,Dave_Haller1,Dave_Haller2,Hallerzeta}. The special case of Rumin complexes arising on contact manifolds, as introduced by Rumin \cite{rumincomp}, are studied in more detail \cite{julgsun1}. BGG-sequences have been used for applications in the Baum-Connes conjecture \cite{bch} in operator $K$-theory, most notably by Yuncken \cite{yunckensl3} for $SL(3,\C)$, by Julg-Kasparov \cite{julgsun1} for $SU(n,1)$ and by Julg as an approach towards $Sp(n,1)$ \cite{julgsp}. An important feature needed in the applications to the Baum-Connes conjecture is equivariance, for noncompact groups this causes substantial analytic subtleties. Furthermore, BGG-sequences hold out hope for extending the ideas of noncommutative geometry to quantum groups through Heckenberger-Kolb's \cite{hkbggq} quantized BGG-sequence, first described by Rosso \cite{rossoslq}. Such ideas saw applications in work of Voigt-Yuncken \cite{voigtyunck} to the equivariant index theory of $SU_q(3)$ and its Drinfeld double, notably proving the Baum-Connes conjecture with trivial coefficients for $\widehat{SU_q(3)}$.

 A curved BGG-sequence often arises from compressing a de Rham or Dolbeault type complex to a subquotient complex, leaving the Fredholm index unaltered. The analytic machinery for studying curved BGG-sequences was set up by Dave-Haller \cite{Dave_Haller1} and refined in \cite{goffkuz} to show that curved BGG-sequences define elements in $K$-homology. This line of thought opens up two natural questions: 
\begin{enumerate}
\item The Fredholm index is unaltered by BGG-compression. Does the same statement hold for the associated $K$-homology classes?
\item A curved BGG-sequence captures the differential geometric features of a parabolic geometry, does it fit into Connes' program for spectral noncommutative geometry?
\end{enumerate} 
We study and solve the first problem in this paper with a view towards the equivariant case as needed for application to the Baum-Connes conjecture. This solves the index problem for curved BGG-sequences in the Baum-Douglas sense of describing the associated $K$-homology classes geometrically. Further clarifying the construction of $K$-homology classes from \cite{goffkuz} will indeed provide us with a better grasp on equivariance properties. The solution to the index problem for curved BGG-sequences can be seen as the first step towards incorporating methods of noncommutative local index theorems \cite{connesmosc} and its associated spectral geometric invariants into the study of parabolic geometry, and would be of relevance for better understanding a potential solution to the second listed problem above. The problem can be set up also in the larger generality of graded Rockland sequences, and in the equivariant setting there is an additional abstract existence assumption that, in this paper, is only guaranteed in real rank one or less.     

\subsection*{Acknowledgements} We thank the referee for their careful reading of the paper and their many helpful comments. The author was supported by the Swedish Research Council Grant VR 2018-0350. We wish to thank Magnus Fries, Ada Masters, Ryszard Nest, Christian Voigt and Robert Yuncken for stimulating conversations and encouragement at various stages.

\section{Main results}

The geometric objects we study in this paper are Carnot manifolds $\mathfrak{X}$. The precise definition is reviewed below in Definition \ref{carnotdef} and further context can be found in Section \ref{subsec:carnottnlknad}. For now, the reader unacquainted with this notion should think of $\mathfrak{X}$ carrying a filtration of the tangent bundle by sub-bundles such that the Lie bracket of vector fields induces a fibrewise structure of a nilpotent Lie algebra on the associated graded bundle $\mathfrak{t}_H\mathfrak{X}$ integrating to a bundle $T_H \mathfrak{X}\to  \mathfrak{X}$ of osculating Lie groups in each fibre. We also consider group actions on a Carnot manifold $\mathfrak{X}$, i.e. groups acting by diffeomorphisms preserving the filtration, which therefore lift to fibrewise Lie algebra isomorphisms on $\mathfrak{t}_H\mathfrak{X}$. The bundle of osculating Lie groups $T_H\mathfrak{X}\to \mathfrak{X}$ and the affiliated Heisenberg calculus will play a prominent role in this work, for more details see Section \ref{subsec:heis}.

The main results of the paper concern curved BGG-complexes, and to some extent we work in the generality of graded Rockland sequences. To describe the main results of the paper, we will provide a colloquial description of a low-dimensional example showcasing the overall structure and difficulties, that we connect to in parallel to the general results in this section. A more precise overview can be found below in Section \ref{subsec:bggrock}, more precisely Example \ref{sl3c}, or in the literature \cite{Dave_Haller1,goffkuz,julgsun1,yunckensl3}. Take as base field $F=\R$ or $F=\C$ and write $\mathsf{H}_3(F)$ for the three-dimensional Heisenberg group. That is, the Lie algebra $\mathfrak{h}_3(F)$ of $\mathsf{H}_3(F)$ is spanned over $F$ by $X$, $Y$, and $Z$ with Lie bracket defined from the commutation relation $[X,Y]=Z$. If $\mathfrak{X}$ is a compact Carnot manifold (assumed to be complex if $F=\C$) whose osculating groups are isomorphic to $\mathsf{H}_3(F)$, we can construct an associated curved BGG-complex which is locally of the form
\small
\begin{equation}
\label{bgg13d}
\begin{tikzcd}
&&C^\infty(\mathfrak{X};L_{\alpha_{\mathsf{x}}})\ar{rr}{-D_{XY+Z}}\ar{ddrr}{\qquad-D_{Y^2}}&&C^\infty(\mathfrak{X};L_{2\alpha_{\mathsf{x}}\!+\!\alpha_{\mathsf{y}}})\arrow{dr}{-D_Y}&& \\
0\ar[r]&C^\infty(\mathfrak{X})\ar{ur}{D_X}\ar{dr}[swap]{D_Y} &&&&C^\infty(\mathfrak{X};L_{2\alpha_{\mathsf{x}}+2\alpha_{\mathsf{y}}}) \arrow{r}&0\\
&&C^\infty(\mathfrak{X};L_{\alpha_{\mathsf{y}}}) \ar[rr,"D_{YX-Z}"]\ar{uurr}[swap]{\!\!\!\!\!\!\!\!\!\!\!\!\!\!\!\!\!\!\!\!\!\!\!\!\!\!\!\!\!\!\!\!\!\!\!\!\!\!\!\!\!\!\!\!D_{X^2}}&&C^\infty(\mathfrak{X};L_{\alpha_{\mathsf{x}}\!+\!2\alpha_{\mathsf{y}}})\ar{ur}[swap]{D_X}&&
\end{tikzcd}
\end{equation}
\normalsize
The operators in the complex \eqref{bgg13d} can in each point on $\mathfrak{X}$ up to lower order terms be written as the indicated element of the universal enveloping algebra $\mathcal{U}(\mathfrak{h}_3(F))$, e.g. $D_{XY+Z}$ is $XY+Z$ up to lower order terms. Here the complex line bundles $L_\alpha\to \mathfrak{X}$ are summands in a subquotient of the exterior bundle $\wedge^*T^*\mathfrak{X}$ derived from taking fibrewise Lie algebra cohomology of $\mathfrak{t}_H\mathfrak{X}$ and a certain reduction of the structure group (following Theorem \ref{knlkanlknad} below). To make this example more precise, the reader can keep in mind the following two cases: firstly the full flag manifold $\mathfrak{X}=SL(3,F)/B(F)$, where $B(F)$ is the Borel subgroup of upper triangular matrices in $SL(3,F)$, and secondly the nilmanifold $\mathfrak{X}=H_3(F)/\Gamma$ where $\Gamma$ is a lattice. For $\mathfrak{X}=SL(3,F)/B(F)$, the inclusion $\mathsf{H}_3(F)\subseteq SL(3,F)$ as the unipotent lower triangular matrices produces an embedding of $\mathsf{H}_3(F)$ as a Zariski open subset of $\mathfrak{X}$. The nilmanifold  $\mathfrak{X}=H_3(F)/\Gamma$ is locally diffeomorphic as a Carnot manifold to $H_3(F)$. As such, in both cases the operators in the complex \eqref{bgg13d} are well defined up to lower order terms. In both cases, $H^*(\mathfrak{h}_{3}(F))$ decomposes as indicated in the diagram and the subscripts in $L_\alpha$ indicate an appropriate character that describes $L_\alpha$ as a homogeneous vector bundle, for more details see Subsection \ref{susubbgg}.  We can also collect the vertical terms in the diagram above to a complex
\begin{equation}
\label{bgg23d}
0\to C^\infty\left(\mathfrak{X}\right)\xrightarrow{D_1}
C^\infty\left(\mathfrak{X};\begin{matrix}L_{\alpha_{\mathsf{x}}}\\\oplus \\L_{\alpha_{\mathsf{y}}}\end{matrix}\right)\xrightarrow{D_2}
C^\infty\left(\mathfrak{X};\begin{matrix}L_{2\alpha_{\mathsf{x}}+\alpha_{\mathsf{y}}}\\\oplus \\ L_{\alpha_{\mathsf{x}}+2\alpha_{\mathsf{y}}}\end{matrix}\right)\xrightarrow{D_3}
C^\infty\left(\mathfrak{X};L_{2\alpha_{\mathsf{x}}+2\alpha_{\mathsf{y}}}\right)\to 0.
\end{equation}
The cohomology of the complex \eqref{bgg23d} coincides with that of the de Rham complex if $F=\R$ and with that of the Dolbeault-de Rham complex if $F=\C$. The case $F=\C$ is less studied in the literature, except for $\mathfrak{X}=SL(3,\C)/B(\C)$ \cite{yunckenthesis}, but the techniques of \cite[Section 4]{Dave_Haller1} apply also to Lie algebras over $\C$, see Section \ref{subsec:bggrock} below, showing that  \eqref{bgg23d} is a quasi-isomorphic, quotient sub-complex of the Dolbeault-de Rham complex. In this particular case, the curved BGG-complex \eqref{bgg13d} is a Rockland complex. In general, we arrive at a sequence that is a complex only in leading term and is only Rockland in a graded sense, for precise definitions see Section \ref{subsec:bggrock}.

In Section \ref{seconrock} we study how curved BGG-complexes, or more generally graded Rockland sequences, fit into $K$-homology. Phrased in terms of Atiyah's initial idea for $K$-homology \cite{atiyahog}, we ask for an abstract elliptic operator on $\mathfrak{X}$ that encodes a graded Rockland sequence. The following result can be found in Subsection \ref{subsdsseconrock}.

\begin{thm*}
\label{firstmain}
Let $\mathfrak{X}$ be a compact Carnot manifold with a continuous filtered action by diffeomorphisms of a locally compact group $G$. Assume that $D_\bullet$ is a $G$-equivariant graded Rockland sequence admitting $G$-equivariant Rockland splittings $B_\bullet$ in the sense of Definition \ref{splittindefedf}. Then the operator $F_\bullet:=D_\bullet+B_\bullet$ defines a $K$-homology cycle for $C(\mathfrak{X})$ such that $g^{-1} F_\bullet g-F_\bullet$ is a compact operator depending continuously on $g\in G$.
\end{thm*}

The fact that graded Rockland sequences admit Heisenberg splittings (in the sense of Definition \ref{splittindefedf}) when dropping the group action can be found in Lemma \ref{splittinprop}. Let us return to our three dimensional example \eqref{bgg13d}-\eqref{bgg23d}. In this case as in most other cases in this paper, there is no obvious method to prove existence of $G$-equivariant Rockland splittings (in the sense of Definition \ref{splittindefedf}) beyond the Heisenberg calculus. If $\mathfrak{X}=SL(3,F)/B(F)$ the complex $D_\bullet$ is an $SL(3,F)$-equivariant graded Rockland sequence but as we shall see below, we do not have any $SL(3,F)$-equivariant Heisenberg splittings but $G$-equivariant Heisenberg splittings for $G\subseteq SL(3,F)$ of rank one or less. It is unclear to the author if there is a more relevant method for finding $G$-equivariant Rockland splittings, but the example of $SL(3,F)$ shows how a Heisenberg calculus for multifiltered manifolds in the sense of \cite{yunckenhab} holds potential to solve the problem. For $\mathfrak{X}=\mathsf{H}_3(F)/\Gamma$ the complex $D_\bullet$ is an $\mathsf{H}_3(F)$-equivariant graded Rockland sequence and by building a Heisenberg splitting locally, we see that there are equivariant Heisenberg splittings.

\begin{remark}
We note that Theorem \ref{firstmain} is stated in the context of a group action. In most interesting situations, the action of $G$ is not unitary (unless $G$ is compact). Therefore Theorem \ref{firstmain} does not produce a bona fide cycle for equivariant $KK$-theory. However, for an appropriate length function $\ell$, such as that in Proposition \ref{lengthsknon}, we can therefore from Theorem \ref{lknlknjkjaudn} conclude that $F_\bullet$ forms a well defined class associated with the equivariant graded Rockland sequence in a $KK$-group $KK^{G,\ell}_0(C(\mathfrak{X}),\C)$ defined using exponentially bounded group actions on Hilbert spaces.
\end{remark}

\begin{remark}
\label{lknlknadrem}
If $\mathfrak{X}$ is a Carnot manifold with a filtered action by diffeomorphisms of a compact group $G$, and $D_\bullet$ is a $G$-equivariant graded Rockland sequence we show that $D_\bullet$ even admits $G$-equivariant Heisenberg splittings $B_\bullet$ in the sense of Definition \ref{splittindefedf}. This is done via an algebraic trick (see Lemma \ref{splittinprop}) and averaging over the group. In particular, for a compact group $G$ a $G$-equivariant graded Rockland sequence canonically defines a $G$-equivariant $K$-homology cycle for $C(X)$ via the operator $F_\bullet$.
\end{remark}

The assumption in Remark \ref{lknlknadrem} that the group $G$ is compact can in some cases be relaxed. We do however note that there are rigidity results for the size of the automorphism groups of parabolic geometries, such as Ferrand-Obata's theorem on conformal actions, and Schoen's generalization \cite{schoencr} to $CR$-manifolds, as well as Bader-Frances-Melnick's work \cite{badframel} in higher rank that we recall in Theorem \ref{thmbaderfranc} below. Such rigidity results force automorphism groups of parabolic geometries to be lower rank than the geometry, and in rank one even to be compact \cite{francesjnkjnkjad} unless $\mathfrak{X}=\mathsf{G}/\mathsf{P}$ is the flat parabolic manifold. However, for many applications the flat parabolic manifold $\mathfrak{X}=\mathsf{G}/\mathsf{P}$ is of primary importance. For this situation we have the following no-go result for constructing equivariant splittings in the Heisenberg calculus found in Subsection \ref{nogosubsec}. For notational clarity, \emph{we write $\mathsf{G}$ for Lie groups used to define parabolic geometries and $G$ for groups we consider actions of}. In the case of flat parabolic geometries, $\mathsf{G}/\mathsf{P}$ we often have $G=\mathsf{G}$.

\begin{thm*}
Let $G=\mathsf{G}$ be a connected, semisimple Lie group and $\mathsf{P}\subseteq \mathsf{G}$ a parabolic subgroup. Write $\mathfrak{X}=\mathsf{G}/\mathsf{P}$ for the flat parabolic manifold of type $(\mathsf{G},\mathsf{P})$ with $\mathsf{G}$ acting on $\mathfrak{X}$ as filtered diffeomorphisms. We write $D_\bullet^{\mathsf{BGG}(W)}$ for the BGG-complex associated with a finite-dimensional $\mathsf{G}$-representation $W$. Then $D_\bullet$ admits $\mathsf{G}$-equivariant Heisenberg splittings $B_\bullet$ in the sense of Definition \ref{splittindefedf} if $\mathsf{G}$ has rank less than or equal to one but does not if the rank of $\mathsf{G}/\mathsf{P}$ is greater than one.
\end{thm*}

This result can be exemplified for our three-dimensional example \eqref{bgg13d}-\eqref{bgg23d} for $G=\mathsf{G}=SL(3,F)$ and $\mathfrak{X}=SL(3,F)/B(F)$ where the rank is $2$. Indeed, if we look for the first Heisenberg splitting among Heisenberg operators
$$B_1:C^\infty\left(\mathfrak{X};\begin{matrix}L_{\alpha_{\mathsf{x}}}\\\oplus \\L_{\alpha_{\mathsf{y}}}\end{matrix}\right)\to C^\infty(X),$$
then the incompatibility of the characters and the homogeneity of the Heisenberg calculus forces any such operator $B_1=(B_{11},B_{12})$ in the Heisenberg calculus to be a differential operator, for more details see Corollary \ref{jnljnjknkjnad}. But $B_1D_1=1$ has no solution $B_1$ among the space of differential operators. Further discussion on the higher rank situation can be found in \cite{yunckenhab}, where multifiltered geometries pose a potential solution to the issues at hand.

We have by Theorem \ref{firstmain}, and the accompanying Remark \ref{lknlknadrem}, that a graded Rockland sequence defines an equivariant $K$-cycle for a compact group action. We describe the associated $K$-homology class in the case of BGG-complexes on parabolic geometries in Section \ref{sec:aojdnajdnadojn}. The results of that section can be summarized in the following theorem.

\begin{thm*}
\label{mainthm3}
Let $\mathfrak{X}$ be a compact parabolic manifold with a filtered action by diffeomorphisms of a compact group $G$. Assume that $\mathsf{BGG}(\nabla_{\pmb{E}})$ is the $G$-equivariant curved BGG-sequence associated with a tractor bundle $\pmb{E}$ with tractor connection $\nabla_{\pmb{E}}$. Then the class of $\mathsf{BGG}(\nabla_{\pmb{E}})$ in $K_0^G(\mathfrak{X})$ is explicitly described by the Euler class $\mathrm{Euler}(\mathfrak{X})\in K_0^G(\mathfrak{X})$ as
$$[\mathsf{BGG}(\nabla_{\pmb{E}})]=\mathrm{Euler}(\mathfrak{X})\cap [\pmb{E}].$$
If $\mathfrak{X}$ is a complex parabolic manifold, and the $G$-action as well as the tractor bundle are holomorphic, then the class of $\mathsf{BGG}(\nabla_{\pmb{E}})$ in $K_0^G(\mathfrak{X})$ is explicitly described by the fundamental class $[\mathfrak{X}]\in K_0^G(\mathfrak{X})$ (associated with the complex structure) as
$$[\mathsf{BGG}(\nabla_{\pmb{E}})]=[\mathfrak{X}]\cap [\pmb{E}].$$
\end{thm*}

\begin{remark}
By a result of Rosenberg \cite{rosentrivial}, when $\mathfrak{X}$ is connected, the Euler class of $\mathfrak{X}$ in $K$-homology is ``trivial'' in the sense that $\mathrm{Euler}(\mathfrak{X})=\chi(\mathfrak{X})[\mathrm{pt}]$. Here $\chi(\mathfrak{X})\in \Z$ is the Euler characteristic and $[\mathrm{pt}]\in K_0(\mathfrak{X})$ the class defined from including a point into $\mathfrak{X}$. In particular, in the case of $G$ being trivial and $\mathfrak{X}$ a connected, real parabolic geometry
$$[\mathsf{BGG}(\nabla_{\pmb{E}})]=\mathrm{Euler}(\mathfrak{X})\cap [\pmb{E}]=\chi(\mathfrak{X})\mathrm{rk}(\pmb{E})[\mathrm{pt}].$$ 
\end{remark}

\section{Carnot manifolds}
\label{subsec:carnottnlknad}

The main source of geometric examples in this paper is that of parabolic geometry, that in turn fits into a larger scheme of Carnot manifolds. We recall the notion of a Carnot manifold and after that we specialize to parabolic manifolds.

\begin{definition}
\label{carnotdef}
A Carnot manifold of depth $r\in \N_+$ is a manifold $\mathfrak{X}$ equipped with a filtration 
$$T\mathfrak{X}=T^{-r}\mathfrak{X}\supsetneq T^{-r+1}\mathfrak{X}\supsetneq \ldots\supsetneq T^{-2}\mathfrak{X}\supsetneq T^{-1}\mathfrak{X}\supsetneq 0,$$
of sub-bundles such that $[T^{j}\mathfrak{X},T^{k}\mathfrak{X}]\subseteq T^{j+k}\mathfrak{X}$ for any $j,k$. For simplicity, we set $T^{j}\mathfrak{X}=T\mathfrak{X}$ for $j\leq -r$ and $T^{j}\mathfrak{X}=0$, $j\geq 0$.

If $\mathfrak{X}$ is a complex manifold, and each $T^{-j}\mathfrak{X}\subseteq T\mathfrak{X}$ is a holomorphic subbundle, we say that $\mathfrak{X}$ is a complex Carnot manifold.
\end{definition}

Associated with a Carnot structure, we have a fibrewise Lie algebra structure on the graded bundle 
$$\mathrm{gr}(T\mathfrak{X}):=\oplus_j T^{-j}\mathfrak{X}/T^{-j+1}\mathfrak{X}.$$
More precisely, the Lie bracket on vector fields induces a bundle map 
$$(T^{-j}\mathfrak{X}/T^{-j+1}\mathfrak{X})\times (T^{-k}\mathfrak{X}/T^{-k+1}\mathfrak{X})\to T^{-j-k}\mathfrak{X}/T^{-j-k+1}\mathfrak{X},$$
for any $j,k$ which in turn induces a graded Lie bracket on $\mathrm{gr}(T\mathfrak{X})$. We call this Lie bracket on $\mathrm{gr}(T\mathfrak{X})$ the Levi bracket induced from the Carnot structure. When equipped with the anchor mapping $0:\mathrm{gr}(T\mathfrak{X})\to T\mathfrak{X}$, $\mathrm{gr}(T\mathfrak{X})$ forms a Lie algebroid that we denote by $\mathfrak{t}_H\mathfrak{X}$. Since the Lie bracket is nilpotent, $\mathfrak{t}_H\mathfrak{X}$ integrates to a Lie groupoid $T_H\mathfrak{X}$ with unit space $\mathfrak{X}$ with source and range map coinciding. The source/range mapping $T_H\mathfrak{X}\to \mathfrak{X}$ defines a fibre bundle; each fibre is a simply connected nilpotent Lie group. 

We call $\mathfrak{t}_H\mathfrak{X}$ the osculating Lie algebroid of $\mathfrak{X}$ and $T_H\mathfrak{X}$ the osculating Lie groupoid of $\mathfrak{X}$. For a point $x\in \mathfrak{X}$, we call $\mathfrak{t}_H\mathfrak{X}_x$ the osculating Lie algebroid in $x$ and $T_H\mathfrak{X}_x$ the osculating Lie groupoid in $x$. Let $\mathsf{N}$ be a graded, simply connected, nilpotent Lie group with Lie algebra $\mathfrak{n}$. We say that $\mathfrak{X}$ is a regular Carnot manifold of type $\mathsf{N}$ (or $\mathfrak{n}$) if $T_H\mathfrak{X}\to \mathfrak{X}$ is a locally trivial bundle of graded, nilpotent Lie groups with fibre $\mathsf{N}$. 

\begin{thm}[Morimoto, {\cite[Chapter 3]{Morimoto}}]
Let $\mathfrak{X}$ be a Carnot manifold and $\mathfrak{n}$ a graded, nilpotent Lie algebra. Then $\mathfrak{X}$ is regular of type $\mathfrak{n}$ if and only if for any point $x\in \mathfrak{X}$, there is a graded Lie algebra isomorphism $\mathfrak{t}_H\mathfrak{X}_x\cong \mathfrak{n}$.
\end{thm}

\begin{definition}
If $\mathfrak{X}$ and $\mathfrak{Y}$ are Carnot manifolds, and $f:\mathfrak{X}\to \mathfrak{Y}$ is a smooth mapping we say that $f$ is a Carnot morphism if 
$$Df(T^{j}\mathfrak{X})\subseteq T^j\mathfrak{Y}, \quad\forall j$$
If $f$ additionally is a diffeomorphism with $Df(T^{j}\mathfrak{X})= T^j\mathfrak{Y}$ for all $j$, we say that $f$ is a Carnot isomorphism or if $\mathfrak{X}=\mathfrak{Y}$ in which case we say that $f$ is a Carnot automorphism. We write $\mathrm{Aut}_C(\mathfrak{X})$ for the group of Carnot automorphisms of $\mathfrak{X}$. 

A complex Carnot isomorphism (or automorphism) $f:\mathfrak{X}\to \mathfrak{Y}$ of complex Carnot manifolds is a Carnot isomorphism which is also holomorphic. We write $\mathrm{Aut}_{\mathcal{O}C}(\mathfrak{X})$ for the group of holomorphic Carnot automorphisms of a complex Carnot manifold $\mathfrak{X}$. 
\end{definition}

\begin{definition}
Let $G$ be a second countable, locally compact group. A (complex) $G$-Carnot manifold is a (complex) Carnot manifold $\mathfrak{X}$ equipped with a continuous action of $G$ as (complex) Carnot automorphisms.
\end{definition}

\subsection{Restricted frame bundles}

Let $\mathsf{N}$ be a graded, simply connected, nilpotent Lie group with graded Lie algebra $\mathfrak{n}=\oplus_{j=-r}^{-1}\mathfrak{n}_j$. We write $\mathrm{Aut}_\R(\mathfrak{n})$ for the group of linear, invertible maps on $\mathfrak{n}$. We let $\mathrm{Aut}_{\rm gr}(\mathsf{N})\cong \mathrm{Aut}_{\rm gr}(\mathfrak{n})$ denote its closed subgroup of graded Lie group automorphisms. By setting $\mathfrak{n}^k:=\oplus_{j=k}^{-1}\mathfrak{n}_j$ we obtain a filtration 
\begin{equation}
\label{filtofn}
\mathfrak{n}=\mathfrak{n}^{-r}\supsetneq\mathfrak{n}^{-r+1}\supsetneq \cdots \supsetneq \mathfrak{n}^{-2}\supsetneq\mathfrak{n}^{-1}\supsetneq 0,
\end{equation}
of $\mathfrak{n}$ which is clearly $\mathrm{Aut}_{\rm gr}(\mathsf{N})$-invariant. We write $\mathrm{Aut}_{\rm fil}(\mathsf{N})$ for the group of Lie algebra automorphisms of $\mathfrak{n}$ preserving the filtration \eqref{filtofn}. For a manifold $\mathfrak{X}$ of dimension $\dim(\mathfrak{n})$, we can consider the frame bundle $P_{\mathfrak{X}}\to \mathfrak{X}$ as a principal $\mathrm{Aut}_\R(\mathfrak{n})$-bundle since $\mathrm{Aut}_\R(\mathfrak{n})\cong GL(n,\R)$ for $n=\dim(\mathfrak{n})$. 

\begin{definition}
\label{defofdedgrde}
We introduce the notion of a filtered Carnot principal bundle of type $\mathsf{N}$ on a manifold $\mathfrak{X}$ of dimension $\dim(\mathfrak{n})$ to be a principal $\mathrm{Aut}_{\rm fil}(\mathsf{N})$-bundle $P\to \mathfrak{X}$ satisfying the following:
\begin{enumerate}
\item the frame bundle $P_\mathfrak{X}\to \mathfrak{X}$ reduces to $P$ under the natural map $\mathrm{Aut}_{\rm fil}(\mathsf{N})\to \mathrm{Aut}_\R(\mathfrak{n})$;
\item the natural inclusions $T^{k}\mathfrak{X}:=P\times_{\mathrm{Aut}_{\rm fil}(\mathfrak{n})}\mathfrak{n}^k\subseteq T\mathfrak{X}$, $k=-r,\ldots, -1$, defined from the reduction $P\hookrightarrow P_{\mathfrak{X}}$, form a Carnot structure on $\mathfrak{X}$; and 
\item the identity map on $T\mathfrak{X}$ induces an isomorphism of Lie groupoids $T_X\mathfrak{X}\cong P\times_{\mathrm{Aut}_{\rm fil}(\mathsf{N})}\mathsf{N}$, or in other words the Levi bracket (defined from the Carnot structure) and the algebraic bracket (defined from $\mathrm{gr}(T\mathfrak{X})\cong P\times_{\mathrm{Aut}_{\rm fil}(\mathfrak{n})}\mathfrak{n}$ and the Lie bracket on $\mathfrak{n}$) coincide.
\end{enumerate}
A filtered $G$-equivariant Carnot principal bundle of type $\mathsf{N}$ is a filtered Carnot principal bundle $P\to \mathfrak{X}$ of type $\mathsf{N}$ where $P$ and $\mathfrak{X}$ are equipped with $G$-actions such that 
\begin{enumerate}
\item[i)] the projection mapping $P\to \mathfrak{X}$ is $G$-equivariant;
\item[ii)] the reduction mapping $P\hookrightarrow P_{\mathfrak{X}}$ is $G$-equivariant; and 
\item[iii)] the $G$-action on $P$ commutes with the $\mathrm{Aut}_{\rm fil}(\mathsf{N})$-action.
\end{enumerate}
\end{definition}

\begin{prop}
Consider a graded, simply connected Lie group $\mathsf{N}$. There is an essentially surjective functor from the category of filtered $G$-equivariant Carnot principal bundles of type $\mathsf{N}$ to regular $G$-Carnot manifolds $\mathfrak{X}$ of type $\mathsf{N}$. The functor is defined from
$$P\mapsto (T^{k}\mathfrak{X}=P\times_{\mathrm{Aut}_{{\rm fil}}(\mathfrak{n})}\mathfrak{n}^k)_{k=-r,\ldots, -1}.$$ 
In particular, any regular Carnot structure can be realized from a (non-canonical) reduction of structure group to the group of filtered automorphisms.
\end{prop}

\subsection{Parabolic geometry}
\label{jnkajnakjdnapara}
A parabolic geometry is a Cartan geometry of type $(\mathsf{G},\mathsf{P})$ for a semisimple Lie group $\mathsf{G}$ and a parabolic subgroup $\mathsf{P}$. This is to say that a manifold $\mathfrak{X}$ is a parabolic geometry of type $(\mathsf{G},\mathsf{P})$ if it in a certain infinitesimal sense looks like $\mathsf{G}/\mathsf{P}$ at each point. A fuller picture of parabolic geometry is found in the textbook \cite{capslovak}. We here recall its salient feature and in the meanwhile set notations. \emph{We use the convention that Lie groups used to define parabolic geometries are denoted by $\mathsf{G}$, to distinguish from a group $G$ acting on a geometry.}

Consider a connected, semisimple Lie group $\mathsf{G}$. Write $\mathfrak{g}$ for its Lie algebra. We often assume that $\mathfrak{g}$ is $|k|$-graded, i.e. we have fixed a grading 
$$\mathfrak{g}=\mathfrak{g}_{-k}\oplus \mathfrak{g}_{-k+1}\oplus \cdots \oplus \mathfrak{g}_{-1}\oplus \mathfrak{g}_0\oplus \mathfrak{g}_{1}\oplus \cdots \oplus\mathfrak{g}_{k-1}\oplus \mathfrak{g}_{k},$$
with $[\mathfrak{g}_{j},\mathfrak{g}_l]\subseteq \mathfrak{g}_{j+l}$ for any $j,l$. We let $\mathsf{P}\subseteq \mathsf{G}$ denote a parabolic subgroup that we tacitly assume to relate to the $|k|$-grading of $\mathfrak{g}$ via its Lie algebra $\mathfrak{p}$ and the identity 
$$\mathfrak{p}=\mathfrak{g}_0\oplus \mathfrak{g}_{1}\oplus \cdots \oplus\mathfrak{g}_{k-1}\oplus \mathfrak{g}_{k}.$$
Write $\mathfrak{p}_+=\mathfrak{g}_+:=\mathfrak{g}_{1}\oplus \cdots \oplus\mathfrak{g}_{k-1}\oplus \mathfrak{g}_{k}$ for the nilradical of $\mathfrak{p}$. Let $\mathsf{G}_0\subseteq \mathsf{G}$ be the reductive, closed sub-group integrating $\mathfrak{g}_0$. The Langlands decomposition $\mathsf{P}=\mathsf{M}\mathsf{A}\mathsf{P}_+=\mathsf{G}_0\mathsf{P}_+$, where $\mathsf{P}_+$ integrates $\mathfrak{p}_+$, factors $\mathsf{G}_0=\mathsf{M}\mathsf{A}$ into its semisimple and abelian part. 

\begin{example}
\label{kljnkjandaretrot}
The prototypical example of a $|k|$-grading to keep in mind comes about the other way; starting from a standard parabolic subgroup $\mathsf{P}\subseteq \mathsf{G}$ there is an associated $|k|$-grading of $\mathfrak{g}$ defined from the total degree in the restricted root spaces of $\mathsf{P}$ (see \cite{capslovak}, Chapter 3.2.1 for the complex case and Chapter 3.2.9 for the real case), in this $|k|$-grading $\mathfrak{p}=\mathfrak{g}_0\oplus \mathfrak{g}_{1}\oplus \cdots \oplus\mathfrak{g}_{k-1}\oplus \mathfrak{g}_{k}$ coincides with the Lie algebra of $\mathsf{P}$. 

Let us describe the parabolic subgroups and the associated gradings in some more details, following \cite[Chapter VI and VII]{knappbeyond} for the former and \cite[Chapter 3.2]{capslovak} for the latter. The standard parabolic subgroups are standard relative to a choice of Cartan involution $\theta$ on $\mathsf{G}$, with corresponding maximally abelian subalgebra $\mathfrak{a}$ of the $\theta$-invariant subspace of $\mathfrak{g}$, restricted roots $\Sigma$, a choice of positive roots $\Sigma^+$ and simple roots $\Pi$. The minimal parabolic subgroup is defined from $\mathfrak{p}_{\rm min}:=\mathfrak{g}_0\oplus_{\lambda\in \Sigma^+}\mathfrak{g}_\lambda$ where $\mathfrak{g}_0$ is the centralizer of $\mathfrak{a}$. The standard parabolic subgroups stand in a one-to-one correspondence with subsets $\Pi'\subseteq \Pi$ via 
$$\mathfrak{p}=\mathfrak{p}_{\rm min}\oplus_{\lambda\in \mathrm{span}(\Pi')} \mathfrak{g}_\lambda.$$
If $\Pi'$ is the set of simple restricted roots corresponding to $\mathsf{P}$, we can define $\mathrm{ht}_\mathsf{P}:\mathfrak{a}^*\to \R$ following \cite[Chapter 3.2.1]{capslovak} by
\begin{equation}
\label{adlknjandadjkn}
\mathrm{ht}_\mathsf{P}(\sum_{\alpha\in \Pi}a_\alpha \alpha)=\sum_{\alpha\in \Pi\setminus \Pi'}a_\alpha.
\end{equation}
The $|k|$-grading is defined by 
$$\mathfrak{g}_j:= \oplus_{\alpha: \mathrm{ht}_\mathsf{P}(\lambda)=j} \mathfrak{g}_\lambda, \quad j=-k,-k+1,\ldots, k-1,k,$$
for $k$ being the largest value of $\mathrm{ht}_\mathsf{P}$ on $\Sigma$. We note that for $\mathfrak{a}_{\mathrm{P}}:=\cap \ker_{\beta\in \mathrm{span}(\Pi')}\ker\beta\subseteq \mathfrak{a}$ having dimension $>1$, i.e. $\mathsf{P}$ has rank $>1$, the $|k|$-grading is overlooking a great deal of information pertained by the multifiltering \cite{yunckenhab} defined from the restricted roots.

\end{example}

We shall make heavy use of the graded nilpotent Lie algebra
$$\mathfrak{n}:=\mathfrak{g}_{-k}\oplus \mathfrak{g}_{-k+1}\oplus \cdots \oplus \mathfrak{g}_{-1}.$$
Write $\mathsf{N}$ for the simply connected Lie group integrating $\mathfrak{n}$. The $|k|$-grading ensures that $\mathsf{G}_0$ acts as graded Lie algebra automorphisms of $\mathfrak{n}$. Note that there is an abstract Lie algebra isomorphism $\mathfrak{n}\cong \mathfrak{p}_+$ since $\mathfrak{g}$ is semisimple. The Killing form implements a natural $\mathsf{G}_0$-equivariant isomorphism
\begin{equation}
\label{gzeroadaond}
\mathfrak{n}\cong \mathfrak{p}_+^*.
\end{equation}
By an abuse of notation, we shall use this isomorphism as an identity. We also write $\mathfrak{n}^{-j}:=\mathfrak{g}_{-j}\oplus \mathfrak{g}_{-j+1}\oplus \cdots \oplus \mathfrak{g}_{-1}.$

The flag manifold 
$$\mathfrak{X}:=\mathsf{G}/\mathsf{P},$$
carries a natural $\mathsf{G}$-action and a $\mathsf{G}$-equivariant filtration 
$$0\subseteq T^{-1}\mathfrak{X}\subseteq T^{-2}\mathfrak{X}\subseteq T^{-k+1}\mathfrak{X}\subseteq T^{-k}\mathfrak{X}=T\mathfrak{X},$$
defined from identifying $\mathfrak{g}/\mathfrak{p}=\mathfrak{n}$ and setting 
$$T^{-j}\mathfrak{X}:=\mathsf{G}\times_\mathsf{P}\mathfrak{n}^{-j}=\bigoplus_{l=1}^j (\mathsf{G}/\mathsf{P}_+)\times_{\mathsf{G}_0}\mathfrak{g}_{-l}.$$
Here the second equality is not $\mathsf{G}$-equivariant, cf. the discussion below on page \pageref{lknaldknaada}. As such, $\mathfrak{X}$ is a $\mathsf{G}$-Carnot manifold of type $\mathsf{N}$. The Carnot frame bundle of $\mathfrak{X}$ naturally identifies with $\mathsf{G}\times_{\mathsf{P}}\mathrm{Aut}_{\rm fil}(\mathsf{N})$. More generally, we have the following definition.

\begin{definition}
Let $\mathsf{G}$ be a $|k|$-graded, semisimple, connected Lie group with associated parabolic subgroup $\mathsf{P}\subseteq \mathsf{G}$. A parabolic manifold of type $(\mathsf{G},\mathsf{P})$ is a pair $(\mathfrak{X},\omega)$ of a smooth manifold $\mathfrak{X}$ and a Cartan connection $\omega:\mathcal{V}\to T^*\mathcal{V}\otimes \mathfrak{g}$ on a principal $\mathsf{P}$-bundle $\mathcal{V}\to \mathfrak{X}$ such that 
\begin{enumerate}
\item $\mathrm{Ad}(p)(p^*\omega)=\omega$, for all $p\in \mathsf{P}$, where in $T^*\mathcal{V}\otimes \mathfrak{g}$ the operation $\mathrm{Ad}(p)$ is on the second factor and the pullback is in the first factor;
\item $\omega(\mathfrak{X}_\xi)=\xi$ for $\xi\in \mathfrak{p}$, where $X_\xi$ is the vector field on $\mathcal{V}$ associated with $\xi\in \mathfrak{p}$ and the $\mathsf{P}$-action on $\mathcal{V}$;
\item $\omega_p:T_p\mathcal{V}\to \mathfrak{g}$ is a linear isomorphism for all points $p\in \mathcal{V}$.
\end{enumerate}
If $\mathsf{G}$ and $\mathsf{P}$ are complex, we tacitly assume all structures to be holomorphic and speak of a complex parabolic manifold of type $(\mathsf{G},\mathsf{P})$.

For two parabolic manifolds $(\mathfrak{X},\omega)$ and $(\mathfrak{X}',\omega')$ of type $(\mathsf{G},\mathsf{P})$ we say that a diffeomorphism $f:\mathfrak{X}\to \mathfrak{X}'$ is parabolic if $f$ lifts to a $\mathsf{P}$-equivariant diffeomorphism $f_{\mathcal{V}}:\mathcal{V}\to \mathcal{V}'$ such that $f_{\mathcal{V}}^*\omega'=\omega$. If $(\mathfrak{X},\omega)$ and $(\mathfrak{X}',\omega')$ are complex parabolic manifolds of type $(\mathsf{G},\mathsf{P})$, a parabolic diffeomorphism $f:\mathfrak{X}\to \mathfrak{X}'$ is a parabolic diffeomorphism of the underlying real structure which is holomorphic and lifts to a holomorphic $f_{\mathcal{V}}$. We write $\mathrm{Aut}(\mathfrak{X},\omega)$ for the group of parabolic self-diffeomorphisms of a parabolic geometry, and if $(\mathfrak{X},\omega)$ is a complex parabolic manifold we write $\mathrm{Aut}(\mathfrak{X},\omega)$  for the group of complex parabolic self-diffeomorphisms.
\end{definition}

\begin{remark}
We can topologize $\mathrm{Aut}(\mathfrak{X},\omega)$ in many different ways, we shall choose the $C^\infty$-topology. By work of Frances-Melnick \cite{framel} it coincides with the $C^0$-topology.
\end{remark}

If $(\mathfrak{X},\omega)$ is a parabolic geometry, the Cartan connection induces a restriction of the frame bundle of $\mathfrak{X}$ to the $\mathsf{G}_0$-principal bundle $\mathcal{V}/\mathsf{P}_+\to \mathfrak{X}$. By the discussion above, there is a group homomorphism $\mathsf{G}_0\to \mathrm{Aut}_{\rm gr}(\mathsf{N})$ so there is an associated filtration $T^{-k}\mathfrak{X}:=(\mathcal{V}/\mathsf{P}_+)\times_{\mathsf{G}_0}\mathfrak{n}^k$ of $T\mathfrak{X}$. This filtration together with the information of the $\mathsf{G}_0$-principal bundle $\mathcal{V}/\mathsf{P}_+\to \mathfrak{X}$ and the $\mathsf{G}_0$-invariant forms $\theta_i: T^i(\mathcal{V}/\mathsf{P}_+)\to \mathfrak{g}_i$ induced from the Cartan connection $\omega$, is called the infinitesimal flag structure of $(\mathfrak{X},\omega)$, see more in \cite[Definition 3.1.6]{capslovak}. The infinitesimal flag structure does not automatically define a Carnot structure, and we make the following definition.

\begin{definition}
We say that $(\mathfrak{X},\omega)$ is a {\bf regular} parabolic geometry of type $(\mathsf{G},\mathsf{P})$ if the $\mathrm{Aut}_{\rm fil}(\mathsf{N})$-principal bundle 
$$\mathcal{V}\times_{\mathsf{P}}\mathrm{Aut}_{\rm fil}(\mathsf{N})\to \mathfrak{X},$$
is a filtered Carnot principal bundle of type $\mathsf{N}$ in the sense of Definition \ref{defofdedgrde}. Or in other words, $(\mathfrak{X},\omega)$ is a {\bf regular} parabolic geometry if the infinitesimal flag structure induces a Carnot structure whose Levi bracket coincides with the algebraically defined bracket defined from $\mathcal{V}$. 
\end{definition}

For parabolic geometries, there is more control of geometric structure than for general Carnot manifolds. For instance, it is shown in \cite[Chapters 3.1.14, and 3.1.16]{capslovak} that there is an equivalence of categories between regular infinitesimal flag structures of type $(\mathsf{G},\mathsf{P})$ on $\mathfrak{X}$ and normal, regular parabolic geometries of type $(\mathsf{G},\mathsf{P})$ on $\mathfrak{X}$. Here normal means that the curvature function 
$$\kappa:\mathcal{V}\to \mathrm{Hom}(\wedge^2(\mathfrak{g}/\mathfrak{p}),\mathfrak{g}),\quad [\kappa(p)](X,Y):=[X,Y]-\omega_p([\omega^{-1}(X),\omega^{-1}(Y)]_p),$$ 
is a cycle for the Kostant boundary map (cf. \cite[Chapter 3.1.12]{capslovak}).

\begin{remark}
If $(\mathfrak{X},\omega)$ is a regular parabolic manifold the discussion above shows that there is an induced Carnot structure. It is clear that we have an inclusion of groups $\mathrm{Aut}(\mathfrak{X},\omega)\hookrightarrow\mathrm{Aut}_C(\mathfrak{X})$. By the same token, if $(\mathfrak{X},\omega)$ is a regular complex parabolic manifold, then $\mathrm{Aut}(\mathfrak{X},\omega)\hookrightarrow\mathrm{Aut}_{\mathcal{O}C}(\mathfrak{X})$. 

The inclusion $\mathrm{Aut}(\mathfrak{X},\omega)\subseteq \mathrm{Aut}_C(\mathfrak{X})$ is in general strict. For instance, if $\mathfrak{X}=S^{2n-1}$ with $\omega$ defined to make $(\mathfrak{X},\omega)$ into the flat parabolic geometry $S^{2n-1}=SU(n,1)/\mathsf{P}$ for the parabolic subgroup $\mathsf{P}\subseteq SU(n,1)$. Then $\omega$ is a CR-structure and $\mathrm{Aut}(\mathfrak{X},\omega)=SU(n,1)$ consists of the CR-automorphisms of $\mathfrak{X}$. However, $\mathfrak{X}$ viewed as a Carnot manifold is a contact manifold and complex conjugation in $S^{2n-1}\subseteq \C^n$ is a filtered diffeomorphism so we have an inclusion $\mathrm{Aut}(\mathfrak{X},\omega)\rtimes (\Z/2)\subseteq \mathrm{Aut}_C(\mathfrak{X})$. Of course, in the case at hand when $\mathfrak{X}$ is a contact manifold, $\mathrm{Aut}_C(\mathfrak{X})$ is the group of contactomorphisms which is infinite-dimensional.
\end{remark}

We can in general expect there to be few parabolic automorphisms as the following generalization of Ferrand-Obata's theorem and Schoen's theorem \cite{schoencr} shows. 

\begin{thm}[Bader-Frances-Melnick \cite{badframel}]
\label{thmbaderfranc}
Let $(\mathfrak{X},\omega)$ be a regular parabolic manifold of type $(\mathsf{G},\mathsf{P})$ but not a discrete quotient of the universal cover of $\mathsf{G}/\mathsf{P}$. Then $\mathrm{Aut}(\mathfrak{X},\omega)$ is smaller than $\mathsf{G}$ in the sense that the adjoint representation of any connected Lie subgroup $H\subseteq \mathrm{Aut}(\mathfrak{X},\omega)$ has rank strictly less than the rank of $\mathsf{G}$. In particular, if $\mathrm{rk}(\mathsf{G})=1$, then $\mathrm{Aut}(\mathfrak{X},\omega)$ is compact. 
\end{thm}

\section{Heisenberg calculus}
\label{subsec:heis}

An important technical tool throughout the paper is the Heisenberg calculus on Carnot manifolds. We shall only recall its overall structure and refer the details to the literature \cite{androerp,Dave_Haller1, ewertjncg, goffkuz,vanerpyuncken}. In this section we tacitly assume $\mathfrak{X}$ to be a compact Carnot manifold. Recall the Lie algebroid $\mathfrak{t}_H\mathfrak{X}:=\oplus_j T^{-j}\mathfrak{X}/T^{-j+1}\mathfrak{X}$ defined from a Carnot structure. Each fibre of the vector bundle $\mathfrak{t}_H\mathfrak{X}$ is a nilpotent Lie algebra, and we can integrate to a Lie groupoid $T_H\mathfrak{X}\to \mathfrak{X}$. As a fibre bundle, $T_H\mathfrak{X}\to \mathfrak{X}$ is a vector bundle but the fibrewise Lie group structure is a polynomial operation determined from the Baker-Campbell-Hausdorff formula and may vary in general. There is an $\R_+$-action on $\mathfrak{t}_H\mathfrak{X}$ defined by $\delta_\lambda(X):=\lambda^{-j}X$ for $X\in T^{-j}\mathfrak{X}/T^{-j+1}\mathfrak{X}$. By construction, $\delta_\lambda$ is a fibrewise Lie algebra automorphism integrating to a smooth $\R_+$-action on the Lie groupoid $T_H\mathfrak{X}$. 

\subsection{The parabolic tangent groupoid} 
To build the Heisenberg calculus, \cite{vanerpyuncken} combined Connes' tangent groupoid with Debord-Skandalis' approach to homogeneity. The analogue of Connes' tangent groupoid is the parabolic tangent groupoid $\mathbb{T}_H\mathfrak{X}\rightrightarrows \mathfrak{X}\times [0,\infty)$. Set theoretically, the groupoid  $\mathbb{T}_H\mathfrak{X}$ is defined by
$$\mathbb{T}_H\mathfrak{X}:= T_H\mathfrak{X}\times \{0\} \, \dot{\bigcup}  \mathfrak{X}\times \mathfrak{X}\times (0,\infty).$$
The groupoid $\mathbb{T}_H\mathfrak{X}$ is endowed with a Lie groupoid structure by choosing a graded connection $\nabla$ on $\mathfrak{t}_H\mathfrak{X}$ and declaring the mapping
$$\psi:T_HX\times [0,\infty)\to \mathbb{T}_H\mathfrak{X}, \quad \psi(x,v,t):=
\begin{cases} 
(x,v,0), \; &t=0,\\
(\mathrm{exp}^\nabla(\delta_t(v)),x,t), \; &t>0,
\end{cases}$$
to be a local diffeomorphism. Here we are implicitly using a splitting isomorphism $T\mathfrak{X}\cong \mathfrak{t}_H\mathfrak{X}\equiv \mathrm{gr}(T\mathfrak{X})$ (in the sense of \cite[Definition 9]{vanerpyuncken2}) to define $\mathrm{exp}^\nabla(\delta_t(v))$. The Lie groupoid $T_H\mathfrak{X}$ is independent of choice of splitting isomorphism and graded connection by \cite[Theorem 16]{vanerpyuncken2}. The scaling action of $\R_+$ extends to a smooth $\R_+$-action on the parabolic tangent groupoid by extending it to $t>0$ via 
$$\delta_\lambda(x,y,t)=(x,y,\lambda^{-1}t).$$

\subsection{van Erp-Yuncken's calculus}
Following \cite{vanerpyuncken}, for two vector bundles $E_1,E_2\to \mathcal{X}$, we consider the space $\mathcal{E}'_r(\mathbb{T}_H\mathfrak{X};E_1,E_2)$ of properly supported $r$-fibred distributions on the Lie groupoid $\mathbb{T}_H\mathfrak{X}\rightrightarrows \mathfrak{X}\times [0,\infty)$ with coefficients in $r^*E_2\otimes s^*E_1^*$. Pushing forward along the $\R_+$-action induces an isomorphism $(\delta_\lambda)_*$ of $\mathcal{E}'_r(\mathbb{T}_H\mathfrak{X};E_1,E_2)$. Let $|\Lambda|$ denote the $1$-densities on $\mathfrak{X}$ and $|\Lambda_r|:=r^*|\Lambda|$. We write $C^\infty_p( \mathbb{T}_H\mathfrak{X};r^*E_2\otimes s^*E_1^*\otimes |\Lambda_r|)$ for the space of smooth properly supported sections of $r^*E_2\otimes s^*E_1^*\otimes |\Lambda_r|$, this is precisely the space of smooth elements of $\mathcal{E}'_r(\mathbb{T}_H\mathfrak{X};E_1,E_2)$. Evaluation in $t=0$ produces a mapping 
$$\mathrm{ev}_{t=0}:\mathcal{E}'_r(\mathbb{T}_H\mathfrak{X};E_1,E_2)\to \mathcal{E}'_r(T_H\mathfrak{X};E_1,E_2),$$
onto the analogous space of distributions on the Lie groupoid of osculating Lie groups. Evaluation in $t\in (0,\infty)$ produces a mapping 
$$\mathrm{ev}_{t}:\mathcal{E}'_r(\mathbb{T}_H\mathfrak{X};E_1,E_2)\to \mathcal{E}'_r(\mathfrak{X}\times \mathfrak{X};E_1,E_2).$$
We note that by the Schartz kernel theorem, 
$$\mathcal{E}'_r(\mathfrak{X}\times \mathfrak{X};E_1,E_2)=C^\infty(\mathfrak{X}; E_2\otimes \mathcal{D}'(\mathfrak{X}; E_1^*\otimes |\Lambda|)).$$

\begin{definition}
\label{liftotflknkjnad}
Let $\mathfrak{X}$ be a compact Carnot manifold, $m\in \C$ and $E_1,E_2\to \mathfrak{X}$ two vector bundles. We write $\pmb{\Psi}^m_H( \mathfrak{X},E_1, E_2)$ for the space of properly supported, $r$-fibred distribution $\pmb{A}\in \mathcal{E}_r'(\mathbb{T}_H\mathfrak{X};E_1,E_2)$ which is essentially homogeneous of degree $m$ in the sense that 
$$(\delta_\lambda)_*\pmb{A}-\lambda^m\pmb{A}\in C^\infty_p( \mathbb{T}_H\mathfrak{X};r^*E_2\otimes s^*E_1^*\otimes |\Lambda_r|).$$

We say that an operator $A:C^\infty(\mathfrak{X},E_1)\to C^\infty(\mathfrak{X},E_2)$ is a Heisenberg pseudodifferential operator of order $m\in \C$ if there exists a properly supported, $r$-fibred distribution $\pmb{A}\in \pmb{\Psi}^m_H( \mathfrak{X},E_1, E_2)$ such that the Schwartz kernel $k_A\in C^\infty(\mathfrak{X}; E_2\otimes \mathcal{D}(\mathfrak{X}; E_1^*\otimes |\Lambda|))$ is given by 
$$k_A:=\mathrm{ev}_{t=1}(\pmb{A}).$$
\end{definition}

It can be shown \cite{Dave_Haller1,goffkuz,vanerpyuncken} that composition, adjoints, and addition are well defined and respect homogeneity so we arrive at a filtered $*$-algebra (under composition of operators) of Heisenberg pseudodifferential operators $\Psi^*_H( \mathfrak{X},E)$ for any vector bundle $E\to \mathfrak{X}$. We write $\Psi^*_H(\mathfrak{X};E_1,E_2)$ for the space of Heisenberg pseudodifferential operator $C^\infty(\mathfrak{X},E_1)\to C^\infty(\mathfrak{X},E_2)$ that viewed as a summand in $\Psi^*_H( \mathfrak{X},E_1\oplus E_2)$ inherits the relevant algebraic properties.

\subsection{Heisenberg-Sobolev spaces}
Turning to analytic properties, still assuming $\mathfrak{X}$ to be a compact Carnot manifold, we can construct a positive, invertible, even order differential operator $\mathfrak{D}\in \Psi^{2m}_H( \mathfrak{X},E)$ admitting complex powers $\mathfrak{D}^z\in \Psi^{2mz}_H( \mathfrak{X};E)$, $z\in \C$, by \cite{Dave_Haller2}. From this family we define the scale of Heisenberg-Sobolev spaces
$$W^s_H(\mathfrak{X};E):=\mathfrak{D}^{-s/2m}L^2(\mathfrak{X};E).$$
By \cite{Dave_Haller1}, for $A\in \Psi^m_H(X;E_1,E_2)$ the operator $A:C^\infty(\mathfrak{X},E_1)\to C^\infty(\mathfrak{X},E_2)$ extends by density to a continuous operator
\begin{equation}
\label{ljnadjnakjnakdnadkjnadkjn}
A:W^s_H(X;E_1)\to W^t_H(X;E_2), \quad\mbox{for $s\geq t+\mathrm{Re}(m)$},
\end{equation}
which is compact if $s> t+\mathrm{Re}(m)$. 

If $\tilde{\mathfrak{D}}$ is another such choice of a positive, invertible, even order differential operator of order $2\tilde{m}$ then $T(z):=\tilde{\mathfrak{D}}^{-\frac{z}{2\tilde{m}}}\mathfrak{D}^{\frac{z}{2m}}\in \Psi^0_H(\mathfrak{X};E)$ for a complex number $z$. Therefore, the property stated in Equation \eqref{ljnadjnakjnakdnadkjnadkjn} implies that the spaces $W^s_H(\mathfrak{X};E)$, $s\in \R$, are independent of the choice of $\mathfrak{D}$.

\subsection{Symbolic structure}
Let us now describe the microlocal structure of a Heisenberg pseudodifferential operator. Take $A\in \Psi^*_H( \mathfrak{X},E)$ with lift $\pmb{A}$ as in Definition \ref{liftotflknkjnad}. The $r$-fibered distribution $\pmb{A}$ admits a Taylor expansion at $t=0$, that translates into an asymptotic expansion 
$$
(\psi^{-1})_*\pmb{A}(x,v,t)\sim \sum_{j=0}^\infty t^jk_j(x,v),
$$
where $k_j\in \mathcal{E}'_r(T_H\mathfrak{X})\cap C^\infty(T_H\mathfrak{X}\setminus \mathfrak{X})$ are essentially homogeneous in the sense that $(\delta_\lambda)_*k_j-\lambda^{m-j}k_j\in C^\infty_c(T_H\mathfrak{X})$, for $\lambda>0$. In particular, we can deduce that for any $k\in \N$ there exists an $N\in \N$ such that 
$$k_A(x,\mathrm{exp}_x^\nabla(v))- \sum_{j=0}^N k_j(x,v)\in C^k(T_H \mathfrak{X}; E_2\otimes E_1^*\otimes |\Lambda|).$$
By construction, we see that $\pmb{A}$ is determined modulo smooth, properly supported kernels by $A$. In particular, we can define the principal symbol of $A$ as the homogeneous element
$$\sigma_H^m(A):=\mathrm{ev}_{t=0}(\pmb{A})\in \mathcal{E}'_r(T_H\mathfrak{X})/C^\infty_c(T_H\mathfrak{X}).$$
Clearly, $\sigma_H^m(A)$ can be represented by $k_0+C^\infty_c(T_H\mathfrak{X})$. The main features of the principal symbol can be captured in the following theorem. We use the notation $\mathcal{P}^m(\mathfrak{X};E_1,E_2)\subseteq C^\infty(T_H \mathfrak{X}; E_2\otimes E_1^*\otimes |\Lambda|)$ for the space of fibrewise polynomial densities of total degree $m$. In particular, $\mathcal{P}^m(\mathfrak{X};E_1,E_2)=0$ unless $m\in-\dim_h(\mathfrak{X})-\N$.

\begin{thm}
\label{lkjandlkjandkjadn}
Let $\mathfrak{X}$ be a compact Carnot manifold and $E\to \mathfrak{X}$ a vector bundle and write $\Psi^*_H( \mathfrak{X},E)$ for the filtered algebra of Heisenberg pseudodifferential operators. For any $m\in \C$, $\Psi^m_H( \mathfrak{X},E)$ fits into a short exact sequence 
$$0\to \Psi^{m-1}_H( \mathfrak{X},E)\to \Psi^m_H( \mathfrak{X},E)\to \Sigma^m_H( \mathfrak{X},E)\to 0,$$
where the symbol algebra $\Sigma^m_H( \mathfrak{X},E)\subseteq  \mathcal{D}'_r(T_H\mathfrak{X})/\mathcal{P}^m(\mathfrak{X};E)$ consists of elements of the form 
\begin{equation}
\label{deceomoked}
k_m+p_m\log|\cdot|_H,
\end{equation}
for $k_m\in \mathcal{D}'_r(T_H\mathfrak{X};E)/\mathcal{P}^m(X;E)$ being homogeneous of degree $m$, $p_m\in \mathcal{P}(\mathfrak{X};E)$ and $|\cdot|_H$ a fibrewise gauge. The symbol algebra $\Sigma^m_H( \mathfrak{X},E)$ also fits into the short exact sequence 
$$0\to C^\infty_c(T_H \mathfrak{X},\Hom(E)\otimes |\Lambda|)\to \tilde{\Sigma}^m_H( \mathfrak{X},E)\to \Sigma^m_H( \mathfrak{X},E)\to 0,$$
where 
\begin{align*}
\tilde{\Sigma}^m_H( \mathfrak{X},E):=&\mathrm{im}(\mathrm{ev}_{t=0}:\pmb{\Psi}^m_H( \mathfrak{X},E)\to \mathcal{E}'_r(T_H\mathfrak{X};\Hom(E)\otimes |\Lambda|))=\\
=&\left\{ k\in \mathcal{E}'_r(T_H\mathfrak{X};\Hom(E)\otimes |\Lambda|): \lambda_* k-\lambda^m k\in  C^\infty_c, \; \lambda>0\right\}.
\end{align*}
\end{thm}

\subsection{Represented symbols and the Rockland condition}
To better understand the symbol algebra $\Sigma^m_H( \mathfrak{X},E)$ we study its action as fibrewise convolution operators. We write $\mathcal{S}(T_H\mathfrak{X};E)$ for the subspace of elements in $C^\infty(T_H\mathfrak{X};E)$ that together with all of its derivatives decay faster than polynomially. We introduce the notation $\mathcal{S}_0(T_H\mathfrak{X};E)$ for the subspace of elements $f\in \mathcal{S}(T_H\mathfrak{X};E)$ such that for any $x\in \mathfrak{X}$ and any polynomial $p$ on $T_H\mathfrak{X}_x$ we have $\int_{T_H\mathfrak{X}_x}p(v)f(x,v)\mathrm{d}x=0$. Convolving elements of  $\mathcal{S}_0(T_H\mathfrak{X};E)$ with $\mathcal{P}^m(X;E)$ is trivial, so any $a\in \Sigma^m_H(\mathfrak{X};E)$ can be identified with a convolution operator
$$a:\mathcal{S}_0(T_H\mathfrak{X};E)\to \mathcal{S}_0(T_H\mathfrak{X};E).$$
In fact, if $m\notin-\dim_h(\mathfrak{X})-\N$ then this convolution operator extends to an operator $a:\mathcal{S}(T_H\mathfrak{X};E)\to \mathcal{S}(T_H\mathfrak{X};E)$. The finer analytic structure of this convolution operator is studied in \cite[Part 5]{goffkuz}. This construction can be localized in representations of the osculating Lie groupoid. For $x\in \mathfrak{X}$ and $\pi$ a unitary representation of $T_H\mathfrak{X}_x$ on a Hilbert space $\mathcal{H}_\pi$, we write $\mathcal{S}_0(\pi):=\pi(\mathcal{S}_0(T_H\mathfrak{X}_x))\mathcal{H}_\pi$. If $\pi$ is irreducible and not the trivial representation, $\mathcal{S}_0(\pi)\subseteq \mathcal{H}_\pi$ is dense in the norm topology and in fact 
$$\mathcal{S}_0(\pi)=\mathcal{S}(\pi):=\pi(\mathcal{S}(T_H\mathfrak{X}_x))\mathcal{H}_\pi.$$ 
And convolving with a symbol $a\in \Sigma^m_H(\mathfrak{X};E)$ localizes to an operator
$$\pi(a):\mathcal{S}(\pi)\otimes E_x\to \mathcal{S}(\pi)\otimes E_x.$$
The same structures hold also for $a\in \Sigma^m_H( \mathfrak{X},E_1,E_2)$. An operator $A\in \Psi^m_H( \mathfrak{X};E_1,E_2)$ is said to satisfy the {\bf Rockland condition} if for any $x\in \mathfrak{X}$ and $\pi$ any irreducible, unitary, non-trivial representation of $T_H\mathfrak{X}_x$ the represented symbol 
$$\pi(\sigma^m_H(A)):\mathcal{S}(\pi)\otimes (E_1)_x\to \mathcal{S}(\pi)\otimes (E_2)_x,$$
is injective. 

Recall that a linear operator $T$ between Hilbert spaces is said to be left Fredholm if the image of $T$ is closed and the kernel of $T$ is finite-dimensional. By Atkinson's theorem, $T:\mathcal{H}_1\to \mathcal{H}_2$ is left Fredholm if and only if there is an operator $S:\mathcal{H}_2\to \mathcal{H}_1$ such that $1_{\mathcal{H}_1}-ST$ is compact.

\begin{thm}
\label{rockandhypo}
Let $\mathfrak{X}$ be a compact Carnot manifold and $A\in \Psi^m_H( \mathfrak{X};E_1,E_2)$. Then the following are equivalent:
\begin{enumerate}
\item $A$ satisfies the Rockland condition.
\item There is a $B\in \Psi^{-m}_H( \mathfrak{X};E_2,E_1)$ such that $BA-1\in \Psi^{-\infty}( \mathfrak{X};E_1)$.
\item There is a $b\in \Sigma^{-m}_H( \mathfrak{X};E_2,E_1)$ such that $b\sigma_H^m(A)=1$.
\item The operator 
$$A:W^s_H(X;E_1)\to W^{s-\mathrm{Re}(m)}_H(X;E_2),$$
is left-Fredholm for some $s$.
\end{enumerate}
\end{thm}

The proof of this theorem is referred to the literature. It is clear that 2) and 3) are equivalent, that 2) implies 4) and that 3) implies 1). That 1) implies 3) is proven in \cite{Dave_Haller1} and that 4) implies 1) follows from \cite[Theorem 3.43]{androerp}. The reader should note that \cite[Theorem 3.43]{androerp} vastly extends Theorem \ref{rockandhypo}, but the latter is sufficient for our purposes.

\subsection{Graded calculus}
\label{subsecgraded}
In many situations we consider in this paper, filtered vector bundles and filtered operators with filtered or graded analogues of the Rockland condition arise. The setup was developed in \cite[Section 4.1]{Dave_Haller1} and is close to the unfiltered situation, but to see the similarity appropriate conventions are needed. The reader can also compare to the setup of graded pseudodifferential calculus \cite{BGS}, also known as the Douglis-Nirenberg calculus \cite{dougnir} as used in the study of boundary value problems \cite{grubb}, see also \cite{horIII}. We write $\pmb{E}\to \mathfrak{X}$ for a filtered vector bundle, i.e. $\pmb{E}=\cup_p\pmb{E}^p$ for a finite filtration by 
vector bundles 
$$\cdots \supseteq \pmb{E}^{p-1}\supseteq \pmb{E}^{p}\supseteq \pmb{E}^{p+1}\supseteq \cdots,$$
with $\cap_p\pmb{E}^p=0$. There is an associated graded vector bundle 
$$\mathrm{gr}(\pmb{E}):=\oplus_p \pmb{E}[p],$$
where $\pmb{E}[p]:=\pmb{E}^{p}/ \pmb{E}^{p+1}$. The assignment $\pmb{E}\mapsto \mathrm{gr}(\pmb{E})$ is functorial. There always exists a splitting isomorphism $S_{\pmb{E}}: \mathrm{gr}(\pmb{E})\xrightarrow{\sim}\pmb{E}$, i.e. an isomorphism with $\mathrm{gr}(S_{\pmb{E}})=\mathrm{id}_{\mathrm{gr}(\pmb{E})}$. We shall fix such a splitting, but all the constructions will a posteriori be independent of splitting.

For the graded vector bundle $\mathrm{gr}(\pmb{E})\to \mathfrak{X}$, we define 
$$W^s_{H,{\rm gr}}(\mathfrak{X};\pmb{E}):=\bigoplus_{p} W^{s-p}(\mathfrak{X};\pmb{E}[p]), \quad s\in \R.$$
We define also the filtered Sobolev spaces
\begin{equation}
\label{filtsob}
W^s_{H,{\rm fil}}(\mathfrak{X};\pmb{E}):=S_{\pmb{E}}\bigoplus_{p} W^{s-p}(\mathfrak{X};\pmb{E}[p]), \quad s\in \R.
\end{equation}
It follows from the fact that $\mathrm{gr}(S_{\pmb{E}})=\mathrm{id}_{\mathrm{gr}(\pmb{E})}$ that $W^s_{H,{\rm fil}}(\mathfrak{X};\pmb{E})$ is independent of choice of splitting $S_{\pmb{E}}$.

Given two filtered bundles, $\pmb{E}_1,\pmb{E}_2\to \mathfrak{X}$ and a linear operator $T:C^\infty(\mathfrak{X}; \pmb{E}_1)\to C^\infty(\mathfrak{X}; \pmb{E}_2)$ we can write the operator $S_2^{-1}TS_1:C^\infty(\mathfrak{X}; \mathrm{gr}(\pmb{E}_1))\to (\mathfrak{X}; \mathrm{gr}(\pmb{E}_2))$, for splittings $S_1$ and $S_2$ of $\pmb{E}_1$ and $\pmb{E}_2$, respectively, as a matrix $S_2^{-1}TS_1=((S_2^{-1}TS_1))_{pq})_{p,q}$ for operators 
$$(S_2^{-1}TS_1)_{pq}:C^\infty(\mathfrak{X}; \pmb{E}_1[q])\to C^\infty(\mathfrak{X}; \pmb{E}_2[p]),$$ 
Following \cite[Section 5]{Dave_Haller1}, we define the graded calculus 
\begin{align*}
\Psi^m_{H,{\rm gr}}(\mathfrak{X}; \pmb{E}_1,\pmb{E}_2):=\bigg\{A:C^\infty(\mathfrak{X}; \pmb{E}_1)&\to C^\infty(\mathfrak{X}; \pmb{E}_2):\\
& (S_2^{-1}AS_1)_{pq}\in \Psi^{m+p-q}_H(\mathfrak{X};\pmb{E}_1[q],\pmb{E}_2[p]) \; \forall p,q \bigg\}.
 \end{align*}
Composition defines a well defined operation $\Psi^{m_1}_{H,{\rm gr}}(\mathfrak{X}; \pmb{E}_2,\pmb{E}_3)\times \Psi^{m_2}_{H,{\rm gr}}(\mathfrak{X}; \pmb{E}_1,\pmb{E}_2)\to \Psi^{m_1+m_2}_{H,{\rm gr}}(\mathfrak{X}; \pmb{E}_1,\pmb{E}_3)$ and the graded principal symbol map is defined as 
$$\sigma_{H,{\rm gr}}^m(A)=(\sigma_H^{m+p-q}((S_2^{-1}AS_1)_{pq}))_{p,q}.$$
We also consider the symbol algebra 
\begin{align*}
\Sigma^m_{H,{\rm gr}}(\mathfrak{X}; \pmb{E}_1,\pmb{E}_2):=\bigg\{a\in \mathcal{D}'_r(T_H\mathfrak{X}; \mathrm{gr}(\pmb{E}_1),\mathrm{gr}(\pmb{E}_2))&/\mathcal{P}(\mathfrak{X};\mathrm{gr}(\pmb{E}_1),\mathrm{gr}(\pmb{E}_2)): \\
 &a_{pq}\in \Sigma^{m+p-q}_H(\mathfrak{X};\pmb{E}_1[q],\pmb{E}_2[p]), \; \forall p,q \bigg\},
 \end{align*}
for the range (as well as codomain) of the graded symbol mapping. The graded symbol mapping is independent of the choice of splittings $S_1$ and $S_2$. We can again view the principal symbol as a convolution operator
$$\sigma_{H,{\rm gr}}^m(A):\mathcal{S}_0(T_H\mathfrak{X};\mathrm{gr}(\pmb{E}_1))\to \mathcal{S}_0(T_H\mathfrak{X};\mathrm{gr}(\pmb{E}_2)),$$
that for $x\in \mathfrak{X}$ and an irreducible, nontrivial, unitary representation $\pi$ of $T_H\mathfrak{X}_x$ localizes to an operator
$$\pi(\sigma_{H,{\rm gr}}^m(A)):\mathcal{S}(\pi)\otimes \mathrm{gr}(\pmb{E}_1)_x\to \mathcal{S}(\pi)\otimes \mathrm{gr}(\pmb{E}_2)_x.$$
If $\pi(\sigma_{H,{\rm gr}}^m(A))$ is injective for all irreducible, nontrivial, unitary representation $\pi$ we say that $A$ satisfies the Rockland condition. In the vein of Theorem \ref{rockandhypo}, $A\in \Psi^m_{H,{\rm gr}}(\mathfrak{X}; \pmb{E}_1,\pmb{E}_2)$ satisfies the Rockland condition if and only if $\sigma_{H,{\rm gr}}^m(A)$ admits a left inverse $b\in \Sigma^{-m}_{H,{\rm gr}}(\mathfrak{X}; \pmb{E}_2,\pmb{E}_1)$ if and only if 
$$A:W^s_{H,{\rm fil}}(\mathfrak{X};\pmb{E}_1)\to W^{s-m}_{H,{\rm fil}}(\mathfrak{X};\pmb{E}_2),$$
is left Fredholm.

\subsection{The symbol algebra on flat parabolic geometries} 
We end this section with a description of the symbol algebra of flat parabolic geometries $\mathfrak{X}=\mathsf{G}/\mathsf{P}$ where we later will be interested in the structure of equivariant symbols. We consider a connected, semisimple Lie group $\mathsf{G}$ and a standard parabolic subgroup $\mathsf{P}\subseteq \mathsf{G}$  with respect to a Cartan involution $\theta$ and a maximal compact subgroup $\mathsf{K}$. We have a Langlands decomposition 
$$\mathsf{P}=\mathsf{M}_\mathsf{P}\mathsf{A}_\mathsf{P}\mathsf{N}_\mathsf{P},$$
of $\mathsf{P}$. Here $\mathsf{M}_\mathsf{P}$ is a semisimple Lie group, $\mathsf{A}_\mathsf{P}$ a split torus over $\R$ and $\mathsf{N}_\mathsf{P}$ a nilpotent Lie group. We have that $\mathsf{G}/\mathsf{P}=\mathsf{K}/(\mathsf{K}\cap \mathsf{M}_\mathsf{P})$ as smooth manifolds, and the conjugate nilpotent group $\overline{\mathsf{N}}_\mathsf{P}:=\theta(\mathsf{N}_\mathsf{P})$ defines a dense open chart $\overline{\mathsf{N}}_\mathsf{P}\hookrightarrow \mathsf{G}/\mathsf{P}$ using Bruhat decompositions. Write $\mathcal{P}^m(\overline{\mathsf{N}}_\mathsf{P})$ for the polynomial densities of degree $m$ on $\overline{\mathsf{N}}_\mathsf{P}$. We introduce the notation $\Sigma^m_H\overline{\mathsf{N}}_\mathsf{P}\subseteq \mathcal{D}'_r(\overline{\mathsf{N}}_\mathsf{P})/\mathcal{P}^m(\overline{\mathsf{N}}_\mathsf{P})$ for the set of elements of the form in \eqref{deceomoked} but independent of the base space. Convolution defines a product 
$$\Sigma^m_H\overline{\mathsf{N}}_\mathsf{P}\times \Sigma^{m'}_H\overline{\mathsf{N}}_\mathsf{P}\to \Sigma^{m+m'}_H\overline{\mathsf{N}}_\mathsf{P}.$$
Since $\mathsf{M}_\mathsf{P}\mathsf{A}_\mathsf{P}$ normalizes $\mathsf{N}_\mathsf{P}$ and $\overline{\mathsf{N}}_\mathsf{P}$, it acts as automorphisms of $\Sigma^m_H\overline{\mathsf{N}}_\mathsf{P}$ respecting the product. The following is clear from Theorem \ref{lkjandlkjandkjadn}.

\begin{prop}
Consider the flat parabolic geometry $\mathfrak{X}=\mathsf{G}/\mathsf{P}$. For finite-dimensional $\mathsf{A}_\mathsf{P}\mathsf{M}_\mathsf{P}$-representations $W_1$ and $W_2$ we define the homogeneous vector bundles $E_j=\mathsf{G}\times_\mathsf{P} W_j\to \mathfrak{X}$. Then $\Sigma_H^m(X; E_1,E_2)$ consists of the smooth sections of a $\mathsf{G}$-equivariant bundle of algebras whose fibre is 
$$\Sigma^m_H\overline{\mathsf{N}}_\mathsf{P}\otimes \Hom(W_1,W_2).$$
More precisely, we have a $\mathsf{G}$-equivariant isomorphism 
$$\Sigma_H^m(\mathfrak{X}; E_1,E_2)\cong C^\infty(\mathsf{G}/\mathsf{N}_\mathsf{P},\Sigma^m_H\overline{\mathsf{N}}_\mathsf{P}\otimes \Hom(W_1,W_2))^{\mathsf{M}_\mathsf{P}\mathsf{A}_\mathsf{P}},$$
that respects products. Here the $\mathsf{M}_\mathsf{P}\mathsf{A}_\mathsf{P}$-action is the diagonal one and from the right on $\mathsf{G}/\mathsf{N}_\mathsf{P}$. In particular, the $\mathsf{G}$-equivariant symbols are described by
$$\Sigma_H^m(\mathfrak{X}; E_1,E_2)^\mathsf{G}\cong (\Sigma^m_H\overline{\mathsf{N}}_\mathsf{P}\otimes \Hom(W_1,W_2))^{\mathsf{M}_\mathsf{P}\mathsf{A}_\mathsf{P}}$$
\end{prop}

Recall the definition of $\mathrm{ht}_\mathsf{P}$ from Equation \eqref{adlknjandadjkn} and write $\mathfrak{a}^*_\C:=\mathfrak{a}^*\otimes_\R \C$. We can extend this map to a complex linear map $\mathrm{ht}_\mathsf{P}:\mathfrak{a}^*_\C\to \C$ and restrict it to a map $\mathrm{ht}_P:\mathfrak{a}^*_{\mathsf{P},\C}\to \C$ which is an isomorphism if $\mathsf{A}_\mathsf{P}$ has rank one. A quasicharacter on $\mathsf{A}_\mathsf{P}$ is a homomorphism $\mathsf{A}_\mathsf{P}\to \C^\times$, and the space of quasicharacters coincides with $\mathfrak{a}^*_{P,\C}$ under the logarithm. For a quasicharacter $\chi$ we define $\mathrm{ht}_\mathsf{P}(\chi)$ in this way. We let $\C_\chi$ denote the $\mathsf{A}_\mathsf{P}$-representation $\C$ with the group acting via $\chi$.

\begin{lemma}
\label{onlypol}
Let $\mathsf{G}$ be a semisimple Lie group, $\mathsf{P}=\mathsf{M}_\mathsf{P}\mathsf{A}_\mathsf{P}\mathsf{N}_\mathsf{P}$ a standard parabolic subgroup and $\chi_1$ and $\chi_2$ two quasi-characters of $\mathsf{A}_\mathsf{P}$. Then 
$$(\Sigma^m_H\overline{\mathsf{N}}_\mathsf{P}\otimes \Hom(\C_{\chi_1},\C_{\chi_2}))^{\mathsf{A}_\mathsf{P}}=
\begin{cases}
(\mathcal{U}_m(\overline{\mathfrak{n}}_\mathsf{P})\otimes\Hom(\C_{\chi_1},\C_{\chi_2}))^{\mathsf{A}_\mathsf{P}}, \quad&\mbox{if $\mathrm{rk}(\mathsf{A}_\mathsf{P})>1$ or}\\
& m\neq \mathrm{ht}_\mathsf{P}(\chi_1)-\mathrm{ht}_\mathsf{P}(\chi_2),\\
\Sigma^m_H\overline{\mathsf{N}}_\mathsf{P}, \quad&\mbox{if $\mathrm{rk}(\mathsf{A}_\mathsf{P})=1$ and}\\
&m= \mathrm{ht}_\mathsf{P}(\chi_1)-\mathrm{ht}_\mathsf{P}(\chi_2).
\end{cases}$$

\end{lemma}

The content of this lemma is that either the $\mathsf{A}_\mathsf{P}$-invariant part consists of invariant elements of the universal enveloping algebra or it is the whole symbol algebra. To prove Lemma \ref{onlypol}, we will need a proposition. We note that if $\R^r_+$ acts faithfully and semisimply on $\R^N$, there exists invertible, diagonalizable, commuting matrices $\mathfrak{d}_1,\ldots, \mathfrak{d}_r $ under which $\R^r_+$ acts as $\lambda=(\lambda_1,\ldots, \lambda_r)\mapsto \lambda_1^{\mathfrak{d}_1}\cdots, \lambda_r^{\mathfrak{d}_r}$. 

\begin{lemma}
\label{homopoly}
Let $N>0$ and $r>1$ be integers, and let $\chi:\R^r_+\to \C^\times$ be a quasicharacter. Assume that the $\R^r_+$-action on $\R^N$ is faithful, semisimple with no fixed points outside $x=0$, i.e. $\lambda.x=x$ for all $\lambda$ if and only if $x=0$. Then $(C^\infty(\R^N\setminus \{0\})\otimes \C_\chi)^{\R^r_+}$ consists of polynomials whose degree with respect to the $\R^r_+$-action is determined by $\chi$. In particular, an element of $(C^\infty(\R^N\setminus \{0\})\otimes \C_\chi)^{\R^r_+}$ extends uniquely to an element of $(C^\infty(\R^N)\otimes \C_\chi)^{\R^r_+}$.
\end{lemma}

\begin{proof}
By an argument with Taylor series at $0$ as in the proof of \cite[Lemma 3.8]{Dave_Haller1pre}, $(C^\infty(\R^N)\otimes \C_\chi)^{\R^r_+}$ consists of polynomials so it suffices to prove that an element of $(C^\infty(\R^N\setminus \{0\})\otimes \C_\chi)^{\R^r_+}$ extends uniquely to an element of $(C^\infty(\R^N)\otimes \C_\chi)^{\R^r_+}$. In particular, for a quasicharacter $\chi_0$ of $\R_+$, $(C^\infty(\R^n)\otimes \C_{\chi_0})^{\R_+}$ consists of polynomials.

The proof relies on the fact of the preceding paragraph and consists of a backward induction argument reducing to $r=2$ where again the fact of the preceding paragraph kicks in. Semi-simplicity ensures that after choosing an appropriate basis of $\R^N$ there are diagonal matrices $D_1,\ldots, D_r$ such that the $\R_+^r$-action takes the form 
$$(\lambda_1,\ldots,\lambda_r).x=\lambda_1^{D_1}\cdots \lambda_r^{D_r}x.$$
Since the $\R^r_+$-action is faithful with no fixed points outside $x=0$, there is no joint kernel of $D_1,\ldots, D_r$. Up to an automorphism of $\R^r_+$ and a base change, we can assume that for some $0<n_1\leq N$, $D_1$ is invertible on $\R^{n_1}\times \{0\}\subseteq \R^N$ yet $D_2,\ldots, D_r$ act trivially on the subspace $\R^{n_1}\times \{0\}\subseteq \R^N$. Similarly, up to an automorphism of $\R^{r-1}_+$ and a base change, we can assume that for some $0<n_2\leq N-n_1$, $D_2$ is invertible on $\R^{n_2}\times \{0\}\subseteq \R^{N-n_1}$ yet $D_3,\ldots, D_r$ act trivially on the subspace $\R^{n_2}\times \{0\}\subseteq \R^{N-n_1}$. By induction, we can find an integer partition $N=n_1+n_2+\ldots +n_r$ such that up to an automorphism $\R^r_+$ acts on $x=(x_1,x_2,\ldots, x_r)\in \oplus_{j=1}^r \R^{n_j}=\R^N$ as
\begin{align*}
(\lambda_1,&\ldots,\lambda_r).(x_1,x_2,\ldots, x_r):=\\
&(\lambda_1.x_1,\alpha_{1,2}(\lambda_1) \lambda_2.x_2,\alpha_{1,3}(\lambda_1)\alpha_{2,3}(\lambda_2) \lambda_3.x_3,\ldots, \alpha_{1,r}(\lambda_1)\alpha_{2,r}(\lambda_2)\cdots \alpha_{r-1,r}(\lambda_{r-1})\lambda_r.x_r),
\end{align*}
for some $\R_+$-actions $(\lambda,x_j)\mapsto \lambda.x_j$ and $(\alpha_{i,j})_{i=1}^{j-1}$ on $\R^{n_j}$, $j=1,\ldots, r$, with the $\R_+$-actions $(\lambda,x_j)\mapsto \lambda.x_j$ being faithful without fixed points outside $x_j=0$. Indeed, this decomposition can be proven by induction on $r$.

We write $\chi(\lambda)=\lambda_1^{m_1}\cdots \lambda_r^{m_r}$ for some $(m_1,\ldots,m_r)\in \C^r$ and $\chi_r(\lambda_r)=\lambda_r^{m_r}$. Set $N':=N-n_r$. Similarly, write $x\in \R^N$ as $x=(x',x_r)\in  \R^{N'}\oplus \R^{n_r}$. Take $f\in (C^\infty(\R^N\setminus \{0\})\otimes \C_\chi)^{\R^r_+}$. Then for any $x'\neq 0$, the function $x_r\mapsto f(x',x_r)$ is smooth and belongs to $(C^\infty(\R^{n_r})\otimes \C_{\chi_r})^{\R_+}$. The first statement in this proof allows us to conclude that there exists a finite collection of monomials $p_1,\ldots, p_M$ in $x_r$ such that for some characters 
 $\chi_j'$ of $\R^{r-1}$, $j=1,\ldots, M$, there are elements $a_j\in (C^\infty(\R^{N'}\setminus \{0\})\otimes \C_{\chi'_j})^{\R^{r-1}_+}$ for $j=1,\ldots, M$, with 
$$f(x',x_r)=\sum_{l=1}^M p_l(x_r)a_l(x').$$ 
By backward induction, we can assume that $r=2$. 

For $r=2$, we have $\chi(\lambda_1,\lambda_2)=\chi_1(\lambda_1)\chi_2(\lambda_2)=\lambda_1^{m_1} \lambda_2^{m_2}$. The argument in the preceding paragraph implies by symmetry that we can write 
$$f(x_1,x_2)=\sum_{l=1}^M p_l(x_1)a_l(x_2)=\sum_{l=1}^M b_l(x_1)q_l(x_2),$$
where $a_1,\ldots, a_M\in (C^\infty(\R^{n_2}\setminus \{0\})\otimes \C_{\chi_2})^{\R_+}$ and polynomials $p_1,\ldots, p_M$ in $x_1$, and $b_1,\ldots, b_M\in (C^\infty(\R^{n_1}\setminus \{0\})\otimes \C_{\chi_1})^{\R_+}$ and polynomials $q_1,\ldots, q_M$ in $x_2$. It follows that $f$ is a polynomial in $(x_1,x_2)$ and the proof is complete.
\end{proof}

\begin{proof}[Proof of Lemma \ref{onlypol}]
We first note that the dilation on $\overline{\mathsf{N}}_\mathsf{P}$ factors over the action of $\mathsf{A}_\mathsf{P}$ on $\overline{\mathsf{N}}_\mathsf{P}$ defined from conjugation. Let $\mathsf{A}_\mathsf{P}^0\subseteq \mathsf{A}_\mathsf{P}$ be a complement to the corresponding embedding $\R_+\hookrightarrow \mathsf{A}_\mathsf{P}$. And in fact, the basis vector $1\in \Hom(\C_{\chi_1},\C_{\chi_2})$ implements an isomorphism
$$(\Sigma^m_H\overline{\mathsf{N}}_\mathsf{P}\otimes \Hom(\C_{\chi_1},\C_{\chi_2}))^{\mathsf{A}_\mathsf{P}}\cong 
\begin{cases}
(\Sigma^{m}_H\overline{\mathsf{N}}_\mathsf{P}\otimes \C_\chi)^{\mathsf{A}_\mathsf{P}^0}, \quad &m= \mathrm{ht}_\mathsf{P}(\chi_1)-\mathrm{ht}_\mathsf{P}(\chi_2)\\
0,\quad&\mbox{otherwise.}\end{cases}$$
where $\chi$ is the restriction of $\chi_2\chi_1^{-1}$ to $\mathsf{A}_\mathsf{P}^0$.  We see from this identity that the case $\mathrm{rk}(\mathsf{A}_\mathsf{P})=1$ holds true. To simplify the remainder of the argument, we can without loss of generality assume that $\chi_1$ is trivial and that $m=-\mathrm{ht}_\mathsf{P}(\chi_2)$.

It remains to show that $(\Sigma^m_H\overline{\mathsf{N}}_\mathsf{P}\otimes \C_\chi)^{\mathsf{A}_\mathsf{P}^0}=(\mathcal{U}_m(\overline{\mathfrak{n}}_\mathsf{P})\otimes \C_\chi)^{\mathsf{A}_\mathsf{P}^0}$ when $\mathrm{rk}(\mathsf{A}_\mathsf{P})>1$. In fact, it is sufficient to show that $(\Sigma^m_H\overline{\mathsf{N}}_\mathsf{P}\otimes \C_\chi)^{\mathsf{A}_\mathsf{P}^0}$ is in the kernel of the restriction mapping $\Sigma^m_H\overline{\mathsf{N}}_\mathsf{P}\otimes \C_\chi\to C^\infty(\overline{\mathsf{N}}_\mathsf{P}\setminus \{0\})/\mathcal{P}^m(\overline{\mathsf{N}}_\mathsf{P})$. On both sides we use the $\mathsf{A}_\mathsf{P}$-action $\delta^\mathsf{P}_a$ defined by $\delta^\mathsf{P}_ak(n):=k(ana^{-1})$ defined for $a\in \mathsf{A}_\mathsf{P}$ and $k$ in any of the two spaces. The $A_{\mathsf{P}}$-action on $\overline{\mathsf{N}}_\mathsf{P}$ is clearly semisimple. 

We shall from the definition prove that the $A_{\mathsf{P}}$-action on $\overline{\mathsf{N}}_\mathsf{P}$ has no fixed points outside $x=0$ and is faithful, for notations see Example \ref{kljnkjandaretrot}. For all points and purposes in the argument, we can view $\overline{\mathsf{N}}_\mathsf{P}$ as the linear space $\overline{\mathfrak{n}}_\mathsf{P}$. Assume that  $\mathsf{P}$ is defined from a set of simple roots $\Pi'\subseteq \Pi$ via $\mathfrak{p}=\mathfrak{p}_{\rm min}\oplus_{\lambda\in \mathrm{span}(\Pi')} \mathfrak{g}_\lambda$. Then $\overline{\mathfrak{n}}_\mathsf{P}=\oplus_{\lambda\in -\N_+(\Pi\setminus \Pi')} \mathfrak{g}_\lambda$ while $\mathsf{A}_{\mathsf{P}}$ is defined from its Lie algebra $\mathfrak{a}_{\mathrm{P}}·=\cap_{\beta\in \mathrm{span}(\Pi')}\ker\beta\subseteq \mathfrak{a}$. If $\mathsf{A}_{\mathsf{P}}$ were to act non-faithfully or have fixed points outside $x=0$ on $\overline{\mathfrak{n}}_\mathsf{P}$ then there would exist a non-zero element $H\in \mathfrak{a}_{\mathrm{P}}$ with $\lambda(H)=0$ for all $\lambda\in \Pi\setminus \Pi'$ contradicting that fact that $0=\cap_{\beta\in \mathrm{span}(\Pi)}\ker(\beta)=\cap_{\beta\in \mathrm{span}(\Pi')}\ker(\beta)\cap \mathfrak{a}_{\mathrm{P}}$.

 The range of $(\Sigma^m_H\overline{\mathsf{N}}_\mathsf{P}\otimes \C_\chi)^{\mathsf{A}_\mathsf{P}^0}$ under the restriction mapping $\Sigma^m_H\overline{\mathsf{N}}_\mathsf{P}\otimes \C_\chi\to C^\infty(\overline{\mathsf{N}}_\mathsf{P}\setminus \{0\})/\mathcal{P}^m(\overline{\mathsf{N}}_\mathsf{P})$ is contained in the set of elements $k\in C^\infty(\overline{\mathsf{N}}_\mathsf{P}\setminus \{0\})/\mathcal{P}^m(\overline{\mathsf{N}}_\mathsf{P})$ such that $\delta^\mathsf{P}_ak=\chi_2(a)k$. However, since $ r:=\mathrm{rk}(\mathsf{A}_\mathsf{P})>1$ the only element $k\in C^\infty(\overline{\mathsf{N}}_\mathsf{P}\setminus \{0\})/\mathcal{P}^m(\overline{\mathsf{N}}_\mathsf{P})$ such that $\delta^\mathsf{P}_ak=\chi_2(a)k$ is the zero element by Lemma \ref{homopoly} applied to $\mathsf{A}_{\mathsf{P}}\cong \R^r_+$. And the lemma follows.
\end{proof}

For later purposes, we work with filtered homogeneous bundles. We say that a finite-dimensional $\mathsf{A}_\mathsf{P}\mathsf{M}_\mathsf{P}$-representation $W$ is graded if $W=\oplus_{\nu\in \R} W[\nu]$ for finite-dimensional $\mathsf{A}_\mathsf{P}\mathsf{M}_\mathsf{P}$-representations $W[\nu]$. If so, the associated  homogeneous bundle $\pmb{E}=\mathsf{G}\times_\mathsf{P} W\to \mathfrak{X}$ is filtered by $\pmb{E}^p=\mathsf{G}\times_\mathsf{P} (\oplus_{\nu\geq p} W[\nu])$.

\begin{cor}
\label{jnljnjknkjnad}
Consider the flat parabolic geometry $\mathfrak{X}=\mathsf{G}/\mathsf{P}$. For finite-dimensional graded $\mathsf{A}_\mathsf{P}\mathsf{M}_\mathsf{P}$-representations $W_1$ and $W_2$ we define the homogeneous vector bundles $\pmb{E}_j=\mathsf{G}\times_\mathsf{P} W_j\to \mathfrak{X}$. Then the invariant algebra $\Sigma_H^m(\mathfrak{X}; \pmb{E}_1,\pmb{E}_2)^\mathsf{G}$ consists of symbols of differential operators if $\mathrm{rk}(\mathsf{A}_\mathsf{P})>1$ and if $\mathrm{rk}(\mathsf{A}_\mathsf{P})=1$ we have that 
$$\Sigma_{ H,{\rm gr}}^m(\mathfrak{X}; \pmb{E}_1,\pmb{E}_2)^\mathsf{G}\cong \oplus_k (\Sigma^{m-k}_H\overline{N}_P\otimes \Hom_{\mathsf{A}_\mathsf{P}}(W_1^{(k)},W_2))^{\mathsf{M}_\mathsf{P}},$$
where $W_1^{(k)}=W_1\otimes \C_{\chi_k}$ for the quasicharacter $\chi_k:=\mathrm{ht}_P^{-1}(m)$. 
\end{cor}

\section{BGG-sequences and graded Rockland sequences} 
\label{subsec:bggrock}

Throughout this section, we fix a $G$-Carnot manifold $\mathfrak{X}$. For simplicity, we assume that $\mathfrak{X}$ is compact and remark that local statements will extend to noncompact $\mathfrak{X}$. 

\begin{definition}
\label{gradedrock}
Consider a collection $\pmb{E}_\bullet=(\pmb{E}_0,\pmb{E}_1,\ldots, \pmb{E}_N)$ of filtered vector bundles $\pmb{E}_j\to \mathfrak{X}$ and numbers $\pmb{m}=(m_1,\ldots, m_N)\in \R^N$. We let
\begin{equation}
\label{lkankland}
D_\bullet: \quad 0\to C^\infty(\mathfrak{X};\pmb{E}_0)\xrightarrow{D_1}C^\infty(\mathfrak{X};\pmb{E}_1)\xrightarrow{D_2}\cdots \xrightarrow{D_N}C^\infty(\mathfrak{X};\pmb{E}_N)\to 0,
\end{equation}
be a sequence of maps $D_j\in \Psi^{m_j}_{H,{\rm gr}}(\mathfrak{X}; \pmb{E}_{j-1},\pmb{E}_j)$. We say that the sequence $D_\bullet$ in Equation \eqref{lkankland} is a \emph{graded Rockland sequence} if the symbol sequence $\sigma_H(D_\bullet)$ defined by
\begin{align*}
0\to \mathcal{S}_0(T_H\mathfrak{X};\mathrm{gr}(\pmb{E}_0))&\xrightarrow{\sigma_{H,{\rm gr}}(D_1)}\mathcal{S}_0(T_H\mathfrak{X};\mathrm{gr}(\pmb{E}_1))\xrightarrow{\sigma_{H,{\rm gr}}(D_2)}\cdots \\
&\xrightarrow{\sigma_{H,{\rm gr}}(D_{N-1})}\mathcal{S}_0(T_H\mathfrak{X};\mathrm{gr}(\pmb{E}_{N-1}))\xrightarrow{\sigma_{H,{\rm gr}}(D_N)}\mathcal{S}_0(T_H\mathfrak{X};\mathrm{gr}(\pmb{E}_N))\to 0,
\end{align*}
localizing in any non-trivial unitary representation of the osculating Lie groupoid $T_H\mathfrak{X}$ to an exact sequence. We say that $\pmb{m}$ is the order of $D_\bullet$.
\begin{itemize}
\item In the special case that a Rockland sequence $D_\bullet$ as in Equation \eqref{lkankland} is a complex, we call $D_\bullet$ a \emph{graded Rockland complex}.
\item In the special case that $\pmb{E}_0,\pmb{E}_1,\ldots, \pmb{E}_N\to \mathfrak{X}$ are all trivially filtered, a Rockland sequence $D_\bullet$ as in Equation \eqref{lkankland} is called a \emph{Rockland sequence}.
\end{itemize}
\end{definition}

\begin{remark}
Upon shifting the filtering of each $\pmb{E}_j$ in a graded Rockland sequence, we can achieve that $D_\bullet$ has order $\pmb{0}=(0,\ldots, 0)\in \R^N$.
\end{remark}

\subsection{BGG-sequences}
\label{susubbgg}

In the paper \cite{bggoriginal}, Bernstein-Gelfand-Gelfand studied the resolution of a highest weight module of a complex, semisimple Lie algebra by means of Verma modules. This led Bernstein-Gelfand-Gelfand to the category $\mathcal{O}$, for an introduction thereto see \cite{humphreys}. The construction of Bernstein-Gelfand-Gelfand was later extended to generalized Verma modules by Lepowsky \cite{lepog}. The construction is by now understood in larger generality. We restrict to the setting of finite-dimensional representations of complex, semisimple Lie groups using algebraic tools in this section and return to the  general case below in Subsection \ref{lknaldkna}.

We start by discussing Bernstein-Gelfand-Gelfand's algebraic setup whereafter we bring it back to the geometric interpretation of relevance to us. Let $\mathsf{G}$ be a connected, complex, semisimple Lie group with Lie algebra $\mathfrak{g}$. We fix a maximal compact subgroup $\mathsf{K}\subseteq \mathsf{G}$ and a Cartan subalgebra $\mathfrak{h}\subseteq \mathfrak{g}$ integrating to $\mathsf{H}\subseteq \mathsf{G}$. Set $\mathsf{M}:=\mathsf{K}\cap \mathsf{H}$. We can factorize $\mathsf{G}=\mathsf{K}\mathsf{A}\mathsf{N}_+$ where $\mathsf{A}$ comes from the factorization $\mathsf{H}=\mathsf{M}\mathsf{A}$ and $\mathsf{N}_+$ is a (complex) nilpotent Lie group. Here we use notation mimicking that of parabolic geometry from Subsection \ref{jnkajnakjdnapara}. The Borel subgroup is defined by 
$$\mathsf{B}:=\mathsf{M}\mathsf{A}\mathsf{N}_+=\mathsf{H}\mathsf{N}_+,$$
and is solvable with nilpotent radical $[\mathsf{B},\mathsf{B}]=\mathsf{N}_+$. Let $\mathfrak{b}$ denote the Lie algebra of $\mathsf{B}$ and $\mathfrak{n}_+$ the Lie algebra of $\mathsf{N}_+$. In terms of the roots $\Delta\subseteq \mathfrak{h}^*$, and a suitable choice of positive root system $\Delta_+$, the description above stems from the root decompositions
$$\mathfrak{g}=\mathfrak{h}+\sum_{\alpha\in \Delta} \mathfrak{g}_\alpha, \quad\mathfrak{b}=\mathfrak{h}+\sum_{\alpha\in \Delta_+} \mathfrak{g}_{\alpha}\quad \mbox{and}\quad \mathfrak{n}_+=\sum_{\alpha\in \Delta_+} \mathfrak{g}_{\alpha},$$
where $\mathfrak{g}_\alpha:=\{Y\in \mathfrak{g}: [h,Y]=\alpha(h)Y\, \forall h\in \mathfrak{h}\}$. We shall also make use of the nilpotent Lie algebra 
$$ \mathfrak{n}_-=\sum_{\alpha\in \Delta_+} \mathfrak{g}_{-\alpha}.$$
In terms of a $|k|$-grading as in Subsection \ref{jnkajnakjdnapara}, the above description is related via a height of roots as in Example \ref{kljnkjandaretrot}. For complex Lie groups, the set of restricted roots $\Sigma$ coincides with the set of all roots $\Delta$. The height with respect to the Borel group is $\mathrm{ht}_\mathsf{B}(\sum_{\alpha\in \Pi} a_\alpha \alpha):=\sum_{\alpha\in \Pi} a_\alpha$ for a choice of simple roots $\Pi$. The Iwasawa decomposition described above can as in Subsection \ref{jnkajnakjdnapara} be constructed from the $|k|$-grading 
$$\mathfrak{g}_j:=\bigoplus_{\alpha\in \Delta: \mathrm{ht}_B(\alpha)=j} \mathfrak{g}_\alpha.$$

The Verma module associated with a weight $\lambda\in \mathfrak{h}^*$ is the $\mathfrak{g}$-module defined by 
$$\mathcal{M}(\lambda):=\mathcal{U}(\mathfrak{g})\otimes_{\mathcal{U}(\mathfrak{b})} \C_\lambda.$$
Here $\mathcal{U}(\mathfrak{g})$ denotes the universal enveloping algebra and $\C_\lambda$ is the $\mathfrak{b}$-module defined from the character $\lambda$ extended to $\mathfrak{b}$ via the quotient map $\mathfrak{b}\to \mathfrak{h}$. By \cite[Proposition 7.1.8]{dixmier}, the Verma module is cofinal among all $\mathfrak{g}$-modules with highest weight $\lambda$. The Poincaré-Birkhoff-Witt theorem induces a vector space isomorphism
\begin{equation}
\label{pbwiso}
\mathcal{M}(\lambda)\cong  \mathcal{U}(\mathfrak{n}_-).
\end{equation}
The idea underlying the construction of the BGG-resolution is the fact that there are very few morphisms between Verma modules. Indeed, to quantify what we mean by ``very few'' we recall the following result due to Bernstein-Gelfand-Gelfand \cite[page 41]{bggoriginal}, see also \cite{dixmier}. First, we recall the standard notation 
$$\rho:=\frac{1}{2}\sum_{\alpha\in \Delta_+}\alpha.$$
We also let $\mathcal{W}$ denote the Weyl group and equip it with its Bruhat order (cf. \cite[Chapter 3.2.14]{capslovak}).

\begin{thm}
Let $\lambda$ and $\mu$ be weights. Then there is a non-zero morphism of $\mathfrak{g}$-modules
$$\mathcal{M}(\mu)\leftarrow \mathcal{M}(\lambda),$$
if and only if $\lambda$ is an integral weight such that $\lambda+\rho$ and $\mu+\rho$ are in the same orbit of the Weyl group, and more precisely for $w_1\geq w_2\in W$ the weight $w_1^{-1}(\lambda+\rho)=w_2^{-1}(\mu+\rho)$ is the unique dominant weight in the orbit of $\lambda+\rho$. If so, the morphism is uniquely determined up to a complex scalar and injective.
\end{thm}

Bernstein-Gelfand-Gelfand \cite{bggoriginal} proved that when assembling these morphisms together according to lengths in the Weyl group, they fit into a resolution. Note that  if $W$ is a $\mathfrak{g}$-module with highest weight $\lambda$, there is a natural map $\mathcal{M}(\lambda)\to W$.

\begin{thm}[The Bernstein-Gelfand-Gelfand resolution]
\label{inadonad9}
Let $W$ be a finite-dimensional $\mathfrak{g}$-module with highest weight $\lambda$. Write $d$ for the longest length in the Weyl group $\mathcal{W}$. Define the $\mathfrak{g}$-modules
$$C_k:=\bigoplus_{w\in \mathcal{W}, \ |w|=k}\mathcal{M}(w(\lambda+\rho)-\rho), \quad k=0, \ldots, d.$$
Then there are elements $\delta_k\in \Hom_{\mathcal{U}(\mathfrak{g})}(C_k, C_{k-1})$ making the following sequence exact
$$0\leftarrow W\leftarrow C_0\xleftarrow{\delta_1}C_1\xleftarrow{\delta_2}\cdots \xleftarrow{\delta_{d-1}}C_{d-1}\xleftarrow{\delta_d}C_d\leftarrow 0.$$

\end{thm}

To compute with the BGG-complex, one uses bundles of Lie algebra homology $\mathpzc{H}_j(W)\to \mathfrak{X}=\mathsf{G}/\mathsf{P}$, defined below in \eqref{lknalkdnaodjn}, that in turn come from the $\mathsf{G}_0$-module structure on $H_j(\mathfrak{p}_+,W)$. For complex $\mathsf{G}$ and $\mathsf{P}$ its Borel subgroup, the $\mathsf{G}_0$-module structure on $H_j(\mathfrak{p}_+,W)$ was described by Kostant \cite{kostant61} and the general case can be found in Cap-Slovak's monograph \cite[Theorem 3.3.5 and Proposition 3.3.6]{capslovak}. We recall it in a form relevant to this work. 

\begin{thm}
\label{knlkanlknad}
Let $\mathfrak{g}$ be a $|k|$-graded semisimple Lie algebra and $W$ a complex, finite-dimensional highest weight representation of $\mathfrak{g}$. Write $\lambda$ for the highest weight of $W$, $\delta$ for the lowest form on $\mathfrak{g}$ and $\mathcal{W}^{\mathfrak{p}}\subseteq \mathcal{W}$ for the Hasse diagram of $\mathfrak{p}$ in the Weyl group $\mathcal{W}$ of $\mathfrak{g}$ (see \cite[Chapter 3.2.15]{capslovak}).
\begin{itemize}
\item Assume that $\mathfrak{g}$ is complex and the representation $W$ is complex linear. For a $\mathfrak{g}_0$-dominant weight $\nu$ the $G_0$-isotypical component $H^k(\mathfrak{p}_+,W)^\nu$ vanishes unless $\nu=w(\lambda+\delta)-\delta$ for a $w\in W^{\mathfrak{p}}$ with $\ell(w)=k$, in which case $H^{\ell(w)}(\mathfrak{p}_+,W)^{w(\lambda+\delta)-\delta}$ is irreducible. In particular, as $\mathsf{G}_0$-modules
$$H^k(\mathfrak{p}_+,W)=\bigoplus_{w\in \mathcal{W}^{\mathfrak{p}}, \ \ell(w)=k} \mathpzc{V}_{w(\lambda+\delta)-\delta},$$
where $\mathpzc{V}_{w(\lambda+\delta)-\delta}$ are irreducible finite-dimensional $\mathsf{G}_0$-representations with highest weight $w(\lambda+\delta)-\delta$. 
\item Assume that $\mathfrak{g}$ is real. Then 
$$H^k(\mathfrak{p}_+,W)=H^k(\mathfrak{p}_{+,\C},W),$$
as $\mathsf{G}_0$-modules where the right hand side is computed from the complex $|k|$-graded semisimple Lie algebra $\mathfrak{g}_\C=\mathfrak{g}\otimes_\R\C$ as in the previous item.
\end{itemize}
\end{thm}

The algebraic construction in Theorem \ref{inadonad9} has a geometric counterpart on the full flag manifold $\mathfrak{X}:=\mathsf{G}/\mathsf{B}$. Making the transition to geometry, we use the following correspondence. 

\begin{prop}
\label{lknjknjng}
Let $W$ be a finite-dimensional representation of the Borel group $\mathsf{B}$ and let $E(W):=\mathsf{G}\times_\mathsf{B} W$ denote the corresponding homogeneous vector bundle on $\mathfrak{X}:=\mathsf{G}/\mathsf{B}$. Then we have a $\mathsf{G}$-equivariant isomorphism 
\begin{align}
\label{lknalknad}
\Hom_{\mathcal{U}(\mathfrak{g})}(\mathcal{U}(\mathfrak{g})&\otimes_{\mathcal{U}(\mathfrak{b})} W,C^\infty(\mathsf{G}))\xrightarrow{\sim} C^\infty(\mathfrak{X},E(W)^*), \\
\nonumber
 &\varphi\mapsto [(g,v)\mapsto (g^{-1})^*(\varphi(1\otimes v))],
\end{align}
where we equip $C^\infty(\mathsf{G})$ with the $\mathcal{U}(\mathfrak{g})$-module structure from the right action and we equip $\Hom_{\mathcal{U}(\mathfrak{g})}(\mathcal{U}(\mathfrak{g})\otimes_{\mathcal{U}(\mathfrak{b})} W,C^\infty(\mathsf{G}))$ with the $\mathsf{G}$-action defined from left translation on $\mathsf{G}$. 

Moreover, the isomorphism \eqref{lknalknad} is contravariantly functorial for $\mathcal{U}(\mathfrak{g})$-module maps in the sense that if $W_1$ and $W_2$ are two $\mathsf{B}$-representations and 
$$\delta\in \Hom_{\mathcal{U}(\mathfrak{g})}(\mathcal{U}(\mathfrak{g})\otimes_{\mathcal{U}(\mathfrak{b})} W_1,\mathcal{U}(\mathfrak{g})\otimes_{\mathcal{U}(\mathfrak{b})} W_2),$$
there is a $\mathsf{G}$-equivariant differential operator $D\in \mathcal{DO}(\mathfrak{X};E(W_2)^*,E(W_1)^*)$ making the following diagram commute:
$$\begin{CD}
\Hom_{\mathcal{U}(\mathfrak{g})}(\mathcal{U}(\mathfrak{g})\otimes_{\mathcal{U}(\mathfrak{b})} W_2,C^\infty(\mathsf{G}))@>\delta^* >>\Hom_{\mathcal{U}(\mathfrak{g})}(\mathcal{U}(\mathfrak{g})\otimes_{\mathcal{U}(\mathfrak{b})} W_1,C^\infty(\mathsf{G})) \\
@V\cong VV  @VV\cong V \\
C^\infty(\mathfrak{X},E(W_2)^*)@>D>>C^\infty(\mathfrak{X},E(W_1)^*)
\end{CD},$$
where the vertical isomorphism is that of \eqref{lknalknad}. The differential operator is uniquely determined from the pairing with jets at $eB\in \mathfrak{X}$.
\end{prop}

\begin{proof}
Equation \eqref{lknalknad} follows from Frobenius reciprocity. The second part of the statement follows from an analysis of how $\mathfrak{g}$-module maps $\mathcal{U}(\mathfrak{g})\otimes_{\mathcal{U}(\mathfrak{b})} W_1\to \mathcal{U}(\mathfrak{g})\otimes_{\mathcal{U}(\mathfrak{b})} W_2$ pair with jets at $e\mathsf{B}\in \mathfrak{X}$, for more details, see \cite{yunckenthesis}.
\end{proof}

In light of Proposition \ref{lknjknjng}, one has the following geometric reformulation of Theorem \ref{inadonad9}. Note that the longest element of the Weyl group has length $d$, where $d=\dim_\C(\mathsf{G}/\mathsf{B})$ (see \cite[Chapter 3]{capslovak}). 

\begin{thm}[The geometric BGG-sequence]
\label{inadonad9geo}
Let $W$ be a finite-dimensional representation of the complex, semisimple Lie group $\mathsf{G}$ with a highest weight $\lambda$. Define the homogeneous vector bundles on $\mathfrak{X}:=\mathsf{G}/\mathsf{B}$ from
$$E_k^W:=\bigoplus_{w\in \mathcal{W}, \ |w|=k}\mathsf{G}\times_\mathsf{B} \C_{-w(\lambda+\rho)+\rho}, \quad k=0, \ldots, d.$$
Then there are $G$-equivariant differential operators $D_k\in  \mathcal{DO}(\mathfrak{X};E_{k-1}^W,E_k^W)$, unique up to a finite set of choices, making the following sequence into a complex of differential operators
$$0\to C^\infty(\mathfrak{X};E_0^W)\xrightarrow{D_1}C^\infty(\mathfrak{X};E_1^W)\xrightarrow{D_2}\cdots \xrightarrow{D_{d-1}}C^\infty(\mathfrak{X};E_{d-1}^W)\xrightarrow{D_d}C^\infty(\mathfrak{X};E_d^W)\to 0,$$
which is exact except in degree $0$ where we have $\ker(D_1)=W$.
\end{thm}

The reader can find this result in \cite[Theorem 8.3.1]{bastoneast}, see also the discussion in \cite[Section 3.6]{yunckenthesis}.

\begin{definition}
The BGG-complex of a finite-dimensional representation $W$ of a complex, semisimple Lie group $\mathsf{G}$ is the differential complex on $\mathsf{G}/\mathsf{B}$ from Theorem \ref{inadonad9geo}. We denote it by $\mathsf{BGG}_\bullet(W)$.
\end{definition}

The BGG-complex of a finite-dimensional highest weight representation is a prototypical example of a graded Rockland sequence, as we recall below in Subsection \ref{lknaldkna} from the work of Dave-Haller \cite{Dave_Haller1}. The Carnot structure on $\mathfrak{X}=\mathsf{G}/\mathsf{B}$ in which the BGG-complex is graded Rockland, is defined in terms of the simple root system used to define the order in the Weyl group and the height of roots. Let $\Pi\subseteq \Delta_+$ denote this simple root system. 

We note that $T\mathfrak{X}$ is the homogeneous vector bundle 
\begin{equation}
\label{lknlada}
T\mathfrak{X}=\mathsf{G}\times_\mathsf{B} (\mathfrak{g}/\mathfrak{b})\cong \mathsf{G}\times_\mathsf{B} \mathfrak{n}_-,
\end{equation}
where $\mathsf{B}$ acts on $\mathfrak{n}_-=\sum_{\alpha\in\Delta_+}\mathfrak{g}_{-\alpha}$ via the isomorphism $\mathfrak{n}_-\cong \mathfrak{g}/\mathfrak{b}$. The nilpotent Lie algebra $\mathfrak{n}_-$ is graded by the height of roots. In particular, $\mathfrak{n}_-^{-1}=\sum_{\alpha\in\Pi}\mathfrak{g}_{-\alpha}$ and we inductively have that 
$$\mathfrak{n}_{-}^{-k-1}:=\mathfrak{n}_{-}^{-k}+[\mathfrak{n}_{-}^{-1},\mathfrak{n}_{-}^{-k}].$$
The action of $\mathsf{B}$ on $\mathfrak{n}_-$ respects the filtering (as it increases the order), and 
$$\mathsf{G}\times_\mathsf{B} \mathfrak{n}_-^{-j}\to \mathfrak{X}=\mathsf{G}/\mathsf{B},$$ 
is a well defined bundle that we identify with a sub-bundle $T^{-j}\mathfrak{X}\subseteq T\mathfrak{X}$. By equivariance, the subbundles $T^{-j}\mathfrak{X}\subseteq T\mathfrak{X}$ form a complex Carnot structure on $\mathfrak{X}$. We note that, when equipping $\mathfrak{n}_-$ with the $\mathsf{B}$-action defined from the quotient map $\mathsf{B}\to \mathsf{H}$ and the natural $\mathsf{H}$-action on $\mathfrak{n}_-$, we can identify 
\begin{equation}
\label{bladiffb}
\mathfrak{t}_H\mathfrak{X}=\mathsf{G}\times_\mathsf{B} \mathfrak{n}_-,
\end{equation}
as $\mathsf{G}$-equivariant bundles of graded, nilpotent Lie algebras. We emphasize that the $\mathsf{B}$-action on $\mathfrak{n}_-$ in Equation \eqref{lknlada} defining $T\mathfrak{X}$ differs from the $\mathsf{B}$-action on $\mathfrak{n}_-$ in Equation \eqref{bladiffb} defining $\mathfrak{t}_H\mathfrak{X}$, as such $\mathfrak{t}_H\mathfrak{X}\cong T\mathfrak{X}$ but not $\mathsf{G}$-equivariantly. In particular, $\mathfrak{X}=\mathsf{G}/\mathsf{B}$ is regular of type $\mathfrak{n}_-$. We note that we can reduce the structure group to $\mathsf{H}$ via the quotient mapping $\mathsf{B}\to \mathsf{B}/\mathsf{N}_+=\mathsf{H}$ and the graded action 
$$\mathsf{H}\to \Aut_{\rm gr}(\mathfrak{n}_-).$$
Therefore, we have $\mathsf{G}$-equivariant isomorphisms of bundles of graded, nilpotent Lie algebras
$$\mathfrak{t}_H\mathfrak{X}=\mathsf{G}\times_\mathsf{B} \mathfrak{n}_-\cong (\mathsf{G}/\mathsf{N}_+)\times_\mathsf{H} \mathfrak{n}_-,$$
for the $\mathsf{B}$-action on $\mathfrak{n}_-$ defined from the quotient map $\mathsf{B}\to \mathsf{H}$ and the natural $\mathsf{H}$-action on $\mathfrak{n}_-$. If we disregard $\mathsf{G}$-actions and use that $\mathsf{G}/\mathsf{A}\mathsf{N}_+=\mathsf{K}$ so $\mathsf{G}/\mathsf{B}=\mathsf{K}/\mathsf{M}$, we can also write $\mathfrak{t}_H\mathfrak{X}\cong \mathsf{K}\times_\mathsf{M} \mathfrak{n}_-$ as bundles of graded, nilpotent Lie algebras.

\begin{example}
\label{sl3c}
Let us compute a longer example of the BGG-complex on $\mathsf{G}=SL(3,\C)$. For the group $SL(3,\C)$, many more details can be found in the work of Yuncken \cite{yunckenthesis}, who also computed further examples in \cite{yunckenhab}. 

We take the maximal compact subgroup $\mathsf{K}=SU(3)\subseteq SL(3,\C)$ and $\mathsf{H}\cong (\C^\times)^2$ to consist of the diagonal matrices with determinant $1$. We choose the isomorphism $(z_1,z_3)\mapsto \mathrm{diag}(z_1,(z_1z_3)^{-1},z_3)$. Similarly, $\mathsf{A}\cong \R_+^2$ consists of the positive diagonal matrices with determinant $1$.  Now $\mathsf{M}=\mathsf{K}\cap \mathsf{H}\cong U(1)^2$ consists of the unitary diagonal matrices with determinant $1$. The Weyl group of $SL(3,\C)$ is $S_3$.

The nilpotent Lie groups $\mathsf{N}_\pm$ consist of unipotent, lower (+)/upper (-) triangular matrices. There are isomorphisms $\mathsf{N}_\pm\cong \mathsf{H}_{3,\C}$ to the three-dimensional complex Heisenberg group. That is, $\mathfrak{n}_\pm$ is the complex span of elements $e_X^\pm,e_Y^\pm, e_Z^\pm$ subject to the commutation relation 
$$[e_X^\pm,e_Y^\pm]=e^Z_\pm.$$
In terms of the matrix units $(e_{jk})_{j,k=1}^3$ spanning $\mathfrak{gl}(3,\C)$, $e_X^+=e_{21}$,$e_Y^+=e_{32}$, and $e_Z^+=e_{31}$ and $e_X^-,e_Y^-, e_Z^-$ are given by their transposes. 

The corresponding roots $\alpha_{\mathsf{x}}$, $\alpha_{\mathsf{y}}$ and $\alpha_{\mathsf{z}}=\alpha_{\mathsf{x}}+\alpha_{\mathsf{y}}$ are under our chosen isomorphism $\mathfrak{h}\cong \C^2$ computed as 
$$\begin{cases}
\alpha_{\mathsf{x}}(z_1,z_3)=-2z_1-z_3\\
\alpha_{\mathsf{y}}(z_1,z_3)=z_1+2z_3.
\end{cases}$$
The more standard notation for the Chevalley-Serre basis defined from $e_X^\pm,e_Y^\pm, e_Z^\pm$ is $e_j=e_j^+$ and $f_j=e_j^-$ (for $j\in \{X,Y,Z\}$). In particular, as an $\mathsf{H}$-module, 
$$\mathfrak{n}_-\cong \C_{\alpha_{\mathsf{x}}}\oplus \C_{\alpha_{\mathsf{y}}}\oplus \C_{\alpha_{\mathsf{z}}}.$$
The Borel group $\mathsf{B}\cong (\C^\times)^2\rtimes \mathsf{N}_+$ consists of the lower triangular matrices with determinant $1$.
The full flag manifold takes the form 
$$\mathfrak{X}=SL(3,\C)/\mathsf{B}=SU(3)/\mathsf{M}.$$
The complex $\mathsf{G}$-equivariant Carnot structure is depth two and comes from the (non-equivariant) decomposition into homogeneous line bundles
\begin{equation}
\label{bladakn}
T\mathfrak{X}=SL(3,\C)\times_\mathsf{B}\C_{\alpha_{\mathsf{x}}}\oplus SL(3,\C)\times_\mathsf{B}\C_{\alpha_{\mathsf{y}}}\oplus SL(3,\C)\times_\mathsf{B}\C_{\alpha_{\mathsf{z}}},
\end{equation}
as 
\begin{equation}
\label{bladakn-1}
T^{-1}\mathfrak{X}=SL(3,\C)\times_\mathsf{B}(\C_{\alpha_{\mathsf{x}}}\oplus \C_{\alpha_{\mathsf{y}}}).
\end{equation}
We remark that even if the identification in Equation \eqref{bladakn} is not $G$-equivariant, rather the right hand side is $\mathsf{G}$-equivariantly isomorphic to $\mathfrak{t}_H\mathfrak{X}$, the two $\mathsf{B}$-actions defining $T\mathfrak{X}$ and $\mathfrak{t}_H\mathfrak{X}$, respectively, coincide on $\C_{\alpha_{\mathsf{x}}}\oplus \C_{\alpha_{\mathsf{y}}}=\mathfrak{g}_{-1}$. Therefore, $T^{-1}\mathfrak{X}\subseteq T\mathfrak{X}$ is a well defined $\mathsf{G}$-invariant sub-bundle.

Let us construct the BGG-complex of the $\mathsf{G}$-module $W=\C$ algebraically. We follow the ideas of \cite{yunckenthesis}. The highest weight of $\C$ is $\lambda=0$ and 
$$\mathcal{M}(0)=\mathcal{U}(\mathfrak{sl}_3)\otimes_{\mathcal{U}(\mathfrak{b})}\C_0\cong \mathcal{U}(\mathfrak{h}_{3,\C}),$$
where $\mathfrak{h}_{3,\C}=\mathfrak{n}_-$ denotes the complex Heisenberg Lie algebra. For notational simplicity, we write $X,Y,Z$ for the standard generators of $\mathfrak{h}_{3,\C}$. The natural map 
$$\epsilon:\mathcal{M}(0)\to \C,$$ 
is in terms of $\mathcal{U}(\mathfrak{h}_{3,\C})$ given by the trivial character $\mathcal{U}(\mathfrak{h}_{3,\C})\to \C$. The kernel of the natural map $\epsilon:\mathcal{M}(0)\to \C$ is therefore
$$\ker(\epsilon)\cong\mathcal{U}(\mathfrak{h}_{3,\C})X+\mathcal{U}(\mathfrak{h}_{3,\C})Y,$$
because $[X,Y]=Z\in \ker(\epsilon)$. We therefore define the $\mathfrak{g}$-linear map 
$$\delta_1:\mathcal{M}(-\alpha_{\mathsf{x}})\oplus \mathcal{M}(-\alpha_{\mathsf{y}})\to \mathcal{M}(0),$$
from the PBW-isomorphism \eqref{pbwiso} $\mathcal{M}(-\alpha_{\mathsf{x}})\oplus \mathcal{M}(-\alpha_{\mathsf{y}})\cong \mathcal{U}(\mathfrak{h}_{3,\C})\oplus \mathcal{U}(\mathfrak{h}_{3,\C})$ and the map 
$$\mathcal{U}(\mathfrak{h}_{3,\C})\oplus \mathcal{U}(\mathfrak{h}_{3,\C})\to \mathcal{U}(\mathfrak{h}_{3,\C}), \quad (w_1,w_2)\mapsto w_1X+w_2Y.$$ 
By construction, $\delta_1$ surjects onto $\ker(\epsilon)$. After some algebra, we see that $\ker(\delta_1)$ under the PBW-isomorphism is generated by the vectors
$$v_1=\begin{pmatrix} -XY-Z\\X^2\end{pmatrix}, \quad\mbox{and}\quad v_2=\begin{pmatrix} -Y^2\\ YX-Z\end{pmatrix}.$$
We therefore define the $\mathfrak{g}$-linear map 
$$\delta_2:\mathcal{M}(-2\alpha_{\mathsf{x}}-\alpha_{\mathsf{y}})\oplus \mathcal{M}(-\alpha_{\mathsf{x}}-2\alpha_{\mathsf{y}})\to \mathcal{M}(-\alpha_{\mathsf{x}})\oplus \mathcal{M}(-\alpha_{\mathsf{y}}),$$
from the PBW-isomorphism \eqref{pbwiso} and the map 
$$\mathcal{U}(\mathfrak{h}_{3,\C})\oplus \mathcal{U}(\mathfrak{h}_{3,\C})\to \mathcal{U}(\mathfrak{h}_{3,\C})\oplus \mathcal{U}(\mathfrak{h}_{3,\C}), \quad (w_1,w_2)\mapsto w_1v_1+w_2v_2.$$ 
By construction $\delta_2$ surjects onto $\ker(\delta_1)$. Continuing in this fashion we arrive at the algebraic BGG-resolution for $W=\C$, that we write as
\footnotesize
\[
\begin{tikzcd}
& &&\mathcal{M}(-\alpha_X)\ar{dl}[swap]{X} &&\mathcal{M}(-2\alpha_X-\alpha_Y)\arrow{ddll}{\!\!\!\!\!\!\!\!\!\!\!\!\!\!\!\!\!\!\!\!\!\!\!\!\!\!\!\!\!\!\!\!\!\!\!\!\!\!\!\!\!\!\!\! X^2}\ar{ll}[swap]{-XY-Z}&& \\
0&\C\arrow[l]&\mathcal{M}(0)\arrow[l,"\epsilon"]&&&&\mathcal{M}(-2\alpha_X-2\alpha_Y)\arrow{ul}[swap]{-Y} \ar[dl,"X"]&0\arrow{l}\\
& &&\mathcal{M}(-\alpha_Y)\ar[ul,"Y"]&&\mathcal{M}(-\alpha_X-2\alpha_Y)\arrow{uull}[swap]{\qquad-Y^2} \ar[ll,"YX-Z"]&&
\end{tikzcd}
\]
\normalsize
For a weight $\lambda$, write $L_\lambda:=\mathsf{G}\times_\mathsf{B}\C_\lambda$. Dualizing this as in Proposition \ref{lknjknjng}, we arrive at the geometric BGG-complex:

\small
\[
\begin{tikzcd}
&&C^\infty(\mathfrak{X};L_{\alpha_{\mathsf{x}}})\ar{rr}{-D_{XY+Z}}\ar{ddrr}{\qquad-D_{Y^2}}&&C^\infty(\mathfrak{X};L_{2\alpha_{\mathsf{x}}\!+\!\alpha_{\mathsf{y}}})\arrow{dr}{-D_Y}&& \\
0\ar[r]&C^\infty(\mathfrak{X})\ar{ur}{D_X}\ar{dr}[swap]{D_Y} &&&&C^\infty(\mathfrak{X};L_{2\alpha_{\mathsf{x}}+2\alpha_{\mathsf{y}}}) \arrow{r}&0\\
&&C^\infty(\mathfrak{X};L_{\alpha_{\mathsf{y}}}) \ar[rr,"D_{YX-Z}"]\ar{uurr}[swap]{\!\!\!\!\!\!\!\!\!\!\!\!\!\!\!\!\!\!\!\!\!\!\!\!\!\!\!\!\!\!\!\!\!\!\!\!\!\!\!\!\!\!\!\!D_{X^2}}&&C^\infty(\mathfrak{X};L_{\alpha_{\mathsf{x}}\!+\!2\alpha_{\mathsf{y}}})\ar{ur}[swap]{D_X}&&
\end{tikzcd}
\]
\normalsize
This complex $\mathsf{BGG}_\bullet(\C)$ is a graded $SL(3,\C)$-equivariant complex of differential operators by Theorem \ref{inadonad9geo}. Its principal symbol complex in the point $e\mathsf{B}\in \mathsf{G}/\mathsf{B}$ is given by 
\[
\begin{tikzcd}
&&\mathcal{S}_0(\mathsf{H}_{3,\C})\ar{rr}{-XY-Z}\ar{ddrr}{\qquad-Y^2} &&\mathcal{S}_0(\mathsf{H}_{3,\C})\arrow{dr}{-Y} && \\
0\ar[r]&\mathcal{S}_0(\mathsf{H}_{3,\C})\ar{ur}{X}\ar{dr}[swap]{Y}&&&&\mathcal{S}_0(\mathsf{H}_{3,\C})\arrow{r}&0\\
&&\mathcal{S}_0(\mathsf{H}_{3,\C})\ar{rr}[swap]{YX-Z}\ar{uurr}[swap]{\!\!\!\!\!\!\!\!\!\!\!\!\!\!\!\!\!\!\!\!\!\!\!\!\!\!\!\!\!\!\!\!\!\!\!\!\!\!\!\!\!\!\!\!X^2}&&\mathcal{S}_0(\mathsf{H}_{3,\C})\ar[ur,"X"] &&
\end{tikzcd}
\]
Here $X,Y,Z$ denote the generators of $\mathfrak{n}_-$ acting on $\mathcal{S}_0(\mathsf{H}_{3,\C})$ by right invariant vector fields. For the special case of $\mathsf{G}=SL(3,\C)$, there is an ingenious construction involving Kasparov's technical theorem due to Yuncken \cite{yunckenthesis}, that produces a class in $KK_0^{SL(3,\C)}(C(SL(3,\C)/\mathsf{B}),\C)$ from the BGG-complex. Our methods above do not quite even produce classes with exponential bounds in $KK_0^{SL(3,\C),\ell}(C(SL(3,\C)/\mathsf{B}),\C)$.

\end{example}

\subsection{Curved Bernstein-Gelfand-Gelfand complexes}
\label{lknaldkna}

The BGG-complex for complex Lie groups was extended much further to parabolic geometries by Cap-Slovak-Soucek \cite{morecap}. The more recent work of Dave-Haller \cite{Dave_Haller1} provides an explicit construction of the BGG-complex from Kostant codifferentials that we briefly recall. We restrict the global aspects of our discussion to the flag manifolds $\mathfrak{X}=\mathsf{G}/\mathsf{P}$ for $\mathsf{G}$ a connected, semisimple Lie group and $\mathsf{P}$ a parabolic subgroup. Our discussion provides the local foundation for describing curved BGG-complexes for general parabolic geometries and pinpointing their symbol complex. 

Fix a complex finite-dimensional representation $W$ of $\mathsf{G}$, for instance the trivial representation. Let $C_*(\mathfrak{n},W):=\wedge^*\mathfrak{n}\otimes W$ and $C^*(\mathfrak{n},W)=\wedge^*\mathfrak{n}^*\otimes W$ denote the Lie algebra homology and cohomology complex, respectively, with coefficients in $W$. We filter these spaces using the grading of $\mathfrak{n}$. The $\mathsf{G}_0$-equivariant identification \eqref{gzeroadaond} allows us to identify the $\mathsf{G}_0$-spaces
$$C^*(\mathfrak{n},W)=C_*(\mathfrak{p}_+,W).$$
The right hand side is equipped with the differential $\delta:C_*(\mathfrak{p}_+,W)\to C_{*-1}(\mathfrak{p}_+,W)$ defining Lie algebra homology
\begin{align*}
\delta(X_1\wedge \cdots \wedge X_k\otimes v)=&\sum_{j=1}^n (-1)^{j+1}X_1\wedge \cdots\wedge \widehat{X_j}\wedge \cdots \wedge X_k\otimes X_j(v)+\\
&+\sum_{i<j} (-1)^{i+j}[X_i,X_j] \wedge X_1\wedge \cdots\wedge \widehat{X_i}\wedge\cdots \wedge \widehat{X_j}\wedge \cdots \wedge X_k\otimes v
\end{align*}
We can identify $\delta$ with a differential $\partial:C^*(\mathfrak{n},W)\to C^{*-1}(\mathfrak{n},W)$ called Kostant's codifferential. By construction, $\partial$ is of filtered degree $0$ and $\mathsf{P}$-equivariant for the $\mathsf{P}$-action defined by identifying $\mathfrak{n}=\mathfrak{g}/\mathfrak{p}$, for more details see \cite[Section 3.1.11]{capslovak}.

Write $E_W:=\mathsf{G}\times_\mathsf{P}W\to \mathfrak{X}$ and note that
$$\wedge^*T^*\mathfrak{X}\otimes E_W=\mathsf{G}\times_\mathsf{P}(C^j(\mathfrak{n},W))=(\mathsf{G}/\mathsf{P}_+)\times_{\mathsf{G}_0}(C^j(\mathfrak{n},W)),$$ 
where the first but not the second equality is one of $G$-bundles (since $P_+$-acts non-trivially on the fibre). Kostant's codifferential $\partial$ induces a vector bundle morphism 
$$\partial_{\mathfrak{X}}:\wedge^*T^*\mathfrak{X}\otimes E_W\to \wedge^{*-1}T^*\mathfrak{X}\otimes E_W,$$
which is $\mathsf{G}$-equivariant and of filtered degree $0$. Here we use the filtering of $T^*\mathfrak{X}$ dual to the filtering of $T\mathfrak{X}$.

Let $\nabla_W$ denote the tractor connection on $E_W$ and extend it to a differential operator 
$$\nabla_W:C^\infty(\mathfrak{X};\wedge^*T^*\mathfrak{X}\otimes E_W)\to C^\infty(\mathfrak{X};\wedge^{*+1}T^*\mathfrak{X}\otimes E_W),$$
using the Leibniz rule. The differential operator $\nabla_W$ is $\mathsf{G}$-equivariant, of graded Heisenberg degree $0$ and its curvature $\nabla_W^2$ vanishes, but for general parabolic geometries it is of filtered degree $-1$.

We now proceed as in \cite{Dave_Haller1} to construct BGG-sequences. The method produces the same result as in \cite{bggoriginal, morecap,lepog} due to the arguments of \cite{morecap,Dave_Haller1}. We note that $T^{-j}\mathfrak{X}=\bigoplus_{l=1}^j T_{-l}\mathfrak{X}$ where $T_{-l}\mathfrak{X}:=(\mathsf{G}/\mathsf{P}_+)\times_{\mathsf{G}_0}\mathfrak{g}_{-l}$. We can identify $T^{-j}\mathfrak{X}=\mathsf{G}\times_{\mathsf{P}}\mathfrak{n}^{-j}$, where $\mathfrak{n}^{-j}:=\bigoplus_{l=1}^j\mathfrak{g}_{-l}\subseteq \mathfrak{n}$ has the $\mathsf{P}$-action defined from the identification $\mathfrak{n}=\mathfrak{g}/\mathfrak{p}$.  Therefore, there is a canonical (non-equivariant) identification $T\mathfrak{X}\cong \mathrm{gr}(T\mathfrak{X})$. Note that under this identification and \eqref{gzeroadaond}, 
\label{lknaldknaada}
$T^*\mathfrak{X}$ is $\mathsf{G}$-equivariantly filtered by 
$$T^*\mathfrak{X}=T^*_1\mathfrak{X}\supsetneq T^*_2\mathfrak{X}\supsetneq\ldots \supsetneq T^*_{r-1}\mathfrak{X}\supsetneq T^*_r\mathfrak{X}\supsetneq0,$$
 for $T_{l}^*\mathfrak{X}:=\mathsf{G}\times_{\mathsf{P}}(\oplus_{j=l}^r\mathfrak{g}_{j})$ where $\mathsf{P}$ acts on $\oplus_{j=l}^r\mathfrak{g}_{j}$ as a subspace of $\mathfrak{p}_+$. If we set 
 $$T^*_{[l]}\mathfrak{X}:=(\mathsf{G}/\mathsf{P}_+)\times_{\mathsf{G}_0}\mathfrak{g}_l,$$ 
 we may note that as a $\mathsf{G}$-vector bundle 
$$\mathrm{gr}(T^*\mathfrak{X})=\mathsf{G}\times_{\mathsf{P}}\mathfrak{p}_+= (\mathsf{G}/\mathsf{P}_+)\times_{\mathsf{G}_0}\mathfrak{p}_+=\oplus_{l=1}^r T^*_{[l]}\mathfrak{X},$$
where we consider $\mathfrak{p}_+$ as a $\mathsf{P}$-module under the projection mapping $\mathsf{P}\to \mathsf{G}_0$ and the natural $\mathsf{G}_0$-action on $\mathfrak{p}_+$ that respects the grading $\mathfrak{p}_+=\oplus_{l=1}^r\mathfrak{g}_l$. We consider $\wedge^l T^*\mathfrak{X}\otimes E_W$ as a filtered $G$-vector bundle. Note the following consequence of \cite[Lemma 4.4]{Dave_Haller1}. 

\begin{lemma}
\label{decomkdh}
Let $W$ be a complex $\mathsf{G}$-representation, and $\partial_{\mathfrak{X}}$ the Kostant codifferential and $\nabla_W$ the tractor connection as above. For $j=0,\ldots, d=\dim(\mathfrak{X})$, consider the filtered Heisenberg degree $0$ differential operator 
$$\Box:=\nabla_W\partial_\mathfrak{X}+\partial_\mathfrak{X}\nabla_W:C^\infty(\mathfrak{X};\wedge^*T^*\mathfrak{X}\otimes E_W)\to C^\infty(\mathfrak{X};\wedge^*T^*\mathfrak{X}\otimes E_W),$$
and its associated graded operator 
$$\tilde{\Box}_j:=\mathrm{gr}(\Box):C^\infty(\mathfrak{X};\mathrm{gr}(\wedge^*T^*\mathfrak{X}\otimes E_W))\to C^\infty(\mathfrak{X};\mathrm{gr}(\wedge^*T^*\mathfrak{X}\otimes E_W)).$$
Then $\tilde{\Box}$ is a bundle endomorphism whose generalized eigenprojection $\tilde{P}$ for the eigenvalue $0$ is also a bundle endomorphism. 

Moreover, there exists a unique filtered Heisenberg order $0$, differential projector 
$$P:C^\infty(\mathfrak{X};\wedge^*T^*\mathfrak{X}\otimes E_W)\to C^\infty(\mathfrak{X};\wedge^*T^*\mathfrak{X}\otimes E_W),$$
such that 
\begin{enumerate}
\item $P\Box=\Box P$
\item $\mathrm{gr}(P)=\tilde{P}$
\item The decomposition
$$C^\infty(\mathfrak{X};\wedge^*T^*\mathfrak{X}\otimes E_W)=\ker(P)\oplus \mathrm{im}(P),$$
is preserved by $\Box$ which acts invertibly on the first summand and nilpotently on the second. 
\end{enumerate}
\end{lemma}

Consider the operator 
\begin{equation}
\label{lknaldnalkdjnad}
L:=P\tilde{P}+(1-P)(1-\tilde{P}).
\end{equation}
A degree argument shows that $L$ is invertible with inverse being a filtered Heisenberg order $0$, differential operator. These filtered Heisenberg order $0$, differential operators restrict to isomorphisms
\begin{align*}
L|:&C^\infty(\mathfrak{X};\tilde{P}(\wedge^*T^*\mathfrak{X}\otimes E_W))\xrightarrow{\sim} \mathrm{im}(P),\\
L|:&C^\infty(\mathfrak{X};(1-\tilde{P})(\wedge^*T^*\mathfrak{X}\otimes E_W))\xrightarrow{\sim} \ker(P).
\end{align*}

It follows from Hodge theory that the bundle endomorphism $\tilde{P}$ projects onto the bundle $\bigoplus_j\mathpzc{H}_j(W)$, where 
\begin{equation}
\label{lknalkdnaodjn}
\mathpzc{H}_j(W):=(\mathsf{G}/\mathsf{P}_+)\times_{\mathsf{G}_0} H_j(\mathfrak{p}_+,W)\to \mathfrak{X},
\end{equation}
is viewed as a filtered sub-bundle of $\wedge^jT^*\mathfrak{X}\otimes E_W$ in terms of harmonic forms. Note that the $\mathsf{G}_0$-representation $H^j(\mathfrak{p}_+,W)$ can be computed from Theorem \ref{knlkanlknad} and the highest weight $\lambda$ of $W$, the lowest form $\delta$ on $\mathfrak{g}$ and the Hasse diagram $\mathcal{W}^{\mathfrak{p}}\subseteq \mathcal{W}$  of $\mathfrak{p}$, as 
$$\bigoplus_{w\in \mathcal{W}^{\mathfrak{p}}, \ \ell(w)=j} \mathpzc{V}_{w(\lambda+\delta)-\delta},$$
where $\mathpzc{V}_{w(\lambda+\delta)-\delta}$ are irreducible finite-dimensional $\mathsf{G}_0$-representations with highest weight $w(\lambda+\delta)-\delta$.

\begin{definition}
\label{defbggcomp}
Let $W$ be a complex $\mathsf{G}$-representation. Form the filtered Heisenberg order $0$, differential operator 
$$D^{\rm BGG}_j:=\tilde{P}_{j}L_{j}^{-1}\nabla_W L_{j-1}:C^\infty(\mathfrak{X};\mathpzc{H}_{j-1}(W))\to C^\infty(\mathfrak{X};\mathpzc{H}_{j}(W)).$$
The BGG-complex of $W$ is the $\mathsf{G}$-equivariant sequence of differential operators
\begin{align*}
\mathsf{BGG}_\bullet(W):\quad 0\to C^\infty(\mathfrak{X};\mathpzc{H}_0(W))&\xrightarrow{D^{\rm BGG}_1} C^\infty(\mathfrak{X};\mathpzc{H}_1(W))\xrightarrow{D^{\rm BGG}_2}  \cdots \\
&\xrightarrow{D^{\rm BGG}_{d-1}} C^\infty(\mathfrak{X};\mathpzc{H}_{d-1}(W))\xrightarrow{D^{\rm BGG}_d} C^\infty(\mathfrak{X};\mathpzc{H}_d(W))\to 0,
\end{align*}
where $d=\dim(\mathfrak{X})$.
\end{definition}

From \cite[Proposition 4.5]{Dave_Haller1} we conclude the following.

\begin{prop}
\label{akdnalkdn}
The sequence $\mathsf{BGG}_\bullet(W)$ is a graded Rockland complex (cf. Definition \ref{gradedrock}) and its $\mathsf{K}$-equivariant index class $\mathrm{ind}_\mathsf{K}[\mathsf{BGG}_\bullet(W)]\in R(\mathsf{K})$ coincides with the index class of the de Rham complex
\begin{align*}
0\to C^\infty(\mathfrak{X};E_W)&\xrightarrow{\nabla_W} C^\infty(\mathfrak{X};T^*\mathfrak{X}\otimes E_W)\xrightarrow{\nabla_W}  \cdots \\
&\xrightarrow{\nabla_W} C^\infty(\mathfrak{X};\wedge^{d-1}T^*\mathfrak{X}\otimes E_W)\xrightarrow{\nabla_W} C^\infty(\mathfrak{X};\wedge^dT^*\mathfrak{X}\otimes E_W)\to 0.
\end{align*}
\end{prop}

\begin{proof}
By construction, $\nabla_W=L_j(D^{\rm BGG}_j\oplus \mathfrak{B}_j)L_{j-1}^{-1}$ and $\mathfrak{B}_j$ fits into an acyclic complex by item (3) of Lemma \ref{decomkdh}. Since $\mathsf{K}$ is compact, all structures can be assumed to be $\mathsf{K}$-equivariant so $\mathrm{ind}_\mathsf{K}[\nabla_W]=\mathrm{ind}_\mathsf{K}[\mathsf{BGG}_\bullet(W)]+\mathrm{ind}_\mathsf{K}[\mathfrak{B}_\bullet]=\mathrm{ind}_\mathsf{K}[\mathsf{BGG}_\bullet(W)]$.
\end{proof}

\begin{remark}
Since the Euler characteristic of $S^{k}$ vanishes for odd $k$, it is clear that for $\mathsf{G}=Sp(n,1), SU(n,1), SO(2n,1)$ the associated BGG-complexes will have a $\mathsf{K}$-equivariant index of vanishing formal dimension. In particular, since $\mathrm{Res}^\mathsf{G}_\mathsf{K}\gamma=1$ the $\gamma$-element can in these cases not be represented by a BGG-complex. See \cite{julgsun1,julgsp} for more creative approaches. 
\end{remark}

\subsection{General parabolic geometries}
\label{generalpara}

We can now turn to the case of the general parabolic geometries. The ideas of Lemma \ref{decomkdh} and the discussion after it holds for general parabolic geometries, so for any $\mathsf{G}$-representation $W$ and regular parabolic geometry $(\mathfrak{X},\omega)$ of type $(\mathsf{G},\mathsf{P})$ we can define a tractor bundle $E(W):=V\times_\mathsf{P}W\to \mathfrak{X}$ with induced connection $\nabla_W$. The connection $\nabla_W$ is a filtered connection of degree $0$, i.e. for $f\in C^\infty(\mathfrak{X};\pmb{E}^p)$ and $X\in C^\infty(\mathfrak{X};T^{k}\mathfrak{X})$ it holds that $\nabla_{W,X}f\in C^\infty(\mathfrak{X};\pmb{E}^{p+k})$.

Proceeding as above, we arrive at a BGG-sequence $\mathsf{BGG}_\bullet(\mathfrak{X},W)$ which is the $\Aut(\mathfrak{X},\omega)$-equivariant sequence of differential operators of the form 
 \begin{align*}
\mathsf{BGG}_\bullet(\mathfrak{X},W):\quad 0\to C^\infty(\mathfrak{X};\mathpzc{H}_0(W))&\xrightarrow{D^{\rm BGG}_1} C^\infty(\mathfrak{X};\mathpzc{H}_1(W))\xrightarrow{D^{\rm BGG}_2}  \cdots \\
&\xrightarrow{D^{\rm BGG}_{d-1}} C^\infty(\mathfrak{X};\mathpzc{H}_{d-1}(W))\xrightarrow{D^{\rm BGG}_d} C^\infty(\mathfrak{X};\mathpzc{H}_d(W))\to 0,
\end{align*}
where $d=\dim(\mathfrak{X})$ and $\mathpzc{H}_j(W)=V\times_\mathsf{P} H_j(\mathfrak{p}_+,W)$. We again have that the complex $\mathsf{BGG}_\bullet(\mathfrak{X},W)$ is a graded Rockland sequence (cf. Definition \ref{gradedrock}). Moreover, for any compact subgroup $K\subseteq \Aut(\mathfrak{X},\omega)$ we can define a $K$-equivariant index class $\mathrm{ind}_K[\mathsf{BGG}_\bullet(\mathfrak{X},W)]\in R(K)$. The index theory of graded Rockland sequences has been studied in \cite[Section 5.4]{Dave_Haller1} and \cite[Section 29]{goffkuz}. The reader can find a concrete definition of the index class $\mathrm{ind}_K[\mathsf{BGG}_\bullet(\mathfrak{X},W)]\in R(K)$ below in the next section, see in particular Theorem \ref{lknlknjkjaudn}. The argument proving Proposition \ref{akdnalkdn} carries over modulo compact perturbations to general parabolic geometries. By the invariance of the index under compact perturbations, we conclude a special case of Theorem \ref{ljnadoajdn} below that $\mathrm{ind}_K[\mathsf{BGG}_\bullet(\mathfrak{X},W)]$ coincides with the index class of the de Rham sequence
\begin{align*}
0\to C^\infty(\mathfrak{X};E(W))&\xrightarrow{\nabla_W} C^\infty(\mathfrak{X};T^*\mathfrak{X}\otimes E(W))\xrightarrow{\nabla_W}  \cdots \\
&\xrightarrow{\nabla_W} C^\infty(\mathfrak{X};\wedge^{d-1}T^*\mathfrak{X}\otimes E(W))\xrightarrow{\nabla_W} C^\infty(\mathfrak{X};\wedge^dT^*\mathfrak{X}\otimes E(W))\to 0.
\end{align*}
We sometimes write $\mathsf{BGG}_\bullet(\mathfrak{X},W)=\mathsf{BGG}_\bullet(\nabla_E)$ when we want to emphasize the dependence on the tractor connection. 

\section{$K$-homology classes and graded Rockland sequences}
\label{seconrock}

Throughout this section we assume that $\mathfrak{X}$ is a compact Carnot manifold. 

\subsection{Constructing $K$-cycles from graded Rockland sequences}
\label{subsdsseconrock}

\begin{lemma}
\label{splittinprop}
Let $\pmb{m}=(m_1,\ldots, m_N)\in \R^N$. Consider a sequence 
$$D_\bullet: \quad 0\to C^\infty(\mathfrak{X};\pmb{E}_0)\xrightarrow{D_1}C^\infty(\mathfrak{X};\pmb{E}_1)\xrightarrow{D_2}\cdots \xrightarrow{D_N}C^\infty(\mathfrak{X};\pmb{E}_N)\to 0,$$
where $D_j\in \Psi^{m_j}_{H,{\rm gr}}(\mathfrak{X}; \pmb{E}_{j-1},\pmb{E}_j)$ are operators satisfying 
$$\sigma_H^{m_j}(D_j)\sigma_H^{m_{j-1}}(D_{j-1})=0, \quad j=1,\ldots, N-1. $$
Then the following are equivalent
\begin{enumerate}
\item $D_\bullet$ is a graded Rockland sequence.
\item There is a sequence 
$$
B_\bullet: \quad 0\leftarrow C^\infty(\mathfrak{X};\pmb{E}_0)\xleftarrow{B_1}C^\infty(\mathfrak{X};\pmb{E}_1)\xleftarrow{B_2}\cdots \xleftarrow{B_N}C^\infty(\mathfrak{X};\pmb{E}_N)\leftarrow 0,
$$
where $B_j\in \Psi^{-m_j}_{H,{\rm gr}}(\mathfrak{X}; \pmb{E}_{j},\pmb{E}_{j-1})$ satisfy that 
$$
\begin{cases}
B_jB_{j+1}\in \Psi^{-m_j-m_{j+1}-1}_{H,{\rm gr}}(\mathfrak{X};\pmb{E}_{j+1},\pmb{E}_{j-1}),\\
B_jD_j+D_{j-1}B_{j-1}-1\in \Psi^{-1}_{H,{\rm gr}}(\mathfrak{X};\pmb{E}_j).
\end{cases}$$
\end{enumerate}
Moreover, if the sequence satisfies that $D_{j}D_{j-1}\in \Psi^{-\infty}(\mathfrak{X}; \pmb{E}_{j-2},\pmb{E}_j)$ for $j=1,\ldots, N-1$, item (1) and (2) are equivalent to the following statement
\begin{enumerate}
\item[(3)] There is a sequence 
$$
B_\bullet: \quad 0\leftarrow C^\infty(\mathfrak{X};\pmb{E}_0)\xleftarrow{B_1}C^\infty(\mathfrak{X};\pmb{E}_1)\xleftarrow{B_2}\cdots \xleftarrow{B_N}C^\infty(\mathfrak{X};\pmb{E}_N)\leftarrow 0,
$$
where $B_j\in \Psi^{-m_j}_{H,{\rm gr}}(\mathfrak{X}; \pmb{E}_{j},\pmb{E}_{j-1})$ satisfy that 
$$
\begin{cases}
B_jB_{j+1}\in \Psi^{-\infty}(\mathfrak{X};\pmb{E}_{j+1},\pmb{E}_{j-1}),\\
B_jD_j+D_{j-1}B_{j-1}-1\in \Psi^{-\infty}(\mathfrak{X};\pmb{E}_j).
\end{cases}$$
\end{enumerate}
\end{lemma}

We call a sequence $B_\bullet$ satisfying item (2) in Lemma \ref{splittinprop} a Heisenberg splitting. We note that the principal symbol of any $B_\bullet$  satisfying item (2) in Lemma \ref{splittinprop} produces a chain homotopy from the identity map to the zero map on the principal symbol complex and is also a graded Rockland sequence.

\begin{proof}
It is clear that 3) implies 2), and by asymptotic completeness (and an argument as in \cite[Chapter XVIII]{horIII}) 2) implies 3) when $D_{j}D_{j-1}\in \Psi^{-\infty}(\mathfrak{X}; \pmb{E}_{j-2},\pmb{E}_j)$ for $j=1,\ldots, N-1$. It is also clear that 2) implies 1). It remains to prove that a Rockland sequence admits a collection  $B_\bullet$ satisfying item (2).

To construct operators $B_j\in \Psi_{H,{\rm gr}}^{-m_j}(\mathfrak{X};\pmb{E}_{j+1},\pmb{E}_j)$ satisfying item (2) for a Rockland sequence $D_\bullet$, it suffices to construct its principal symbols $b_j:=\sigma_H^{-m_j}(B_j)\in \Sigma^{-m_j}_{H,{\rm gr}}(\mathfrak{X}; \pmb{E}_{j},\pmb{E}_{j-1})$ such that
\begin{align}
\label{homprop}
b_jd_j+d_{j-1}b_{j-1}&=1, \quad\mbox{and}\\
\label{compprop}
b_jb_{j+1}&=0.
\end{align}
Here we write $d_j:=\sigma_H^{m_j}(D_j)\in \Sigma^{m_j}_{H,{\rm gr}}(\mathfrak{X}; \pmb{E}_{j-1},\pmb{E}_{j})$. We construct the collection $(b_j)_{j=1}^N$ through three steps, the base case for an induction argument, then the induction step to obtain the property \eqref{homprop} and then a recursive  procedure to obtain the property \eqref{compprop}.\\

{\bf Claim 1: }\emph{There is a  left inverse $b_1\in \Sigma^{-m_1}_{H,{\rm gr}}(\mathfrak{X}; \pmb{E}_{1},\pmb{E}_{0})$ to $d_1$.}

{\bf Proof of Claim 1.}
Since $D_\bullet$ is a Rockland sequence, $d_1$ satisfies the Rockland condition and admits a left inverse $b_1\in \Sigma^{-m_1}_{H,{\rm gr}}(\mathfrak{X}; \pmb{E}_{1},\pmb{E}_{0})$ by Theorem \ref{rockandhypo}. For the convenience of the reader we note that since $b_1d_1=1$, the element $e_1:=d_1b_1$ satisfies 
$$e_1^2=d_1b_1d_1b_1=d_1b_1=e_1,$$ 
and $e_1$ is a range projection for $d_1$. \\

{\bf Claim 2: }\emph{For $j>1$ and a collection $b_k\in \Sigma^{-m_k}_{H,{\rm gr}}(\mathfrak{X}; \pmb{E}_{k},\pmb{E}_{k-1})$, $k=1,\ldots, j-1$, of symbols such that 
\begin{equation}
\label{chaiandad}
\begin{cases} 
b_{k-1}b_k=0,\\
b_kd_k+d_{k-1}b_{k-1}=1,
\end{cases}
\end{equation}
 we can construct $\tilde{b}_j\in \Sigma^{-m_j}_{H,{\rm gr}}(\mathfrak{X}; \pmb{E}_{j},\pmb{E}_{j-1})$, such that 
$$\tilde{b}_jd_j+d_{j-1}b_{j-1}=1,$$ 
and the element 
$$b_j:=(1-d_{j-1}b_{j-1})\tilde{b}_j,$$
will fulfill
\begin{equation}
\label{chaiandadatj}
\begin{cases} 
b_{j-1}b_j=0,\\
b_jd_j+d_{j-1}b_{j-1}=1,
\end{cases}
\end{equation} }

{\bf Proof of Claim 2.}
We note that Equation \ref{chaiandad} implies that $d_kb_k$ are range projections of $d_k$, for $k=1,\ldots, j-1$. 
Consider the symbol 
$$\hat{d}_j:=\begin{pmatrix} d_j\\ b_{j-1}\end{pmatrix} \in \Sigma^{m_j}_{H, {\rm gr}}(\mathfrak{X}; \pmb{E}_{j-1},\pmb{E}_{j}\oplus \pmb{E}_{j-2}\langle m_j-m_{j-1}\rangle),$$ 
where we write $\pmb{E}_{j-2}\langle m_j-m_{j-1}\rangle$ for $\pmb{E}_{j-2}$ with degree shifted by $m_j-m_{j-1}$. By the induction assumption, $\hat{d}_j$ satisfies the Rockland condition and admits a left inverse $\hat{b}_j\in \Sigma^{-m_j}_{H,{\rm gr}}(\mathfrak{X}; \pmb{E}_{j}\oplus \pmb{E}_{j-2}\langle m_j-m_{j-1}\rangle,\pmb{E}_{j-1})$ by Theorem \ref{rockandhypo}. We write 
$$\hat{b}_j=(\tilde{b}_j,\check{b}_j).$$
Here  $\tilde{b}_j\in \Sigma^{-m_j}_{H,{\rm gr}}(\mathfrak{X}; \pmb{E}_{j},\pmb{E}_{j-1})$ and unravelling the degree shifts we see that 
$$\check{b}_j\in \Sigma^{m_{j-1}}_{H,{\rm gr}}(\mathfrak{X}; \pmb{E}_{j-2},\pmb{E}_{j-1}).$$
Our defining property of $\hat{b}_j$ is that 
\begin{equation}
\label{lkjnkjandakjnakdjad}
\tilde{b}_jd_j+\check{b}_jb_{j-1}=1.
\end{equation}
However, the property \eqref{lkjnkjandakjnakdjad} is unaltered if we replace $\check{b}_j$ with $\check{b}_j(1-d_{j-2}b_{j-2})$. Since $1-d_{j-2}b_{j-2}$ is idempotent by assumption, we therefore have imposed that 
$$\check{b}_j=\check{b}_j(1-d_{j-2}b_{j-2})=\check{b}_jb_{j-1}d_{j-1},$$ 
where we use the induction assumption in the last equality. Multiplying the identity \eqref{lkjnkjandakjnakdjad} with $d_{j-1}$ from the right we see that $\check{b}_j=d_{j-1}$ and so 
\begin{equation}
\label{lkjnkjandakjnakdjadneq}
\tilde{b}_jd_j+d_{j-1}b_{j-1}=1.
\end{equation}

We need to show that the element $b_j:=(1-d_{j-1}b_{j-1})\tilde{b}_j$ fulfills Equation \eqref{chaiandadatj}. Since $d_jd_{j-1}=0$ for all $j$, it is clear that $b_jd_j+d_{j-1}b_{j-1}=1$. We compute that 
$$b_{j-1}b_j=(1-d_{j-2}b_{j-2})b_{j-1}(1-d_{j-1}b_{j-1})\tilde{b}_j=0,$$
because 
$$b_{j-1}(1-d_{j-1}b_{j-1})=b_{j-1}-b_{j-1}d_{j-1}b_{j-1}=b_{j-1}-(1-d_{j-2}b_{j-2})b_{j-1}=b_{j-1}-b_{j-1}=0,$$
where we in the second last equality used the induction assumption $b_{j-2}b_{j-1}=0$.  \\

Claim 1 above gives the first induction step $b_1$, Claim $2$ by induction gives us the required symbols $b_j$, $j=1,\ldots,N$, satisfying Equation \eqref{homprop} and \eqref{compprop}.
\end{proof}

\begin{definition}
\label{splittindefedf}
Consider a graded Rockland sequence $D_\bullet$ of order $\pmb{m}=(m_1,\ldots, m_N)\in \R^N$ on a collection of filtered vector bundles $\pmb{E}_j\to \mathfrak{X}$ with filtered $G$-actions making the symbol complex $\sigma_H(D_\bullet)$ $G$-equivariant. 
\begin{itemize}
\item If the splitting operators $B_\bullet\in \Psi^{-m_\bullet}_{H,{\rm gr}}(\mathfrak{X};\pmb{E}_\bullet,\pmb{E}_{\bullet-1})$ from item 2) in Lemma \ref{splittinprop} can be chosen so that $\sigma_H(B_\bullet)$ is $G$-equivariant, we say that $B_\bullet$ is a $G$-equivariant Heisenberg splitting and we say that $D_\bullet$ is $G$-equivariantly Heisenberg split.
\item We say that $D_\bullet$ has a \emph{$G$-equivariant Rockland splitting} of order $\pmb{s}\in \R^{N+1}$, solving $s_{j-1}-s_{j}=m_j$ for $j=1,\ldots, N$, if there are bounded operators 
$$B_j:W^{s_j}_{H, {\rm fil}}(\mathfrak{X};\pmb{E}_j)\to W^{s_{j-1}}_{H, {\rm fil}}(\mathfrak{X};\pmb{E}_{j-1}),$$
commuting with the $G$-actions and $C^\infty(\mathfrak{X})$ up to compact operators, and 
$$\begin{cases}
B_jD_j+D_{j-1}B_{j-1}-1&\in \mathbb{K}(W^{s_j}_{H, {\rm fil}}(\mathfrak{X};\pmb{E}_j)),\\
B_{j-1}B_{j}&\in \mathbb{K}(W^{s_j}_{H, {\rm fil}}(\mathfrak{X};\pmb{E}_j),W^{s_{j-2}}_{H, {\rm fil}}(\mathfrak{X};\pmb{E}_{j-2})).
\end{cases}$$
\end{itemize}
\end{definition}

A $G$-equivariant Heisenberg splitting is clearly a $G$-equivariant Rockland splitting. We note the following fact that follows from averaging over the group action. 

\begin{prop}
\label{lknalkdna}
Let $G$ be a compact group. Consider a graded Rockland sequence $D_\bullet$ on a collection of filtered vector bundles $\pmb{E}_j\to \mathfrak{X}$ with filtered $G$-actions. If the symbol complex $\sigma_H(D_\bullet)$ is $G$-equivariant, then $D_\bullet$ is $G$-equivariantly Heisenberg split. In fact, we can choose the splitting operators $B_\bullet\in \Psi^{m_\bullet}_{H,{\rm gr}}(\mathfrak{X};\pmb{E}_\bullet,\pmb{E}_{\bullet-1})$ from Lemma \ref{splittinprop} to be $G$-equivariant.
\end{prop}

\begin{proof}
Our proof consists of altering the idea in the proof of Lemma \ref{splittinprop} by carefully averaging over $G$ at correct steps. Arguing as in Claim 1 of Lemma \ref{splittinprop}, we can after averaging over $G$ assume that $b_1$ is $G$-invariant. As in Claim 2 in the proof of Lemma \ref{splittinprop} we can proceed inductively and from a collection $b_k\in \Sigma^{-m_k}_{H,{\rm gr}}(\mathfrak{X}; \pmb{E}_{k},\pmb{E}_{k-1})$, $k=1,\ldots, j-1$ of $G$-invariant symbols such that \eqref{chaiandad} holds for $k=1,\ldots, j-1$, construct a 
$\tilde{b}_j\in \Sigma^{-m_j}_{H,{\rm gr}}(\mathfrak{X}; \pmb{E}_{j},\pmb{E}_{j-1})$, such that 
$$\tilde{b}_jd_j+d_{j-1}b_{j-1}=1.$$ 
After averaging over $G$, we can assume that $\tilde{b}_j$ is $G$-invariant. The element 
$$b_j:=(1-d_{j-1}b_{j-1})\tilde{b}_j,$$
is by construction $G$-invariant and the $G$-invariant collection $b_k\in \Sigma^{-m_k}_{H,{\rm gr}}(\mathfrak{X}; \pmb{E}_{k},\pmb{E}_{k-1})$, $k=1,\ldots, j$ will fulfill \eqref{chaiandad} for $k=1,\ldots, j$. By induction we arrive at a $G$-invariant collection $b_k\in \Sigma^{-m_k}_{H,{\rm gr}}(\mathfrak{X}; \pmb{E}_{k},\pmb{E}_{k-1})$, $k=1,\ldots, N$ satisfying \eqref{chaiandad} for $k=1,\ldots, N$. We take $B_\bullet\in \Psi^{m_\bullet}_{H,{\rm gr}}(\mathfrak{X};\pmb{E}_\bullet,\pmb{E}_{\bullet-1})$ to be any operators with $b_j=\sigma_H^{-m_j}(B_j)$.
\end{proof}

We are interested in making more precise the way a group acts on the Heisenberg-Sobolev spaces, so we shall make the topology of the latter more precise. Fix a volume density $\rd \mu$ on $\mathfrak{X}$. For a graded Hermitean vector bundle $\pmb{E}\to \mathfrak{X}$ with a graded Hermitean connection $\nabla_{\pmb{E}}=\oplus_k \nabla_{\pmb{E}[k]}$, and $s\in \R$, we write $W^s_{H,{\rm fil}} (\mathfrak{X};\pmb{E})$ for the filtered Sobolev space as in Equation \eqref{filtsob}. In \eqref{filtsob}, we equip $W^{s-p}_{H} (\mathfrak{X};\pmb{E}[k])$ with a Hilbert space structure from the connection $\nabla_{\pmb{E}[k]}$, and duality and interpolation.

\begin{prop}
\label{lengthsknon}
Assume that $\ell_0,\ell_1,\ell_2 , \ell_3$ are length functions on $G$ such that 
\begin{enumerate}
\item $\frac{\rd g_*\mu}{\rd \mu}\leq \mathrm{e}^{2\ell_0(g)}$ for any $g\in G$;
\item For any $x\in \mathfrak{X}$ and $j=0,\ldots, N$, and $g\in G$
$$\|g:\pmb{E}_{j,x}\to \pmb{E}_{j,gx}\|\leq \mathrm{e}^{\ell_1(g)}.$$
\item For any $j=0,\ldots, N$, $f\in C^\infty(\mathfrak{X};\pmb{E}_j)$ and $g\in G$
$$\int_\mathfrak{X}\|\nabla_{\pmb{E}_j}(g^*f)\|_{T\mathfrak{X}\otimes \pmb{E}_j}^2\mathrm{d}\mu\leq \mathrm{e}^{2\ell_2(g)}\left(\int_\mathfrak{X}\|g^*(\nabla_{\pmb{E}_j}f)\|_{T\mathfrak{X}\otimes \pmb{E}_j}^2\mathrm{d}\mu+\int_\mathfrak{X}\|f\|_{\pmb{E}_j}^2\mathrm{d}\mu\right).$$
\item  $\|g^{-1}S_{\pmb{E}}\mathrm{gr}(g)\|_{L^\infty(\mathfrak{X}; \mathrm{Hom}(\mathrm{gr}(\pmb{E}),\pmb{E}))}\leq \mathrm{e}^{\ell_3(g)}$ for any $g\in G$.
\end{enumerate}
Then there is a $C>0$ such that the length function defined by  
$$\ell= \ell_0+\ell_1+|\pmb{s}|_{\ell^\infty} (\ell_2+\ell_3)+C.$$
satisfies that for $g\in G$ acting on $W^{\pmb{s}}_{H,{\rm fil}} (\mathfrak{X};\pmb{E}_\bullet)$, we have that 
$$\|g\|_{W^{\pmb{s}}_{H,{\rm fil}} (\mathfrak{X};\pmb{E}_\bullet)}\leq \mathrm{e}^{\ell(g)}.$$

\end{prop}

There exists length functions $\ell_0,\ell_1,\ell_2 , \ell_3$ on $G$ satisfying the conditions of Proposition \ref{lengthsknon}, as one can reverse engineer the conditions into definitions. For instance, we can define $\ell_0(g):=\frac{1}{2}\left\|\log\left(\frac{\rd g_*\mu}{\rd \mu}\right)\right\|_{L^\infty}$.

\begin{proof}
The proof is by a long computation for $\pmb{s}$ integer; commuting the group action through connections and changes of variables. The case of a general $\pmb{s}$ follows from interpolation.
\end{proof}

Write $\pmb{E}_j^{\rm tr}$ for $\pmb{E}_j$ equipped with the trivial filtration. We fix order reducing operators $\Lambda_{j,s}\in \Psi^s_{H,{\rm gr}}(\mathfrak{X};\pmb{E}_j,\pmb{E}_j^{\rm tr})$, i.e. $\Lambda_{j,s}$ is invertible and $H$-elliptic. From the order reducing operators we define an invertible operator 
\begin{equation}
\label{lknaldkjnadln}
U:L^2(\mathfrak{X};\pmb{E}_\bullet^{\rm tr})\to W^{\pmb{s}}_{H,{\rm fil}} (\mathfrak{X};\pmb{E}_\bullet),
\end{equation}
that we can assume to be unitary.

\begin{prop}
Consider the $C^*$-algebra 
$$A:=C(\mathfrak{X})+\mathbb{K}(L^2(\mathfrak{X};\pmb{E}_\bullet^{\rm tr}))\subseteq \mathbb{B}(L^2(\mathfrak{X};\pmb{E}_\bullet^{\rm tr})).$$
The unitary isomorphism \eqref{lknaldkjnadln} induces a faithful representation 
$$\pi:A\to \mathbb{B}(W^{\pmb{s}}_{H,{\rm fil}} (\mathfrak{X};\pmb{E}_\bullet)),\quad a\mapsto UaU^*,$$
satisfies that 
$$\pi(a)-a\in \mathbb{K}(W^{\pmb{s}}_{H,{\rm fil}} (\mathfrak{X};\pmb{E}_\bullet)), \quad a\in  C^\infty(\mathfrak{X}).$$
The action of $G$ on $W^{\pmb{s}}_{H,{\rm fil}} (\mathfrak{X};\pmb{E}_\bullet)$ satisfies that 
$$g^{-1} \pi(a)g-\pi(g^*a) \in \mathbb{K}(W^{\pmb{s}}_{H,{\rm fil}} (\mathfrak{X};\pmb{E}_\bullet)), \quad a\in C(\mathfrak{X}),$$
and the inclusion of the compact operators defines a $G$-equivariant split short exact sequence
$$0\to  \mathbb{K}(L^2(\mathfrak{X};\pmb{E}_\bullet^{\rm tr}))\to A\to C(\mathfrak{X})\to 0.$$

\end{prop}

\begin{proof}
The first statement is clear from construction and the fact that the calculus commute up to lower order terms with smooth functions. The second statement follows from the first. 
\end{proof}

\begin{remark}
The necessity for using the representation $\pi$ is that the pointwise action of $C^\infty(\mathfrak{X})$ on $W^{\pmb{s}}_{H,{\rm fil}} (\mathfrak{X};\pmb{E}_\bullet)$ does not extend to $C(\mathfrak{X})$.
\end{remark}

\begin{thm}
\label{lknlknjkjaudn}
Let $D_\bullet$ be a $G$-equivariant graded Rockland sequence of order $\pmb{m}=(m_1,\ldots, m_N)$ on $\pmb{E}_\bullet=(\pmb{E}_0,\ldots, \pmb{E}_N)$. Assume that the $G$-action on $\pmb{E}$ is by conformal unitaries. Take $\pmb{s}=(s_0,\ldots, s_N)\in \R^{N+1}$ such that 
$$s_{j-1}-s_{j}=m_j, \quad j=1,\ldots, N.$$
Assume that we have a $G$-equivariant Rockland splitting $B_\bullet$ of order $\pmb{s}$.
If we set 
$$F_0:=\sum_jD_{2j+1}+B_{2j}:W^{\pmb{s}}_{H,{\rm fil}} (\mathfrak{X};\oplus_j \pmb{E}_{2j})\to W^{\pmb{s}}_{H,{\rm fil}} (\mathfrak{X};\oplus_j\pmb{E}_{2j+1}),$$
and $U:=F_0|F_0|^{-1}$, then the operator 
$$F:=\begin{pmatrix}
0& U\\
U^*&0\end{pmatrix},$$
fits into the collection 
$$(\pi,W^{\pmb{s}}_{H,{\rm fil}} (\mathfrak{X};\pmb{E}_\bullet),F),$$
that forms a Fredholm module for $C(\mathfrak{X})$ with $g(F)-F\in \mathbb{K}(W^{\pmb{s}}_{H,{\rm fil}} (\mathfrak{X};\pmb{E}_\bullet))$ for any $g\in G$. That is, we have $F=F^*$, $F^2-1$ is compact and $[F,\pi(a)]$ is compact for any $a\in C(\mathfrak{X})$.
\end{thm}

\begin{proof}
We write $\tilde{F}_0:=\sum_jD_{j}+B_{j}:W^{\pmb{s}}_{H,{\rm fil}} (\mathfrak{X};\oplus_j \pmb{E}_{j})\to W^{\pmb{s}}_{H,{\rm fil}} (\mathfrak{X};\oplus_j\pmb{E}_{j})$. By assumption, $\tilde{F}_0^2-1$ is compact and $\tilde{F}_0$ almost commutes with $A$ and the $G$-action. Since the $G$-action on $\pmb{E}$ is by conformal unitaries, the contragredient representation $\pi^*(g):=\pi(g^{-1})^*$ takes the form $\pi^*(g)=a_g\pi(g)$ for an $A$-valued $G$-cocycle $g\mapsto a_g$. Therefore, $\tilde{F}_0$ also almost commutes with the contragredient representation. We conclude that both $\tilde{F}_0$ and $\tilde{F}_0^*$ almost commutes with $A$ and the $G$-action, so $U$ and therefore also $F$ almost commutes with $A$ and the $G$-action.
\end{proof}

\begin{prop}
Let $D_\bullet$ be a $G$-equivariant graded Rockland sequence. The Fredholm module $(\pi,W^{\pmb{s}}_{H,{\rm fil}} (\mathfrak{X};\pmb{E}_\bullet),F)$ from Theorem \ref{lknlknjkjaudn} is up to homotopy independent of the choice of $B_\bullet$ and $\pmb{s}$. In fact, as soon as $D'_\bullet$ is another $G$-equivariant graded Rockland sequence with $\sigma_H(D_\bullet)=\sigma_H(D'_\bullet)$ then the associated Fredholm modules are operator homotopic. 
\end{prop}

\begin{proof}
For fixed $D_\bullet$, we show that the class for two different $B_\bullet$ and $B_\bullet'$ coincides by constructing an explicit operator homotopy. We form the continuous paths 
$$\tilde{B}_{j,t}=tB_j+(1-t)B'_j, \quad \mbox{for $j=1, \ldots N$}.$$ 
It holds that 
$$\tilde{B}_{j,t}D_j+D_{j-1}\tilde{B}_{j-1,t}-1=t(B_{j}D_j+D_{j-1}B_{j-1}-1)+(1-t)(B'_{j}D_j+D_{j-1}B'_{j-1}-1),$$
which is compact since  $B_\bullet$ and $B_\bullet'$ are splittings. 

We set $B_{1,t}:=\tilde{B}_{1,t}$ and define the continuous paths of 
$$B_{j,t}:=(1-D_{j-1}\tilde{B}_{j-1,t})\tilde{B}_{j,t}.$$ 
Since $(\tilde{B}_{j,t}-B_{j,t})D_j$ and $D_j(\tilde{B}_{j,t}-B_{j,t})$ are compact for any $j$ and $t$, $B_{j,t}D_j+D_{j-1}B_{j-1,t}-1$ is compact for any $j$ and $t$. If we write $\sim$ for equality modulo compact operators, then for any $t$ it holds that
\begin{align*}
B_{j-1,t}B_{j,t}=&B_{j-1,t}(1-D_{j-1}\tilde{B}_{j-1,t})\tilde{B}_{j,t}\sim B_{j-1,t}(1-D_{j-1}B_{j-1,t})\tilde{B}_{j,t}\\
=&(1-B_{j-1,t}D_{j-1})B_{j-1,t}\tilde{B}_{j,t}\sim D_{j-2}B_{j-2,t}B_{j-1,t}\tilde{B}_{j,t}
\end{align*}
where in the second step we use that $D_{j-1}D_{j-2}\sim 0$ and in the fourth step that $B_{j-1,t}D_{j-1}+D_{j-2}B_{j-2,t}-1$ is compact. For $j=2$ the argument shows that $B_{1,t}B_{2,t}$ is compact and by induction we see that $B_{j-1,t}B_{j,t}$ is compact for all $j$.

We arrive at a collection $(B_{j,t})_{j=1}^{N}$ such that \eqref{chaiandad} is satisfied modulo compact operators for all $t$, producing the sought after operator homotopy. The statement that the class only depends on the symbol $\sigma_H(D_\bullet)$ follows by considering the operator homotopy defined from $D_{\bullet,t}=tD_\bullet+(1-t)D'_\bullet$. 
\end{proof}

\begin{definition}
When $G$ is compact, we arrive at a class in $KK_0^{G}(C(\mathfrak{X}),\C)$ of the cycle $(\pi,W^{\pmb{s}}_{H,{\rm fil}} (\mathfrak{X};\pmb{E}_\bullet),F)$ constructed in Theorem \ref{lknlknjkjaudn}. We denote it by $[D_\bullet]\in K_0^{G}(\mathfrak{X})$. 
\end{definition}

The class $[D_\bullet]$ was previously defined, using somewhat different yet equivalent methods, in \cite[Section 27]{goffkuz}.

\subsection{On existence of equivariant splittings for $\mathsf{G}/\mathsf{P}$}
\label{nogosubsec}

In the previous subsection we saw that if there are equivariant Rockland splittings we can define Fredholm modules with good equivariance properties. We also saw that for a compact group there is a constructive method of finding equivariant Heisenberg splittings, or in other words equivariant Rockland splittings in the Heisenberg calculus. We now study the existence question of equivariant Heisenberg splittings for BGG-sequences on the flat parabolic geometry $\mathfrak{X}=\mathsf{G}/\mathsf{P}$.

\begin{thm}
Let $\mathsf{G}$ be a connected semisimple Lie group, $\mathsf{P}$ a parabolic subgroup of $\mathsf{G}$ and $W$ a finite-dimensional $\mathsf{G}$-representation. If $\mathsf{G}/\mathsf{P}$ has rank $>1$ then the $\mathsf{G}$-equivariant Rockland sequence $\mathsf{BGG}_\bullet(W)$ on $\mathsf{G}/\mathsf{P}$ does not admit a $\mathsf{G}$-equivariant Heisenberg splitting. 
\end{thm}

\begin{proof}
If $B_\bullet$ is a $\mathsf{G}$-equivariant Heisenberg splitting then Corollary \ref{jnljnjknkjnad} shows that $B_\bullet$ in the leading term is a differential operator. This is clearly a contradiction, as for instance $D_1$ is a differential operator of positive order. 
\end{proof}

\begin{thm}
Let $\mathsf{G}$ be a connected semisimple Lie group, $\mathsf{P}$ a parabolic subgroup of $\mathsf{G}$ and $W$ a finite-dimensional $\mathsf{G}$-representation. If $G$ has rank $1$ then the $\mathsf{G}$-equivariant Rockland sequence $\mathsf{BGG}_\bullet(W)$ on $\mathsf{G}/\mathsf{P}$ admits a $\mathsf{G}$-equivariant Heisenberg splitting. 
\end{thm}

\begin{proof}
It suffices to find $\mathsf{G}$-equivariant $b_\bullet \in \Sigma^{-m_\bullet}_{H,{\rm gr}}(\mathsf{G}/\mathsf{P}; \mathpzc{H}_{\bullet}(W),\mathpzc{H}_{\bullet-1}(W))$ such that for any $j$
$$b_jd_j+d_{j-1}b_{j-1}=1 \quad \mbox{and}\quad b_{j-1}b_j=0.$$
However, Corollary \ref{jnljnjknkjnad} shows that 
$$\Sigma_{H,{\rm gr}}^m(\mathsf{G}/\mathsf{P};\mathpzc{H}_{\bullet}(W),\mathpzc{H}_{\bullet-1}(W))^\mathsf{G}\cong (\Sigma^m_H\overline{\mathsf{N}}_\mathsf{P}\otimes \Hom_{\mathsf{A}_\mathsf{P}}(H^\bullet(\mathfrak{p}_+,W)^{(m)},H^{\bullet-1}(\mathfrak{p}_+,W)))^{\mathsf{M}_\mathsf{P}},$$
this isomorphism is implemented by evaluating in $e\mathsf{P}\in \mathsf{G}/\mathsf{P}$. Since $\mathsf{G}$ is rank one, $\mathsf{M}_\mathsf{P}=\mathsf{P}\cap \mathsf{K}$ is compact. For degree reasons, we have a Heisenberg splitting whose value in $e\mathsf{P}$ belongs to $\Sigma^m_H\overline{\mathsf{N}}_\mathsf{P}\otimes \Hom_{\mathsf{A}_\mathsf{P}}(H^\bullet(\mathfrak{p}_+,W)^{(m)},H^{\bullet-1}(\mathfrak{p}_+,W))$ that we by averaging as in Proposition \ref{lknalkdna} can assume is $\mathsf{M}_\mathsf{P}$-invariant. Therefore there is a $\mathsf{G}$-equivariant Heisenberg splitting.
\end{proof}

\begin{remark}
The reader can note that non-existence of equivariant Heisenberg splittings is proved using that $\mathsf{A}_\mathsf{P}$ has rank $>1$, which is a property of $\mathsf{G}/\mathsf{P}$. However, existence of equivariant Heisenberg splittings is proved requiring not only that $\mathsf{A}_\mathsf{P}$ has rank $=1$ but that $\mathsf{M}_\mathsf{P}$ is compact so that $\mathsf{P}$ must be minimal, and as such we are imposing a rank condition on $\mathsf{G}$.
\end{remark}

\section{The index problem for curved BGG-sequences}
\label{sec:aojdnajdnadojn}

We are now going to study the $K$-theoretical invariants associated with a graded Rockland sequence. To do so, we set up the relevant equivariant Hilbert $C^*$-modules, after which we relate the $K$-theoretical invariants back to the $K$-homological invariants and prove Theorem \ref{mainthm3}. \emph{Throughout this section, we consider a compact group $G$ and a compact $G$–Carnot manifold $\mathfrak{X}$.}

\subsection{An intermezzo on Hilbert $C^*$-modules and $K$-theory}

First, we need some equivariant Hilbert $C^*$-modules fitting together with the symbol complex. Consider a $G$-equivariant graded Rockland sequence $D_\bullet$ acting on the collection of filtered Hermitean vector bundles $\pmb{E}_0,\pmb{E}_1,\ldots, \pmb{E}_N$ with unitary $G$-actions. We define the spaces
$$\mathpzc{E}^\infty_{j,c}:=C^\infty_c(\mathbb{T}_H\mathfrak{X};r^*\pmb{E}_j).$$
The space $\mathpzc{E}^\infty_{j,c}$ has a right $C^\infty_c(\mathbb{T}_H\mathfrak{X})$-action with a compatible $G$-action that we denote by $\pi_{\pmb{E}_j}$. The inner product on each bundle, induces an inner product 
$$\mathpzc{E}^\infty_{j,c}\times \mathpzc{E}^\infty_{j,c}\to C^\infty_c(\mathbb{T}_H\mathfrak{X}).$$

\begin{prop}
The completion $\mathpzc{E}_{j}$ of $\mathpzc{E}^\infty_{j,c}$ in the inner product to a $C^*(\mathbb{T}_H\mathfrak{X})$-Hilbert $C^*$-module is well defined. Moreover, the action of $G$ extends to a unitary continuous action of $G$ that we also denote by $\pi_{\pmb{E}_j}$. 
\end{prop}

Write $\pmb{E}_j^{\rm tr}$ for $\pmb{E}_j$ equipped with the trivial filtration. We fix $G$-invariant order reducing operators $\pmb{\Lambda}_{j,s}\in \pmb{\Psi}^s_{H,{\rm gr}}(\mathfrak{X};\pmb{E}_j,\pmb{E}_j^{\rm tr})$, i.e. $\pmb{\Lambda}_{j,s}$ is invertible and $\Lambda_{j,s}=\mathrm{ev}_{t=1}\pmb{\Lambda}_{j,s}$ is $H$-elliptic. We define the right $C^*(\mathbb{T}_H\mathfrak{X})$-Hilbert $C^*$-module 
$$\mathpzc{E}_{j}^s:=\pmb{\Lambda}_{j,s}^{-1}\mathpzc{E}_{j}.$$
We equip $\mathpzc{E}_{j}^s$ with the $G$-action defined from the $G$-action on $\mathpzc{E}^\infty_{j,c}$ and density of $\mathpzc{E}^\infty_{j,c}\subseteq \mathpzc{E}_{j}^s$. For $t>0$, using the isomorphism $C^*(\mathbb{T}_H\mathfrak{X}|_{\{t\}\times \mathfrak{X}})\cong \mathbb{K}(L^2(\mathfrak{X}))$, we have that 
$$\mathpzc{E}_{j}^s\otimes_{\mathrm{ev}_{t}} \mathbb{K}(L^2(\mathfrak{X}))\cong \mathbb{K}(W^{-s}_{H,{\rm fil}}(\mathfrak{X};\pmb{E}_j),L^2(\mathfrak{X})),$$
and $\mathpzc{E}_{j}^s\otimes_{\mathrm{ev}_{t=0}} C^*(T_H\mathfrak{X})$ is a completion of $C^\infty_c(T_H\mathfrak{X};r^*\pmb{E}_j)$ for a suitable Sobolev inner product. For more details, see \cite[Section 22]{goffkuz}.

We also need some further notions in $K$-theory. These ideas can be found in the non-equivariant setting in \cite[Section 27]{goffkuz}, see also \cite{HRSurI}. 

\begin{definition}
\label{fredseq}
Consider a $G-C^*$-algebra $B$ and two sequences of $G$-equivariant $B$-Hilbert $C^*$-modules 
\begin{align*}
0\to\mathpzc{E}_1\xrightarrow{d_1}\mathpzc{E}_2\xrightarrow{d_2}\cdots \to \mathpzc{E}_N\xrightarrow{d_N}\mathpzc{E}_{N+1}\to 0,\\
0\leftarrow\mathpzc{E}_1\xleftarrow{b_1}\mathpzc{E}_2\xleftarrow{b_2}\cdots \leftarrow \mathpzc{E}_N\xleftarrow{b_N}\mathpzc{E}_{N+1}\leftarrow 0,
\end{align*}
where all maps are adjointable. The data $(\mathpzc{E}_\bullet,d_\bullet,b_\bullet)$ is said to form a $G$-equivariant $B$-Fredholm sequence of length $N+1$ if  
$$\begin{cases}
d_jd_{j-1}\in \mathbb{K}_B(\mathpzc{E}_{j-1},\mathpzc{E}_{j+1}),\\
b_{j-1}b_{j}\in \mathbb{K}_B(\mathpzc{E}_{j+1},\mathpzc{E}_{j-1}),\\
b_jd_j+d_{j-1}b_{j-1}-1\in \mathbb{K}_B(\mathpzc{E}_j),\\
\pi(g)d_j\pi(g)^{-1}-d_j\in \mathbb{K}_B(\mathpzc{E}_{j-1},\mathpzc{E}_j), \\
\pi(g)b_j\pi(g)^{-1}-b_j\in \mathbb{K}_B(\mathpzc{E}_{j},\mathpzc{E}_{j-1}),
\end{cases}$$
for all $j$ and $g\in G$. We use the convention that  $b_j=0$ and $d_j=0$ for $j\notin\{1,\ldots, N+1\}$. To simplify notations, we sometimes write 
$$0\to\mathpzc{E}_1{}_{d_1}\leftrightarrows^{b_1}\mathpzc{E}_2{}_{d_2}\leftrightarrows^{b_2}\cdots \leftrightarrows \mathpzc{E}_N{}_{d_N}\leftrightarrows^{b_N}\mathpzc{E}_{N+1}\to 0,$$ 
instead of $(\mathpzc{E}_\bullet,d_\bullet,b_\bullet)$.

An acyclic $B$-Fredholm sequence is one where 
$$\begin{cases}
d_jd_{j-1}=0,\\
b_{j-1}b_{j}=0,\\
b_jd_j+d_{j-1}b_{j-1}=1,\\
\pi(g)d_j\pi(g)^{-1}=d_j, \\
\pi(g)b_j\pi(g)^{-1}=b_j,
\end{cases}$$
for all $j$ and $g\in G$.

\end{definition}

As in \cite[Section 27]{goffkuz}, we can define isomorphism and homotopies of $G$-equivariant $B$-Fredholm sequences. We make the following definition.

\begin{definition}
For a $G-C^*$-algebra $B$, let $\tilde{K}_0^G(B)$ denote the set of equivalence classes of $G$-equivariant $B$-Fredholm sequence under the equivalence relation generated by homotopy and declaring acyclic $G$-equivariant $B$-Fredholm sequences equivalent to the zero complex. 
\end{definition}

With the same ideas as in \cite[Section 27]{goffkuz}, we have the following result.

\begin{thm}
For a $G-C^*$-algebra $B$, there is a natural isomorphism
$$
KK_0^{G}(\C,B)\xrightarrow{\sim} \tilde{K}_0^{G}(B),$$
defined on cycles by 
$$ \left[\left(\begin{matrix}\mathpzc{E}_+\\\oplus\\ \mathpzc{E}_-\end{matrix}, \begin{pmatrix} 0& F_+\\ F_- &0\end{pmatrix}\right)\right]\mapsto \left[\left(0\to\mathpzc{E}_-{}_{F_+}\leftrightarrows^{F_-}\mathpzc{E}_+\to 0\right)\right].$$
\end{thm}

\subsection{The index class in $K$-theory}

The index class of a  $G$-equivariant graded Rockland sequence will be defined in terms of $G$-equivariant $B$-Fredholm sequences. Write 
$$\mathpzc{E}_{j,t=0}^s:=\mathpzc{E}^s_j\otimes_{\mathrm{ev}_{t=0}} C^*(T_H\mathfrak{X}).$$
Recall the notation $\pmb{\Psi}^{m}_{H}(\mathfrak{X};E_1,E_2)$ from Definition \ref{liftotflknkjnad}, with its obvious filtered generalization denoted by $\pmb{\Psi}^{m_j}_{H,{\rm gr}}(\mathfrak{X};\pmb{E}_1,\pmb{E}_2)$.

\begin{lemma}
\label{liftsintind}
Consider a $G$-equivariant graded Rockland sequence
$$
D_\bullet: \quad 0\to C^\infty(\mathfrak{X};\pmb{E}_0)\xrightarrow{D_1}C^\infty(\mathfrak{X};\pmb{E}_1)\xrightarrow{D_2}\cdots \xrightarrow{D_N}C^\infty(\mathfrak{X};\pmb{E}_N)\to 0,
$$
of order $\pmb{m}=(m_1,\ldots, m_N)$ and take $\pmb{s}=(s_0,\ldots, s_N)\in \R^{N+1}$ such that 
$$s_{j-1}-s_{j}=m_j, \quad j=1,\ldots, N.$$

Choose a $G$-equivariant Heisenberg splitting $B_\bullet \in \Psi^{-m_j}_{H,{\rm gr}}(\mathfrak{X};\pmb{E}_{j},\pmb{E}_{j-1})$ as well as lifts $\pmb{D}_\bullet\in \pmb{\Psi}^{m_j}_{H,{\rm gr}}(\mathfrak{X};\pmb{E}_{j-1},\pmb{E}_j)$ of $D_\bullet$ and $\pmb{B}_\bullet\in \pmb{\Psi}^{-m_j}_{H,{\rm gr}}(\mathfrak{X};\pmb{E}_{j-1},\pmb{E}_j)$ of $B_\bullet$. Then the sequence
$$
\pmb{D}_\bullet: \quad 0\to \mathpzc{E}^{s_0}_{0}{}_{\pmb{D}_N}\leftrightarrows^{\pmb{B}_1}\mathpzc{E}^{s_1}_{j}{}_{\pmb{D}_2}\leftrightarrows^{\pmb{B}_2}\cdots {}_{\pmb{D}_N}\leftrightarrows^{\pmb{B}_N}\mathpzc{E}^{s_N}_{j}\to 0,
$$
is a $G$-equivariant $C^*(\mathbb{T}_H\mathfrak{X}|_{[0,1]\times \mathfrak{X}})$-Fredholm complex that we denote by $\pmb{D}_\bullet$. Up to homotopy, the Fredholm complex $\pmb{D}_\bullet$ depends only on the graded principal symbol $\sigma_{H,{\rm gr}}(D_\bullet)$ and is independent of the choice of $\pmb{s}$, $\pmb{D}_\bullet$ and $\pmb{B}_\bullet$.
\end{lemma}

\begin{proof}
By assumption $D_{j-1}B_{j-1}+B_jD_j-1$ is order $-1$. Therefore the homogeneity modulo $C^\infty_c$ implies that $\pmb{D}_{j-1}\pmb{B}_{j-1}+\pmb{B}_j\pmb{D}_j-1$ acts on $\mathpzc{E}_j$ as an element from $t\pmb{\Psi}^{-1}_{H,{\rm gr}}(\mathfrak{X};\pmb{E}_j,\pmb{E}_j)$. The latter space acts by compact operators on $\mathpzc{E}_j$ by standard arguments using the symbol estimate of \cite[Corollary 41]{vanerpyuncken}. Similarly, $\pmb{B}_{j-1}\pmb{B}_{j}$ and $\pmb{D}_j\pmb{D}_{j-1}$ act as compact operators. The lemma follows. 
\end{proof}

\begin{remark}
\label{lfiininfiafnain}
The reader can recall the notation $\tilde{\Sigma}^{m}_{H}(\mathfrak{X};{E}_1,{E}_2)$ from Theorem \ref{lkjandlkjandkjadn}. We also write $\tilde{\Sigma}^{m}_{H,{\rm gr}}(\mathfrak{X};\pmb{E}_1,\pmb{E}_2)$ for its analogously defined filtered version. In the setting of Lemma \ref{liftsintind}, the choice of $\pmb{D}_\bullet$ fixes lifts
$$\tilde{d}_j:=\mathrm{ev}_{t=0}\pmb{D}_j\in \tilde{\Sigma}^{m_j}_{H,{\rm gr}}(\mathfrak{X};\pmb{E}_{j-1},\pmb{E}_j),$$ 
of the symbols $d_\bullet=\sigma_{H,{\rm gr}}(D_\bullet)$, and the choice of $\pmb{B}_\bullet$ fixes lifts 
$$\tilde{b}_j:=\mathrm{ev}_{t=0}\pmb{B}_j\in \tilde{\Sigma}^{-m_j}_{H,{\rm gr}}(\mathfrak{X};\pmb{E}_{j},\pmb{E}_{j-1}),$$
of the symbols $b_\bullet=\sigma_{H,{\rm gr}}(B_\bullet)$. Lemma \ref{liftsintind} implies that the sequence
$$0\to \mathpzc{E}^{s_0}_{j,t=0}{}_{\tilde{d}_1}\leftrightarrows^{\tilde{b}_1}\mathpzc{E}^{s_1}_{j,t=0}{}_{\tilde{d}_2}\leftrightarrows^{\tilde{b}_2}\cdots {}_{\tilde{d}_N}\leftrightarrows^{\tilde{b}_N}\mathpzc{E}^{s_N}_{j,t=0}\to 0,$$
is a $G$-equivariant $C^*(T_H\mathfrak{X})$-Fredholm complex that we denote by $\sigma_{H,{\rm gr}}(D_\bullet)$. 
\end{remark}

\begin{definition}
We write $[\sigma_{H}(D_\bullet)]\in K_0^G(C^*(T_H\mathfrak{X}))$ for the class of the $C^*(T_H\mathfrak{X})$-Fredholm complex $\sigma_{H,{\rm gr}}(D_\bullet)$ for some $\pmb{s}$.
\end{definition}

From Remark \ref{lfiininfiafnain}, we conclude that 
$$\sigma_{H,{\rm gr}}(D_\bullet)=\pmb{D}_\bullet\otimes {\mathrm{ev}_{t=0}} C^*(T_H\mathfrak{X}).$$
The last identity leads us to the next lemma.

\begin{prop}
\label{proplfiininfiafnain}
The class $[\pmb{D}_\bullet]\in \tilde{K}_0^{G}(C^*(\mathbb{T}_H\mathfrak{X}|_{[0,1]\times \mathfrak{X}}))$ satisfies 
$$\tilde{K}_0(\mathrm{ev}_{t=0})[\pmb{D}_\bullet]=[\sigma_{H,{\rm gr}}(D_\bullet)],$$
and therefore $[\pmb{D}_\bullet]$ lifts the class $[\sigma_{H,{\rm gr}}(D_\bullet)]\in \tilde{K}_0^{G}(C^*(T_H\mathfrak{X}))$ under the natural map
$$\tilde{K}_0(\mathrm{ev}_{t=0}):\tilde{K}_0^{G}(C^*(\mathbb{T}_H\mathfrak{X}|_{[0,1]\times \mathfrak{X}}))\to \tilde{K}_0^{G}(C^*(T_H\mathfrak{X})).$$
\end{prop}

\subsection{Poincaré duality map}

To relate the symbol classes back to operators, we make use of an analytic Poincaré duality. This duality was first introduced for contact manifolds in \cite{baumvanerp} and utilized on Carnot manifolds in \cite{goffkuz,mohsenind}. Let $\partial_H\in KK^{G}(C^*(T_H\mathfrak{X}),\C)$ denote the class of the boundary map associated with 
$$0\to S\mathbb{K}\to C^*(\mathbb{T}_H\mathfrak{X}|_{[0,1)})\xrightarrow{\mathrm{ev}_{t=0}}C^*(T_H\mathfrak{X})\to 0,$$
that uses the identification $S\mathbb{K}\cong C^*(\mathbb{T}_H\mathfrak{X}|_{(0,1)})=C^*(\mathfrak{X}\times \mathfrak{X}\times (0,1)\rightrightarrows \mathfrak{X}\times (0,1))$.

\begin{definition}
The analytic Poincaré duality map 
$$\mathsf{PD}^{\rm an}_H:K_0^{G}(C^*(T_H\mathfrak{X}))\to K_0^{G}(\mathfrak{X}),$$
is defined as the composition
\begin{align*}
K_0^{G}(C^*(T_H\mathfrak{X}))=KK_0^{G}(\C,C^*(T_H\mathfrak{X}))&\xrightarrow{\alpha_\mathfrak{X}} KK_0^{G}(C(\mathfrak{X}),C^*(T_H\mathfrak{X}))\\
&\xrightarrow{\partial_H}KK_0^{G}(C(\mathfrak{X}),\C)=K_0^{G}(\mathfrak{X}).
\end{align*}
Here $\alpha_\mathfrak{X}$ denotes inflation by $C(\mathfrak{X})$ using that $C(\mathfrak{X})$ acts as central multipliers of $C^*(T_H\mathfrak{X})$. For the trivial Carnot structure, we use the notation 
$$\mathsf{PD}^{\rm an}:K_0^{G}(C^*(T\mathfrak{X}))=K^0_G(T^*\mathfrak{X})\to K_0^{G}(\mathfrak{X}),$$
\end{definition}

\begin{prop}
\label{bnkbhb}
Let $\mathfrak{X}$ be a compact $G$-Carnot manifold and $D_\bullet$ a $G$-equivariant graded Rockland sequence on $\mathfrak{X}$. Then it holds that 
$$\mathsf{PD}^{\rm an}_H[\sigma_H(D_\bullet)]=[D_\bullet]\in K_0^{G}(\mathfrak{X}).$$
\end{prop}

The proposition follows from Proposition \ref{proplfiininfiafnain} and the observation that the boundary map $\partial_H$ has as its defining property that $\tilde{K}_0(\mathrm{ev}_{t=0})\otimes_{C^*(T_H\mathfrak{X})}\partial_H=\tilde{K}_0(\mathrm{ev}_{t=1})$. 

To relate the index theory in the Heisenberg calculus to ordinary $K$-theory, we use Nistor's Connes-Thom isomorphism \cite{nistorsolvabel}. Following \cite{nistorsolvabel}, we define the adiabatic groupoid 
$$A_H\mathfrak{X}:=TX\times \{0\}\, \dot{\bigcup}\,  T_H\mathfrak{X}\times (0,\infty)\rightrightarrows  \mathfrak{X}\times [0,\infty),$$
with smooth structure defined from declaring the map 
\begin{equation}
\label{aadlan}
\psi_{\rm adb}: T\mathfrak{X}\times [0,\infty)\to A_H\mathfrak{X}, \quad (x,v,t)\mapsto 
\begin{cases}
(x,v,0), \; &t=0,\\
(x,\mathrm{exp}(tv)),t), \; &t>0,
\end{cases} 
\end{equation}
to be a diffeomorphism. Here we are implicitly using a choice of isomorphism $T\mathfrak{X}\cong \mathfrak{t}_H\mathfrak{X}$.
We can define $\psi\in KK_0^G(C_0(T^*\mathfrak{X}),C^*(T_H\mathfrak{X}))$ from the isomorphisms $C_0(T^*\mathfrak{X})\cong C^*(T\mathfrak{X})$, $SC^*(T_H\mathfrak{X})\cong C^*(A_H\mathfrak{X}|_{(0,\infty)})$ and the short exact sequence 
$$0\to SC^*(T_H\mathfrak{X})\to C^*(A_H\mathfrak{X})\xrightarrow{\mathrm{ev}_{t=0}} C_0(T^*\mathfrak{X})\to 0.$$
The latter defines an extension class in $KK_1^G(C_0(T^*\mathfrak{X}),SC^*(T_H\mathfrak{X}))$. The next theorem can be found in \cite[Theorem 5.4.2]{baumvanerp} for contact manifolds, that later was extended ad verbatim in \cite[Theorem 27.16]{goffkuz} and \cite{mohsenind} to general compact Carnot manifolds. In these references, the result is stated for the trivial group but extends readily to compact group actions.

\begin{thm}
\label{aldknadjn}
Let $\mathfrak{X}$ be a compact $G$-Carnot manifold. Then the following diagram commutes
\[
\begin{tikzcd}
& K_*^{G}(C^*(T_H \mathfrak{X}))\arrow[ddr, "\mathsf{PD}_H^{\rm an}"] & \\
&&\\
 K^*_{G}(T^* \mathfrak{X}) \arrow[rr, "\mathsf{PD}^{\rm an}"] \arrow[uur, "\psi"] &&K_*^{G}(\mathfrak{X}) 
\end{tikzcd}
\]
\end{thm}

\subsection{The index class and BGG-compression}

Above in Subsection \ref{seconrock} we saw how equivariant Rockland sequences gave rise to elements of $K$-homology and in Subsection \ref{lknaldkna} how filtered connections can be compressed into a curved BGG-sequence. We shall see that the $K$-homology class is not altered by BGG-compression. As explained in \cite{Dave_Haller1}, the procedure of compressing to a BGG-sequence requires only the existence of a Kostant differential on the Lie algebroid $\mathfrak{t}_H\mathfrak{X}$, but for simplicity we restrict to parabolic geometries with an action of a compact group $G$. We summarize the context of this subsection into an assumption that will stand for the remainder of the subsection.

\begin{ass}
\label{ljknalknad}
We let $\mathfrak{X}$ denote a compact parabolic geometry $\mathfrak{X}$ with an action of a compact group $G\to \Aut(\mathfrak{X},\omega)$, and $\pmb{E}\to\mathfrak{X}$ a filtered, unitarily $G$-equivariant tractor bundle equipped with a $G$-invariant filtered connection  $\nabla_{\pmb{E}}$ of degree $0$ (see definition in Subsection \ref{generalpara}) with curvature in degree $-1$. If $\mathfrak{X}$ is a complex parabolic geometry, we also assume that $\nabla_{\pmb{E}}$ is a filtered complex connection.
\end{ass}

We remark that in Assumption \ref{ljknalknad}, a filtered tractor bundle automatically carries a filtered $G$-action but we additionally assume that $\pmb{E}$ has a Hermitean structure in which the action is unitary. An example satisfying Assumption \ref{ljknalknad} is the flat parabolic geometry $\mathfrak{X}=\mathsf{G}/\mathsf{P}$ and $G\subseteq \mathsf{G}$ a compact subgroup, and , for $W$ a $\mathsf{G}$-representation, $\pmb{E}=\mathsf{G}\times_{\mathsf{P}}W\to\mathfrak{X}$ a homogeneous bundle filtered by the weights of its generic fibres with the induced connection as in Subsection \ref{generalpara}. 

It follows from Assumption \ref{ljknalknad} that the classical symbol complex 
\begin{align}
\label{lnlknad}
\sigma(\nabla_{\pmb{E},\bullet}): \quad 0\to \mathcal{S}(T^*\mathfrak{X};\pmb{E})\xrightarrow{\sigma(\nabla_{\pmb{E}})}\mathcal{S}(T^*\mathfrak{X};&T^*\mathfrak{X}\otimes\pmb{E})\xrightarrow{\sigma(\nabla_{\pmb{E}})}\cdots \\
\nonumber
&\cdots\xrightarrow{\sigma(\nabla_{\pmb{E}})}\mathcal{S}(T^*\mathfrak{X};\wedge^nT^*\mathfrak{X}\otimes \pmb{E})\to 0,
\end{align}
defines a $G$-equivariant elliptic complex on $T^*\mathfrak{X}$ and defines a class 
$$[\sigma(\nabla_{\pmb{E}})]\in K^0_G(T^*\mathfrak{X}).$$ 
Of course, much less is needed than Assumption  \ref{ljknalknad} but we restrict to this situation as we can then proceed with modifying the construction to a BGG-sequence. Moreover, still under Assumption  \ref{ljknalknad} but with $\mathfrak{X}$ a complex $G$-Carnot manifold and $\nabla_{\pmb{E}}$ is a filtered complex connection of degree $0$, the classical symbol complex of the associated $\bar{\partial}$-sequence
\begin{align}
\label{lnlknadcomplex}
\sigma(\bar{\partial}_{\pmb{E},\bullet}): \quad 0\to \mathcal{S}(T^*\mathfrak{X};\pmb{E})\xrightarrow{\sigma(\bar{\partial}_{\pmb{E}})}\mathcal{S}(T^*\mathfrak{X};&\wedge^{0,1}T^*\mathfrak{X}\otimes\pmb{E})\xrightarrow{\sigma(\bar{\partial}_{\pmb{E}})}\cdots \\
\nonumber
&\cdots\xrightarrow{\sigma(\bar{\partial}_{\pmb{E}})}\mathcal{S}(T^*\mathfrak{X};\wedge^{0,n}T^*\mathfrak{X}\otimes\pmb{E})\to 0,
\end{align}
defines a $G$-equivariant elliptic complex on $T^*\mathfrak{X}$ and defines a class that we denote by
$$[\sigma(\bar{\partial}_{\pmb{E},\bullet})]\in K^0_G(T^*\mathfrak{X}).$$ 
In both the real and complex case, the symbol does not depend on the choice of connection $\nabla_{\pmb{E}}$ so neither do the cycles $\sigma(\nabla_{\pmb{E},\bullet})$ or $\sigma(\bar{\partial}_{\pmb{E},\bullet})$ nor the associated classes in $K^0_G(T^*\mathfrak{X})$. 

We can also form the graded Heisenberg symbol complex 
\begin{align*}
\sigma_{H,{\rm gr}}(\nabla_{\pmb{E},\bullet}): \quad 0\to \mathcal{S}(T_H\mathfrak{X};\mathrm{gr}(\pmb{E}))\xrightarrow{\sigma_{H,{\rm gr}}(\nabla_{\pmb{E}})}\mathcal{S}(T_H\mathfrak{X};&\mathrm{gr}(T^*\mathfrak{X}\otimes\pmb{E}))\xrightarrow{\sigma_{H,{\rm gr}}(\nabla_{\pmb{E}})}\cdots \\
&\cdots\xrightarrow{\sigma_{H,{\rm gr}}(\nabla_{\pmb{E}})}\mathcal{S}(T_H\mathfrak{X};\mathrm{gr}(\wedge^nT^*\mathfrak{X}\otimes \pmb{E}))\to 0.
\end{align*}
This sequence is a complex since the curvature has filtered degree $-1$ and is the symbol complex of a Rockland sequence by \cite[Lemma 4.15 and 4.16]{Dave_Haller1}. The sequence therefore completes to a $G$-equivariant Fredholm complex over $C^*(T_H\mathfrak{X})$ and defines a class 
$$[\sigma_{H,{\rm gr}}(\nabla_{\pmb{E}})]\in K_0^G(C^*(T_H\mathfrak{X})).$$ 
Moreover, if $\mathfrak{X}$ is a complex $G$-Carnot manifold and $\nabla_{\pmb{E}}$ is a filtered complex connection of degree $0$, the symbol complex of the associated $\bar{\partial}$-sequence
\begin{align*}
\sigma_{H,{\rm gr}}(\bar{\partial}_{\pmb{E},\bullet}): \quad 0\to \mathcal{S}(T_H\mathfrak{X};\mathrm{gr}(\pmb{E}))\xrightarrow{\sigma_{H,{\rm gr}}(\bar{\partial}_{\pmb{E}})}\mathcal{S}(T_H\mathfrak{X};&\mathrm{gr}(\wedge^{0,1}T^*\mathfrak{X}\otimes\pmb{E}))\xrightarrow{\sigma_{H,{\rm gr}}(\bar{\partial}_{\pmb{E}})}\cdots \\
&\cdots\xrightarrow{\sigma_{H,{\rm gr}}(\bar{\partial}_{\pmb{E}})}\mathcal{S}(T_H\mathfrak{X};\mathrm{gr}(\wedge^{0,n}T^*\mathfrak{X}\otimes\pmb{E}))\to 0,
\end{align*}
This sequence is complex and the symbol complex of a Rockland sequence by the same argument as above. The sequence completes to a $G$-equivariant Fredholm complex over $C^*(T_H\mathfrak{X})$ and defines a class 
$$[\sigma_{H,{\rm gr}}(\bar{\partial}_{\pmb{E},\bullet})]\in K_0^G(C^*(T_H\mathfrak{X})).$$
In both the real and complex case, the graded Heisenberg symbol depends on the choice of connection $\nabla_{\pmb{E}}$ but the cycles $\sigma_{H,{\rm gr}}(\nabla_{\pmb{E},\bullet})$ or $\sigma_{H,{\rm gr}}(\bar{\partial}_{\pmb{E},\bullet})$ are independent of the connection up to homotopy and the associated classes in equivariant $K$-theory are independent of connection.

\begin{lemma}
\label{comaodmoadm}
Under Assumption \ref{ljknalknad}, we have the equality
$$[\sigma_{H,{\rm gr}}(\nabla_{\pmb{E}})]=[\sigma_{H,{\rm gr}}(\mathsf{BGG}_\bullet(\nabla_{\pmb{E}}))]\in K_0^G(C^*(T_H\mathfrak{X})).$$
\end{lemma}

\begin{proof}
By the definition of the BGG-sequence (see Definition \ref{defbggcomp}) we can split the symbol complex $\sigma_{H,{\rm gr}}(\nabla_{\pmb{E}}))$ as
$$\sigma_{H,{\rm gr}}(\nabla_{\pmb{E}})=\sigma_{H,{\rm gr}}(L_\bullet)(\sigma_{H,{\rm gr}}(\mathsf{BGG}_\bullet(\nabla_{\pmb{E}}))\oplus \mathfrak{b}_\bullet)\sigma_{H,{\rm gr}}(L_\bullet^{-1})$$ 
for the operators $L_\bullet$ from \eqref{lknaldnalkdjnad} and a symbol complex $\mathfrak{b}_\bullet$ which lifts to a nullhomotopic complex over $\tilde{\Sigma}_H^*$ by item (3) of Lemma \ref{decomkdh}. We conclude that the class of $\mathfrak{b}_\bullet$ in $\tilde{K}_0^G(C^*(T_H\mathfrak{X})$ vanishes and the lemma follows.
\end{proof}

The next result is the main technical result of this section and key in computing the $K$-homology class of curved BGG-sequences.

\begin{thm}
\label{lknlknad}
Under Assumption \ref{ljknalknad} we have the equality
$$[\sigma_{H,{\rm gr}}(\mathsf{BGG}_\bullet(\nabla_{\pmb{E}}))]=\psi[\sigma(\nabla_{\pmb{E}})]\in K_0^{G}(C^*(T_H\mathfrak{X})),$$
and in the complex case 
$$[\sigma_{H,{\rm gr}}(\mathsf{BGG}_\bullet(\nabla_{\pmb{E}}))]=\psi[\sigma(\bar{\partial}_{\pmb{E}})]\in K_0^{G}(C^*(T_H\mathfrak{X})),$$
\end{thm}

Before proving Theorem \ref{lknlknad}, we need some further setup and a lemma. We give the proof in the real case, the complex case goes analogously. We organize the setup into the following four observations.\\

{\bf Observation 1.} We first make a simplifying observation about the right hand side of Theorem \ref{lknlknad}. The classical symbol complex $\sigma(\nabla_{\pmb{E},\bullet})$ is constructed from the $G$-equivariant complex 
\begin{align*}
\sigma(\nabla_{\pmb{E},\bullet}): \quad 0\to \mathcal{S}(T\mathfrak{X};\pmb{E})\xrightarrow{\sigma(\nabla_{\pmb{E}})}\mathcal{S}(T\mathfrak{X};&T^*\mathfrak{X}\otimes\pmb{E})\xrightarrow{\sigma(\nabla_{\pmb{E}})}\cdots \\
&\cdots\xrightarrow{\sigma(\nabla_{\pmb{E}})}\mathcal{S}(T\mathfrak{X};\wedge^nT^*\mathfrak{X}\otimes \pmb{E})\to 0,
\end{align*}
of $\mathcal{S}(T\mathfrak{X})$-modules. Here we write $\sigma(\nabla_{\pmb{E}})$ for the symbol as constructed in van Erp-Yuncken's work \cite{vanerpyuncken}, compatibly with our conventions in the Heisenberg calculus. The complex $\sigma(\nabla_{\pmb{E}})$ can under the Fourier transform $\mathcal{S}(T\mathfrak{X})\cong \mathcal{S}(T^*\mathfrak{X})$ be viewed as a $G$-equivariant elliptic complex on $T^*\mathfrak{X}$ in the $K$-theoretical sense \cite[Section 2]{asI} and defines a class $[\sigma(\nabla_{\pmb{E}})]\in K^0_G(T^*\mathfrak{X})$. 

We can equally well use the classical graded symbol complex $\sigma_{\rm gr}(\nabla_{\pmb{E},\bullet})$ constructed from the $G$-equivariant complex 
\begin{align*}
\sigma_{\rm gr}(\nabla_{\pmb{E},\bullet}): \quad 0\to \mathcal{S}(\mathfrak{t}_H\mathfrak{X};\mathrm{gr}(\pmb{E}))\xrightarrow{\mathfrak{d}_{1,0}}\mathcal{S}(\mathfrak{t}_H\mathfrak{X};&\mathrm{gr}(T^*\mathfrak{X}\otimes\pmb{E}))\xrightarrow{\mathfrak{d}_{2,0}}\cdots \\
&\cdots\xrightarrow{\mathfrak{d}_{n,0}}\mathcal{S}(\mathfrak{t}_H\mathfrak{X};\mathrm{gr}(\wedge^nT^*\mathfrak{X}\otimes \pmb{E}))\to 0,
\end{align*}
of $\mathcal{S}(\mathfrak{t}_H\mathfrak{X})$-modules. Here we write $\mathfrak{d}_{k,0}:=\sigma_{\rm gr}(\nabla_{\pmb{E},k})$ for the graded anisotropic symbol formed with respect to the grading on $\mathfrak{t}_H\mathfrak{X}$. We shall later extend $\mathfrak{d}_{k,0}$ to a family of differentials $(\mathfrak{d}_{k,t})_{t\in [0,1]}$. Again, the complex $\sigma_{\rm gr}(\nabla_{\pmb{E}})$ can be viewed as a $G$-equivariant elliptic complex on $\mathfrak{t}_H^*\mathfrak{X}\cong T^*\mathfrak{X}$ and defines a class $[\sigma_{\rm gr}(\nabla_{\pmb{E}})]\in K^0_G(T^*X)$. Up to splitting isomorphisms (in the sense of Subsection \ref{subsecgraded}), the two elliptic complexes $\sigma(\nabla_{\pmb{E},\bullet})$ and $\sigma_{\rm gr}(\nabla_{\pmb{E},\bullet})$ coincide and therefore 
\begin{equation}
\label{gradedad}
[\sigma(\nabla_{\pmb{E},\bullet})]=[\sigma_{\rm gr}(\nabla_{\pmb{E},\bullet})]\in K^0_G(T^*\mathfrak{X}).
\end{equation}
\\

{\bf Observation 2.} We will make use of an analytically more refined description of the $K$-theory class defined by $\sigma_{\rm gr}(\nabla_{\pmb{E},\bullet})$ as a $G$-equivariant $C^*(\mathfrak{t}_H\mathfrak{X})$-Fredholm sequence in the sense of Definition \ref{fredseq}. Since $\sigma_{\rm gr}(\nabla_{\pmb{E},\bullet})$ is graded anisotropically elliptic, an argument as in the proof of Lemma \ref{splittinprop} and Proposition \ref{lknalkdna} implies that there are $G$-equivariant $\mathfrak{b}_{j,0}\in \mathcal{E}_r'(\mathfrak{t}_H\mathfrak{X};\mathrm{gr}(\wedge^jT^*\mathfrak{X}\otimes \pmb{E}),\mathrm{gr}(\wedge^{j-1}T^*\mathfrak{X}\otimes \pmb{E}))$ that are graded essentially homogeneous of degree $0$ for the dilation action on $\mathfrak{t}_H\mathfrak{X}$ and satisfies that 
$$\begin{cases}
\mathfrak{b}_{j,0}\mathfrak{d}_{j,0}+\mathfrak{d}_{j-1,0}\mathfrak{b}_{j-1,0}-1\in C^\infty_c(\mathfrak{t}_H\mathfrak{X};\mathrm{gr}(\wedge^{j-1}T^*\mathfrak{X}\otimes \pmb{E}))\\
\mathfrak{b}_{j-1,0}\mathfrak{b}_{j,0}\in C^\infty_c(\mathfrak{t}_H\mathfrak{X};\mathrm{gr}(\wedge^jT^*\mathfrak{X}\otimes \pmb{E}),\mathrm{gr}(\wedge^{j-2}T^*\mathfrak{X}\otimes \pmb{E})).
\end{cases}$$

For $j=0,\ldots, n$, we write $C^*(\mathfrak{t}_H\mathfrak{X};\mathrm{gr}(\wedge^jT^*\mathfrak{X}\otimes \pmb{E}))$ for the right  $G$-equivariant $C^*(\mathfrak{t}_H\mathfrak{X})$-module obtained from completing $\mathcal{S}(\mathfrak{t}_H\mathfrak{X};\mathrm{gr}(\wedge^jT^*\mathfrak{X}\otimes \pmb{E}))$ in the right $C^*(\mathfrak{t}_H\mathfrak{X})$-valued inner product obtained from the Hermitean structure on $\mathrm{gr}(\wedge^jT^*\mathfrak{X}\otimes \pmb{E})$. Proceeding as in Section \ref{subsec:heis}, we can for $j=0,\ldots n$ construct  $G$-equivariant, invertible order reducing operators $\lambda_{j,0}\in \mathcal{E}_r'(\mathfrak{t}_H\mathfrak{X};\mathrm{gr}(\wedge^jT^*\mathfrak{X}\otimes \pmb{E}))$ that are graded essentially homogeneous of degree $0$ for the dilation action on $\mathfrak{t}_H\mathfrak{X}$, and graded anisotropically elliptic. Moreover, the operators 
$$\lambda_{j-1,0}\mathfrak{b}_{j,0}\lambda_{j,0} \quad\mbox{and}\quad \lambda_{j+1,0}\mathfrak{d}_{j+1,0}\lambda_{j,0}$$
extend to bounded, adjointable operators $C^*(\mathfrak{t}_H\mathfrak{X};\mathrm{gr}(\wedge^jT^*\mathfrak{X}\otimes \pmb{E}))\to C^*(\mathfrak{t}_H\mathfrak{X};\mathrm{gr}(\wedge^{j-1}T^*\mathfrak{X}\otimes \pmb{E}))$, and $C^*(\mathfrak{t}_H\mathfrak{X};\mathrm{gr}(\wedge^jT^*\mathfrak{X}\otimes \pmb{E}))\to C^*(\mathfrak{t}_H\mathfrak{X};\mathrm{gr}(\wedge^{j+1}T^*\mathfrak{X}\otimes \pmb{E}))$, respectively. We can then form the $G$-equivariant $C^*(\mathfrak{t}_H\mathfrak{X})$-modules 
$$\mathfrak{E}_{j,0}:=\lambda_{j,0}C^*(\mathfrak{t}_H\mathfrak{X};\mathrm{gr}(\wedge^jT^*\mathfrak{X}\otimes \pmb{E})),$$ 
for $j=0,1,\ldots, n$. We arrive at the $G$-equivariant $C^*(\mathfrak{t}_H\mathfrak{X})$-Fredholm sequence $(\mathfrak{E}_{\bullet,0},\mathfrak{d}_{\bullet,0},\mathfrak{b}_{\bullet,0})$ representing the elliptic complex $\sigma_{\rm gr}(\nabla_{\pmb{E},\bullet})$, and we shall henceforth identify the two.\\

{\bf Observation 3.} We can write $A_H\mathfrak{X}=(\mathcal{V}/\mathsf{P}_+)\times_{\mathsf{G}_0} \mathfrak{A}_H$ where $\mathfrak{A}_H\to [0,\infty)$ is the fibrewise adiabatic deformation groupoid $\mathfrak{A}_H= (\mathfrak{n}\times \{0\})\coprod (\mathsf{N}\times (0,\infty))$. The Lie algebroid of $\mathfrak{A}_H\to [0,\infty)$ is $\mathfrak{n}\times [0,\infty)$ but its bracket at $t\in [0,\infty)$ is given by $t[\cdot,\cdot]_{\mathfrak{n}}$ for $t>0$ and trivial at $t=0$. We extend the dilation action $\R_+\ni \lambda\mapsto \delta_\lambda$ on $\mathfrak{n}$ to an action on $\mathfrak{A}_H$ to integrate the constant action of $\lambda\mapsto \delta_\lambda$ on the Lie algebroid $\mathfrak{n}\times [0,\infty)$. This in turn globalizes to an action of $\R_+$ on the deformation groupoid $A_H\mathfrak{X}$.

We lift elements $X\in \mathfrak{n}$ to sections $\overline{X}=(X_t)_{t\geq0}$ of the Lie algebroid of $\mathfrak{A}_H\to [0,\infty)$ as follows. We identify $C^\infty(\mathfrak{A}_H)$ with $C^\infty(\mathfrak{n}\times [0,\infty))$ using the map $\psi_{\rm adb}$ from Equation \eqref{aadlan}. We write $X_t$ for the polynomial vector field on $\mathfrak{n}$ defined as the generator of the one-parameter group $s\mapsto g^{X,t}(s)$ given on $f\in C^\infty(\mathfrak{n})$ by 
$$[g^{X,t}(s)f](x):= f((sX)*_t x),$$
where $*_t$ denotes the polynomial Lie group operation on $\mathfrak{n}$ defined from $t[\cdot,\cdot]_{\mathfrak{n}}$ via the Baker-Campbell-Hausdorff formula. We note that $X_t$ depends polynomially on $t$ with $X_0$ being the directional derivative in the $X$-direction at $t=0$ and $X_1=X$ as a vector field on $\mathfrak{n}$. Moreover, if $X\in \mathfrak{n}$ is homogeneous of degree $k$ then $\overline{X}$ is homogeneous of degree $k$ on $\mathfrak{A}_H$.\\

{\bf Observation 4.} 
The graded Heisenberg symbol $\sigma_{H,{\rm gr}}(\nabla_{\pmb{E}})$ has a distinguished lift to a $G$-equivariant complex of $\mathcal{S}(A_H\mathfrak{X})$-modules defined from $\mathcal{S}(A_H\mathfrak{X},\mathrm{gr}(\wedge^\bullet T^*\mathfrak{X}\otimes \pmb{E})))$ with differentials $\mathfrak{d}_{k}=(\mathfrak{d}_{k,t})_{t\in [0,1]}$ being unbounded multipliers with $\mathfrak{d}_{k}|_{t=0}=\mathfrak{d}_{k,0}$ and $\mathfrak{d}_{k}|_{t=1}=\sigma_{H,{\rm gr}}(\nabla_{\pmb{E}})$.

The construction of this distinguished lift is best carried out in the framework of \cite[Section 4.4]{Dave_Haller1}. As is explained in the proof of \cite[Proposition 4.15]{Dave_Haller1}, the graded Heisenberg symbol $\sigma_{H,{\rm gr}}(\nabla_{\pmb{E}})$ is formed fibrewise in a point $x\in \mathfrak{X}$ from the Chevalley-Eilenberg complex of $\mathcal{S}(\mathsf{N})\otimes \pmb{E}_x$ with differential $d_{\mathfrak{n}}+\omega_x\wedge -$, where $d_{\mathfrak{n}}$ denotes the Chevalley-Eilenberg differential for the representation of $\mathfrak{n}$ on $\mathcal{S}(\mathsf{N})\otimes \pmb{E}_x$ and $\omega=\mathrm{gr}(\nabla)\in C^\infty(\mathfrak{X}; \mathrm{gr}(T^*\mathfrak{X})\otimes \End(\mathrm{gr}(\pmb{E}))))$. 

We can now define the lifted differentials $\mathfrak{d}=(\mathfrak{d}_{t})_{t\in [0,1]}$ fibrewise as 
$$\mathfrak{d}_{t,x}:=d_{\mathfrak{n}_t}+\omega_x\wedge -,$$ 
on the same spaces. Following the framework of \cite[Section 4.4]{Dave_Haller1}, we can conclude that $\mathfrak{d}|_{t=1}= \sigma_{H,{\rm gr}}(\nabla_{\pmb{E}})_x$ and $\mathfrak{d}_k|_{t=0}=\mathfrak{d}_{k,0}$. In other words, the lift $\sigma_{H,{\rm gr}}(\nabla_{\pmb{E}})$ to a $G$-equivariant complex of $\mathcal{S}(A_H\mathfrak{X})$-modules is carried out by deforming the positive-order terms via replacing a homogeneous invariant vector field $X$ by the family $(X_t)_{t\in [0,1]}$, while the zeroth-order terms are deformed according to their filtration degree. By Observation 3, the family $\mathfrak{d}_t$ is polynomial in $t$. The construction forms an adiabatic graded Heisenberg symbol complex 
\begin{align*}
\sigma_{A,{\rm gr}}(\nabla_{\pmb{E},\bullet}): \quad 0\to \mathcal{S}(A_H\mathfrak{X};\mathrm{gr}(\pmb{E}))\xrightarrow{\mathfrak{d}_{1}}\mathcal{S}(A_H\mathfrak{X};&\mathrm{gr}(T^*\mathfrak{X}\otimes\pmb{E}))\xrightarrow{\mathfrak{d}_{2}}\cdots \\
&\cdots\xrightarrow{\mathfrak{d}_{n}}\mathcal{S}(A_H\mathfrak{X};\mathrm{gr}(\wedge^nT^*\mathfrak{X}\otimes \pmb{E}))\to 0.
\end{align*}

\begin{lemma}
\label{lknlknadlem} 
The $G$-equivariant complex $\sigma_{A,{\rm gr}}(\nabla_{\pmb{E},\bullet})$ (of Observation 4) completes to a $G$-equivariant $C^*(A_H\mathfrak{X}|_{[0,1]})$-Fredholm sequence that in $t=0$ localizes to $\sigma_{\rm gr}(\nabla_{\pmb{E},\bullet})$ and in $t=1$ localizes to $\sigma_{H,{\rm gr}}(\nabla_{\pmb{E},\bullet})$

\end{lemma}

\begin{proof}
For $j=0,\ldots, n$, we write $C^*(A_H\mathfrak{X}|_{[0,1]};\mathrm{gr}(\wedge^jT^*\mathfrak{X}\otimes \pmb{E}))$ for the right  $G$-equivariant $C^*(A_H\mathfrak{X}|_{[0,1]})$-module obtained from completing $\mathcal{S}(A_H\mathfrak{X}|_{[0,1]};\mathrm{gr}(\wedge^jT^*\mathfrak{X}\otimes \pmb{E}))$ in the right $C^*(A_H\mathfrak{X}|_{[0,1]})$-valued inner product obtained from the Hermitean structure on $\mathrm{gr}(\wedge^jT^*\mathfrak{X}\otimes \pmb{E})$. Proceeding as above in Observation 2 and as in Section \ref{subsec:heis}, we can for $j=0,\ldots n$ construct  $G$-equivariant, invertible order reducing operators $\lambda_{j}\in \mathcal{E}_r'(A_H\mathfrak{X};\mathrm{gr}(\wedge^jT^*\mathfrak{X}\otimes \pmb{E}))$ that are graded essentially homogeneous of degree $0$ for the dilation action on $A_H\mathfrak{X}$, and graded anisotropically $H$-elliptic. We can then form the $G$-equivariant $C^*(A_H\mathfrak{X})$-modules 
$$\mathfrak{E}_{j}:=\lambda_{j}C^*(A_H\mathfrak{X}|_{[0,1]};\mathrm{gr}(\wedge^jT^*\mathfrak{X}\otimes \pmb{E})),$$ 
for $j=0,1,\ldots, n$. Moreover, $\mathfrak{d}_\bullet$ extends by continuity to $\mathfrak{E}_\bullet$ and $\sigma_{A,{\rm gr}}(\nabla_{\pmb{E},\bullet})$ completes to a complex of $C^*(A_H\mathfrak{X}|_{[0,1]})$-modules
\begin{align*}
(\mathfrak{E}_\bullet,\mathfrak{d}_\bullet): \quad 0\to \mathfrak{E}_0\xrightarrow{\mathfrak{d}_{1}}\mathfrak{E}_1\xrightarrow{\mathfrak{d}_{2}}\cdots \xrightarrow{\mathfrak{d}_{n}}\mathfrak{E}_n\to 0.
\end{align*}

We next proceed with constructing $G$-equivariant splitting operators $\mathfrak{b}_\bullet:\mathfrak{E}_\bullet\to \mathfrak{E}_{\bullet-1}$ as in Definition \ref{fredseq}, making $(\mathfrak{E}_\bullet,\mathfrak{d}_\bullet,\mathfrak{b}_\bullet)$ into a $G$-equivariant $C^*(A_H\mathfrak{X}|_{[0,1]})$-Fredholm sequence localizing to the right objects in $t=0$ and $t=1$. The construction is in spirit close to those of Lemma \ref{splittinprop} and Proposition \ref{lknalkdna}. We write $\mathfrak{E}_{j,(0,1]}:=C_0(0,1]\mathfrak{E}_j$ and $\mathfrak{E}_{j,[0,\epsilon]}:=\mathfrak{E}_j/C_0(\epsilon,1]\mathfrak{E}_j$ for $\epsilon>0$. We note that the localization $\mathfrak{E}_j\otimes_{\mathrm{ev}_{t=0}} C^*(\mathfrak{t}_H\mathfrak{X})$ of $\mathfrak{E}_j$ at $t=0$ coincides with $\mathfrak{E}_{j,0}$ of Observation 2.

By taking the symbol of a Heisenberg splitting as in Proposition \ref{lknalkdna} and Remark \ref{lfiininfiafnain}, we arrive at $G$-equivariant adjointable operators $\mathfrak{b}_j':\mathfrak{E}_{j,(0,1]}\to \mathfrak{E}_{j-1,(0,1]}$. Since the splitting operators depend smoothly on $t\in (0,1]$ and are splitting operators at each $t$, we conclude that for any $\chi\in C_c(0,1]$ and $j=1,\ldots, n$ we have that 
$$\begin{cases}
\chi\mathfrak{b}'_{j-1}\mathfrak{b}'_{j}\in \mathbb{K}_{C^*(A_H\mathfrak{X}|_{[0,1]})}(\mathfrak{E}_{j},\mathfrak{E}_{j-2}),\\
\chi(\mathfrak{b}'_j\mathfrak{d}_j+\mathfrak{d}_{j-1}\mathfrak{b}'_{j-1}-1)\in \mathbb{K}_{C^*(A_H\mathfrak{X}|_{[0,1]})}(\mathfrak{E}_{j-1}),
\end{cases}$$

We take splittings $\mathfrak{b}_{j,0}:\mathfrak{E}_{j,0}\to \mathfrak{E}_{j-1,0}$ at $t=0$ to $\sigma_{\rm gr}(\nabla_{\pmb{E}})$ as in Observation 2 above. By an openness argument, we can for some $\epsilon>0$ find a preimage $\mathfrak{b}_1'':\mathfrak{E}_{1,[0,\epsilon]}\to \mathfrak{E}_{0,[0,\epsilon]}$ to $\mathfrak{b}_{1,0}$ under the localization map 
$$\End_{C^*(A_H\mathfrak{X}|_{[0,\epsilon]})}(\mathfrak{E}_{1,[0,\epsilon]},\mathfrak{E}_{0,[0,\epsilon]})\to \End_{C^*(\mathfrak{t}_H\mathfrak{X})}(\mathfrak{E}_{1,0}, \mathfrak{E}_{0,0}),$$
such that 
$$\mathfrak{b}_1''\mathfrak{d}_1-1\in \mathbb{K}_{C^*(A_H\mathfrak{X}|_{[0,\epsilon]})}(\mathfrak{E}_0).$$
We now define the splitting operator $\mathfrak{b}_1:\mathfrak{E}_1\to \mathfrak{E}_{0}$ by taking a cutoff function $\chi\in C_c(0,1]$ with $\chi(t)=1$ for $t$ near $1$ and setting 
$$\mathfrak{b}_1:=(1-\chi)\mathfrak{b}_1''+\chi\mathfrak{b}_1'.$$
From construction we obtain that $\mathfrak{b}_1\mathfrak{d}_1-1\in \mathbb{K}_{C^*(A_H\mathfrak{X}|_{[0,1]})}(\mathfrak{E}_1)$. By averaging over $G$, we can assume that $\mathfrak{b}_1$ is $G$-equivariant.

We continue by induction to construct a $G$-equivariant splitting $\mathfrak{b}_{j+1}:\mathfrak{E}_{j+1}\to \mathfrak{E}_{j}$ from $G$-equivariant splittings $\mathfrak{b}_{k}:\mathfrak{E}_{k}\to \mathfrak{E}_{k-1}$, for $k=1,\ldots, j$. Again, an openness argument combined with the algebraic manipulations of Lemma \ref{splittinprop}, ensures that for some $\epsilon_{j+1}>0$ we can find a preimage $\mathfrak{b}_{j+1}'':\mathfrak{E}_{j+1,[0,\epsilon_{j+1}]}\to \mathfrak{E}_{j,[0,\epsilon_{j+1}]}$ to $\mathfrak{b}_{j+1,0}$ under the localization map at $t=0$, such that 
$$\begin{cases}
\mathfrak{b}_{j}\mathfrak{b}_{j+1}\in \mathbb{K}_{C^*(A_H\mathfrak{X}|_{[0,\epsilon_{j+1}]})}(\mathfrak{E}_{j+1},\mathfrak{E}_{j-1}),\\
\mathfrak{b}_{j+1}''\mathfrak{d}_{j+1}+\mathfrak{d}_j\mathfrak{b}_j-1\in \mathbb{K}_{C^*(A_H\mathfrak{X}|_{[0,\epsilon_{j+1}]})}(\mathfrak{E}_j).
\end{cases}$$
As above, we define the splitting operator $\mathfrak{b}_{j+1}:\mathfrak{E}_{j+1}\to \mathfrak{E}_{j}$ by setting 
$$\mathfrak{b}_{j+1}:=(1-\chi)\mathfrak{b}_{j+1}''+\chi\mathfrak{b}_{j+1}'.$$
By averaging over $G$, we can assume that $\mathfrak{b}_{j+1}$ is $G$-equivariant. By construction, we have that 
$$\begin{cases}
\mathfrak{b}_{j}\mathfrak{b}_{j+1}\in \mathbb{K}_{C^*(A_H\mathfrak{X}|_{[0,1]})}(\mathfrak{E}_{j+1},\mathfrak{E}_{j-1}),\\
\mathfrak{b}_{j+1}\mathfrak{d}_{j+1}+\mathfrak{d}_{j}\mathfrak{b}_{j}-1\in \mathbb{K}_{C^*(A_H\mathfrak{X}|_{[0,1]})}(\mathfrak{E}_{j}),
\end{cases},$$
completing the induction step. 

We conclude that the collection $\mathfrak{b}_\bullet:\mathfrak{E}_\bullet\to \mathfrak{E}_{\bullet-1}$ constructed above makes $(\mathfrak{E}_\bullet,\mathfrak{d}_\bullet,\mathfrak{b}_\bullet)$ into a $G$-equivariant $C^*(A_H\mathfrak{X}|_{[0,1]})$-Fredholm sequence. By construction, it localizes in $t=0$ to $\sigma_{\rm gr}(\nabla_{\pmb{E},\bullet})$ and by construction it localizes in $t=1$ to $\sigma_{H,{\rm gr}}(\nabla_{\pmb{E},\bullet})$.
\end{proof}

\begin{proof}[Proof of Theorem \ref{lknlknad}]
Since $\sigma_{A,{\rm gr}}(\nabla_{\pmb{E},\bullet})$ completes to a $G$-equivariant Fredholm complex over $C^*(A_H\mathfrak{X}|_{\mathfrak{X}\times [0,1]})$ by Lemma \ref{lknlknadlem} it defines a class 
$$[\sigma_{A,{\rm gr}}(\nabla_{\pmb{E}})]\in K_0^G(C^*(A_H\mathfrak{X}|_{[0,1]})).$$ 
By the construction in Lemma \ref{lknlknadlem} 
\begin{align*}
\mathrm{ev}_{t=0}[\sigma_{A,{\rm gr}}(\nabla_{\pmb{E}})]&=[\sigma(\nabla_{\pmb{E}})]\in K_0^G(C_0(T^*\mathfrak{X}))\quad\mbox{and}\\
 \mathrm{ev}_{t=1}[\sigma_{A,{\rm gr}}(\nabla_{\pmb{E}})]&=[\sigma_{H,{\rm gr}}(\nabla_{\pmb{E}})]=[\sigma_{H,{\rm gr}}(\mathsf{BGG}_\bullet(\nabla_{\pmb{E}})))]\in K_0^G(C^*(T_H\mathfrak{X})),
 \end{align*}
where we in the first equality used \eqref{gradedad} and in the last equality used Lemma \ref{comaodmoadm}. The defining property of $\psi$ is that 
$$K_0(\mathrm{ev}_{t=0})\otimes_{C^*(T_H\mathfrak{X})}\psi=K_0(\mathrm{ev}_{t=1}),$$
and the theorem follows.
\end{proof}

We will now start comparing $K$-theory classes to $K$-homology classes. Write $\mathrm{Euler}(\mathfrak{X})\in K_0^G(\mathfrak{X})$ for the class of the Euler-de Rham operator, i.e. the Dirac operator $\rd+\rd^*$ on $C^\infty(\mathfrak{X},\wedge^*T^*\mathfrak{X})$ graded by form degree modulo $2$. If $\mathfrak{X}$ is a complex manifold with a holomorphic $G$-action, we write $[\mathfrak{X}]\in K_0^G(\mathfrak{X})$ for the fundamental class of the spin$^c$-structure. The latter is defined from the Dolbeault-Dirac operator $\bar{\partial}+\bar{\partial}^*$ on $C^\infty(\mathfrak{X},\wedge^{0,*}T^*\mathfrak{X})$ graded by form degree modulo $2$, where $\wedge^{0,q}T^*\mathfrak{X}$ denotes the bundle of $(0,q)$-forms. The following result is well known. 

\begin{prop}
\label{mlngljfnljend}
Let $\mathfrak{X}$, $\pmb{E}$ and $\nabla_{\pmb{E}}$ be as in Assumption \ref{ljknalknad}. Then 
$$\mathsf{PD}^{\rm an}[\sigma(\nabla_{\pmb{E}})]=\mathrm{Euler}(\mathfrak{X})\cap [\pmb{E}].$$
and if $\mathfrak{X}$ is complex, then
$$\mathsf{PD}^{\rm an}[\sigma(\bar{\partial}_{\pmb{E},\bullet})]=[\mathfrak{X}]\cap [\pmb{E}].$$
\end{prop}

The proposition follows from for instance Proposition \ref{bnkbhb} applied to the trivial Carnot structure, or if one prefers it also follows from Kasparov's index theorem  \cite[Theorem 5]{kaspicm} (see also \cite[Section 12.3]{cumero}). The proof then reduces to standard computations, as in \cite[Section 7]{asI} for instance, equating the class of the symbol complex \eqref{lnlknad} in $K^0_G(T^*\mathfrak{X})$ with the symbol class of the Hodge-Dirac operator twisted by $\pmb{E}$, and similarly, equating the class of the symbol complex \eqref{lnlknadcomplex} in $K^0_G(T^*\mathfrak{X})$ with the symbol class of the Dolbeault-Dirac operator twisted by $\pmb{E}$.

\begin{thm}
\label{ljnadoajdn}
Let $\mathfrak{X}$ be a compact parabolic geometry with an action of a compact group $G\to \Aut(\mathfrak{X},\omega)$ and $\nabla_{\pmb{E}}$ be a connection satisfying Assumption \ref{ljknalknad}. Then in the real case, we have the equality
$$[\mathsf{BGG}_\bullet(\nabla_{\pmb{E}})]=\mathrm{Euler}(\mathfrak{X})\cap [\pmb{E}].$$
and if $\mathfrak{X}$ is complex, the following identity holds
$$[\mathsf{BGG}_\bullet(\nabla_{\pmb{E}})]=[\mathfrak{X}]\cap [\pmb{E}].$$
\end{thm}

\begin{proof}
We consider the real case and compute that
\begin{align*}
\mathrm{Euler}(\mathfrak{X})\cap [\pmb{E}]=\mathsf{PD}^{\rm an}[\sigma(\nabla_{\pmb{E}})]=&\mathsf{PD}^{\rm an}_H\circ \psi[\sigma(\nabla_{\pmb{E}})]=\\
=&\mathsf{PD}_H^{\rm an}[\sigma_{H,{\rm gr}}(\mathsf{BGG}_\bullet(\nabla_{\pmb{E}}))]=[\mathsf{BGG}_\bullet(\nabla_{\pmb{E}})].
\end{align*}
where the first equality uses Proposition \ref{mlngljfnljend}, the second equality uses Theorem \ref{aldknadjn}, the third equality uses Theorem \ref{lknlknad}, and the final equality uses Proposition \ref{bnkbhb}.  The proof in the complex is verbatim the same to the real case. 
\end{proof}

\begin{remark}
We can apply Theorem \ref{ljnadoajdn} to the $K$-equivariant geometry of the full flag manifold $\mathsf{G}/\mathsf{B}$ of a connected, complex, semisimple Lie group $\mathsf{G}$ (with maximal compact subgroup $\mathsf{K}$ and Borel subgroup $\mathsf{B}$). In this case, for any finite-dimensional $\mathsf{G}$-representation $W$ we have that 
$$[\mathsf{BGG}_\bullet(W)]=[\mathsf{G}/\mathsf{B}]\cap [\mathsf{G}\times_\mathsf{B} W]\in K_0^\mathsf{K}(\mathsf{G}/\mathsf{B}).$$
\end{remark}

\end{document}